\documentclass[12pt,a4paper]{amsart}
\usepackage{amssymb}
\usepackage{amsfonts}
\usepackage{amsmath}
\usepackage[mathscr]{eucal}

\usepackage{color}

\newtheorem{thm}{Theorem}
\newtheorem{prop}[thm]{Proposition}
\newtheorem{lem}[thm]{Lemma}
\newtheorem{cor}[thm]{Corollary}

\numberwithin{equation}{section}
\newtheorem{defn}[thm]{Definition}

\def\N{{\Bbb N}}
\def\Z{{\Bbb Z}}
\def\Q{{\Bbb Q}}
\def\R{{\Bbb R}}
\def\C{{\Bbb C}}
\def\A{{\Bbb A}}

\def\bS{{\Bbb S}}

\def\emp{\varnothing}

\def\fa{{\frak a}}
\def\fb{{\frak b}}

\def\ff{{\frak f}}

\def\fm{{\frak m}}
\def\fn{{\frak n}}

\def\fp{{\frak p}}
\def\fq{{\frak q}}

\def\fI{{\frak I}}

\def\fc{{\frak c}}

\def\fX{{\frak X}}

\def\cO{\frak o}

\def\cB{{\mathcal B}}

\def\cN{{\mathcal N}}

\def\cG{{\mathcal G}}

\def\GL{{\operatorname {GL}}}

\def\PGL{{\operatorname{PGL}}}

\def\Re{{\operatorname {Re}}}
\def\Im{{\operatorname {Im}}}
\def\tr{{\operatorname{tr}}}
\def\nr{{\operatorname{N}}}

\def\diag{{\operatorname {diag}}}

\def\Ad{{\operatorname{Ad}}} 
\def\vol{{\operatorname{vol}}}

\def\leq{\leqslant}
\def\geq{\geqslant}
\def\bsl{\backslash}
\def\le{\leq}
\def\ge{\geq}

\def\d {{{d}}}

\def\JJ{{\Bbb J}}
\def\bs{{\bold s}}
\def\bK{{\bold K}}

\def\1{{\bold 1}}
\def\ccA{{\mathscr A}}

\def\tF{f}
\def\bx{{\mathbf x}}

\renewcommand{\a}{\alpha}
\renewcommand{\b}{\beta}

\newcommand{\e}{\epsilon}

\renewcommand{\l}{\lambda}

\newcommand{\s}{\sigma}

\newcommand{\ga}{{\mathfrak{a}}}
\newcommand{\gb}{{\mathfrak{b}}}

\newcommand{\gf}{{\mathfrak{f}}}

\newcommand{\gn}{{\mathfrak{n}}}
\newcommand{\go}{{\mathfrak{o}}}
\newcommand{\gp}{{\mathfrak{p}}}
\newcommand{\gq}{{\mathfrak{q}}}

\newcommand{\Acal}{{\mathcal A}}
\newcommand{\Bcal}{{\mathcal B}}

\newcommand{\Gcal}{{\mathcal G}}

\newcommand{\Ical}{{\mathcal I}}
\newcommand{\Jcal}{{\mathcal J}}

\newcommand{\Ocal}{{\mathcal O}}

\renewcommand{\AA}{\mathbb{A}}

\newcommand{\CC}{\mathbb{C}}

\newcommand{\LL}{\mathbb{L}}
\newcommand{\NN}{\mathbb{N}}

\newcommand{\QQ}{\mathbb{Q}}
\newcommand{\RR}{\mathbb{R}}

\newcommand{\WW}{\mathbb{W}}
\newcommand{\ZZ}{\mathbb{Z}}

\newcommand{\bfc}{{\mathbf c}}

\newcommand{\bfs}{{\mathbf s}}

\newcommand{\bfK}{{\mathbf K}}

\newcommand{\ord}{\operatorname{ord}}

\newcommand{\fin}{{\rm fin}}
\renewcommand{\Re}{\operatorname{Re}}

\newcommand{\Res}{\operatorname{Res}}

\def\AL{{\rm{AL}}}
\def\ADL{{\rm{ADL}}}

\title{Existence of Hilbert cusp forms with non-vanishing $L$-values}

\author{Shingo Sugiyama}
\author{Masao Tsuzuki}

\begin{document}

\maketitle

\begin{abstract}
We give a derivative version of the relative trace formula on $\PGL(2)$ studied in our previous work,
and obtain a formula of an average of central values (derivatives)
of automorphic $L$-functions for Hilbert cusp forms.
As an application, we prove existence of Hilbert cusp forms
with non-vanishing central values (derivatives)
such that the absolute degrees of their Hecke fields are sufficiently large.
\end{abstract}

\setcounter{tocdepth}{1}

\section{Introduction}
This is a continuation of our previous paper \cite{SugiyamaTsuzuki}; we 
freely use the notation introduced there, which is collected at the end 
of this section for a convenience of reference. 

Let $F$ be a totally
real number field of degree $d_F$ and $\cO$ the integer ring of $F$. 
Let $\fn$ be an $\cO$-ideal and $l=(l_{v})_{v\in \Sigma_\infty} \in (2
\N)^{d_F}$ an even weight. For $\pi \in \Pi_{\rm{cus}}^*(l,\fn)$ and an 
idele class character $\eta$ of $F^\times$ such that $\eta^2=\1$, the 
standard $L$-function $L(s,\pi\otimes \eta)$ of $\pi\otimes \eta$ is an 
entire function on $\C$ satisfying the functional equation 
$$
L(s,\pi \otimes \eta)=\epsilon(s,\pi\otimes\eta)\,L(1-s,\pi\otimes \eta),
$$
with $\epsilon(s,\pi\otimes \eta)$ being the $\epsilon$-factor; it is of 
the form $\epsilon(s,\pi \otimes \eta)=\pm(\nr(\fn\ff_\eta^2)\,D_F^2)^{1/2-s}$. The number $\epsilon(1/2,\pi\otimes \eta)\in\{+1,-1\}$, is 
called the sign of the functional equation. The central value $L(1/2,\pi)L(1/2,\pi\otimes \eta)$ and the derivative $L(1/2,\pi)L'(1/2,\pi\otimes\eta)$ has an important arithmetic meaning; there are many works which exploit the nature of these $L$-values in connection with the arithmetic algebraic geometry of modular varieties (\cite{Waldspurger}, \cite{GrossZagier}, \cite{YZZ}, \cite{SWZhan}, \cite{SWZhan2}, \cite{SWZhan3}).

\subsection{}
Let $\fa\subset \cO$ be an ideal relatively prime to $\ff_\eta \fn$ and set $S=S(\fa)$. We write the Satake parameter of $\pi\in \Pi_{\rm{cus}}^*(l,\fn)$ at $v\in S$ as $\diag(q_v^{\nu_{v}(\pi)/2},q_v^{-\nu_{v}(\pi)/2})$ with
$\pm \nu_{v}(\pi)$ belonging to the space $\fX_v=\C/4\pi i (\log q_v)^{-1}\Z$.
In \cite{SugiyamaTsuzuki}, given an even holomorphic function $\alpha(\bs)$ on ${\frak X}_S=\prod_{v\in S}\fX_v$, we studied the asymptotic of the average 
\begin{align}
\AL^*(\fn,\alpha)=\frac{C_l}{\nr(\fn)}\sum_{\pi \in \Pi_{\rm{cus}}^*(l,\fn)}
\frac{L(1/2,\pi)\,L(1/2,\pi \otimes \eta)}{L^{S_{\pi}}(1,\pi;{\rm{Ad}})}\,\alpha(\nu_S(\pi))
 \label{AL*}
\end{align}
with $\nu_S(\pi)=\{\nu_v(\pi)\}_{v\in S}$ and 
\begin{align}
C_l=\prod_{v\in \Sigma_\infty} \frac{2\pi\,(l_v-2)!}{\{(l_v/2-1)!\}^2}
 \label{fDl}
\end{align}
as the norm $\nr(\fn)$ grows under the following conditions.
\begin{itemize}
\item[(a)] The number $(-1)^{\epsilon(\eta)}\,\tilde \eta(\fn)$, the 
common value of $\epsilon(s,\pi)\epsilon(s,\pi\otimes \eta)|_{s=1/2}$ 
for all $\pi \in \Pi_{\rm{cus}}^{*}(l,\fn)$, equals $1$.
\item[(b)] $\eta_{v}(\varpi_{v})=-1$ for any $v\in S(\gn)$.
\end{itemize}
In this paper, imposing the same condition (b) as above but the 
different sign condition $(-1)^{\e(\eta)}\tilde{\eta}(\fn)=-1$  than (a), we 
investigate the asymptotic behavior of the following average involving 
the central derivative of $L$-function $L(s,\pi\otimes \eta)$. 
\begin{align}
\ADL_{-}^*(\fn,\alpha)=\frac{C_l}{\nr(\fn)}\sum_{\substack{\pi \in \Pi_{\rm{cus}}^*(l,\fn) \\ \epsilon(1/2,\pi\otimes \eta)=-1}}
\frac{L(1/2,\pi)\,L'(1/2,\pi \otimes \eta)}{L^{S_{\pi}}(1,\pi;{\rm{Ad}})}\,\alpha(\nu_S(\pi)).
 \label{Intro0}
\end{align}
To state our main result precisely, we need notation. Let $\Ical_{S,\eta}$ be the monoid of ideals
$\fn\subset \cO$ generated by prime ideals $\fp$ prime to $S\cup S(\ff_\eta)$ such that $\tilde \eta(\fp)=-1$, and 
$$
\Ical_{S,\eta}^{\pm}=\{\fn\in \Ical_{S,\eta} \ | \ (-1)^{\e(\eta)}\tilde\eta(\fn)=\pm 1\}.
$$
For $n\in \N$, let $X_n(x)$ be the Tchebyshev polynomial $X_n(x)$ 
defined by the relation 
\begin{align}
X_n(x)={\sin((n+1)\theta)}/{\sin \theta} \quad \text{for $x=2\cos \theta $}
 \label{Tcheby}
\end{align}
and set
\begin{align}
\alpha_{\fa}(\nu)=\prod_{v\in S}X_{n_v}(q_v^{\nu_v/2}+q_v^{-\nu_v/2}), \quad \nu=\{\nu_v\}_{v\in S}\in \fX_{S}
\label{HeckefunctAA}
\end{align}
in terms of the prime ideal decomposition $\fa=\prod_{v\in S}\fp_v^{n_v}$. For such $\fa$, define $\fa_\eta^{\pm}=\prod_{\substack{v\in S(\fa) \\ \tilde\eta(\fp_v)=\pm 1}} \fp_v^{n_v}$, $d_1(\fa)=\prod_{v\in S(\fa)}(n_v+1)$ and $\delta_{\square}(\fa)=\prod_{v\in S(\fa)}\delta(n_v\in 2\N)$. We have an asymptotic formula of $\ADL_{-}^{*}(\fn;\a_{\fa})$ with an error term whose dependence on $\gn\in \Ical_{S,\eta}^{-}$ 
and $\fa$ is made explicated. We also have a similar formula for $\AL^*(\fn;\a_{\fa})$ with $\gn \in \Ical_{S,\eta}^{+}$.

\begin{thm} \label{MAIN-THM1} 
 Suppose ${\underline l}=\min_{v\in \Sigma_\infty}l_v \geq 6$. Set $c=d_F^{-1}({\underline l}/2-1)$.
For an integral ideal $\gn$, set
\begin{align*}
\nu(\fn)&=\{\prod_{v\in S(\gn) - (S_{1}(\fn)\cup S_{2}(\fn))}(1-q_v^{-2})\}\,\{\prod_{v\in S_
2(\fn)}(1-(q_{v}^{2}-q_{v})^{-1})\}.
\end{align*}
For any sufficiently small number $\e>0$, we have
\begin{align}
\AL^*(\fn;\a_\fa)&=4D_F^{3/2}\,L_\fin(1,\eta)\, \nu(\fn)\,\nr(\fa)^{-1/2}\delta_\square(\fa_\eta^{-})d_1(\fa_\eta^{+}) 
 \label{MAIN-THM1-1}
\\
&\quad +{\mathcal O}_{\epsilon,l,\eta}\left(\nr(\fa)^{c+2+\epsilon}\nr
(\fn)^{-\inf (c,1)+\epsilon}\right), \qquad \gn \in \Ical_{S,\eta}
^{+},
 \notag
\end{align}
\begin{align}
& \ADL^*_{-}(\fn;\a_\fa) \label{MAIN-THEM1-2} \\
=&
\,4D_F^{3/2}\,L_\fin(1,\eta)\, \nu(\fn)\,\nr(\fa)^{-1/2}d_1(\fa_\eta^{+})
\biggl\{\delta_{\square}(\fa_\eta^{-})\,\biggl(\log(\sqrt{\nr(\fn)\nr(\fa)^{-1}}\nr(\ff_\eta)D_F)
\notag \\
&+\sum_{v\in S(\gn)-(S_{1}(\gn)\cup S_2(\fn))}\frac{\log q_v}{q_v^2-1} 
+\sum_{v\in S_2(\fn)}\frac{\log q_v}{q_v^2-q_v-1}
+\frac{L'}{L}(1,\eta)+{\frak C}(l)\biggr)
 \notag
\\
&+\sum_{v\in S(\fa_\eta^{-})} \delta_{\square}(\fa_\eta^{-}\fp_v^{-1})\log (q_v^{n_{v}+\frac{1}{2}})
 \notag
\\
&+\Ocal_{\e,l,\eta}\left(\nr(\fa)^{-1/2}d_1(\fa_\eta^{+})\delta_{\square}(\fa_\eta^{-})\,X(\fn)+\nr(\fa)^{c+2+\e}\nr(\fn)^{-\inf(1,c)+\e}\right), \quad \fn\in \Ical_{S,\eta}^{-}, 
 \notag
\end{align}
where 
\begin{align*}
{\frak C}(l)&=\sum_{v\in \Sigma_\infty} \biggl(\sum_{k=1}^{l_{v}/2-1}\frac{1}{k}- \frac{1}{2}\log \pi - \frac{1}{2}C_{\rm Euler}
-\delta(\eta_v(-1)=-1)\,\log 2 \biggr), \\  
X(\fn)&=\sum_{u \in S(\gn)}\frac{\log q_{u}}{q_{u}} + \sum_{u \in S(\gn)}\frac{\log q_{u}}{(q_{u}-1)^{2}}.
\end{align*}
The constants implicit in the $O$-symbols in both formulas are 
independent of $\fn$ and $\fa$.
\end{thm}

\subsection{} 
For a positive integer $N$, let $J_0^{\rm{new}}(N)$ be the new part of 
the Jacobian variety of the modular curve $X_0(N)$ of level $N$. J.-P. 
Serre showed that the largest dimension of $\Q$-simple factors of $J_0^{\rm{new}}(N)$ tends to infinity as $N$ grows (\cite[Theorem 7]{Serre97})
. This result was refined in several ways by E. Royer (\cite{Royer}); he 
obtained a quantitative version of Serre's theorem giving a lower 
bound of the largest dimension of $\Q$-simple factors $A$ of $J_0^{\rm{new}}(N)$ with or without rank conditions for the Model-Weil group of 
$A$. By the correspondence between the $\Q$-simple factors $A$ of $J_0^{\rm{new}}(N)$ and the normalized Hecke eigen newforms $f$ of level $\Gamma_0(N)$ 
and weight $2$, and by invoking the progress toward the Birch and 
Swinnerton-Dyer conjecture, the lower bound for the largest $\dim A$ is
obtained from a lower bound of the maximum value of the absolute degree
of the Hecke field $\Q(f)$ with or without conditions on the order of 
$L$-series $L(s,f)$ at the center of symmetry. Thus, one of Royer's 
results in \cite{Royer} can be stated in the language of modular forms 
as follows. 

\begin{thm} (Royer \cite{Royer}) Let $p$ be a prime. There exist 
constants $C_p>0$ and $N_p>0$ with the following properties:  
\begin{itemize}
\item[(1)] For any $N>N_p$ relatively prime to $p$, there exists a normalized Hecke 
eigen newform $f$ of level $\Gamma_0(N)$ and weight $2$ satisfying the 
conditions:
\begin{itemize}
\item[(i)] $L(1/2,f)\not=0$, where the functional equation of $L(s,f)$ 
relates the values at $s$ and $1-s$.  
\item[(ii)] $[\Q(f):\Q] \geq C_p\,\sqrt{\log\log N}$.
\end{itemize}
\item[(2)] For any $N>N_p$ relatively prime to $p$, there exists a normalized Hecke 
eigen newform $f_1$ of level $\Gamma_0(N)$ and weight $2$ satisfying the 
conditions: 
\begin{itemize}
\item[(i)] The sign of the functional equation of $L(s,f_1)$ is $-1$.
\item[(ii)] $L'(1/2,f_1)\not=0$.  
\item[(iii)] $[\Q(f_1):\Q] \geq C_p\,\sqrt{\log\log N}$. 
\end{itemize}
\end{itemize}
\end{thm}
We obtain an analogue of this theorem for higher weight Hilbert modular
cuspforms by using Theorem~\ref{MAIN-THM1}. For a cupsidal 
representation $\pi \in \Pi_{\rm{cus}}^{*}(l,\gn)$, we denote by $\Q(\pi)$ the field of rationality of $\pi$ (for definition, see 7.1.)

\begin{thm} \label{MAIN-THM2} 
Let $l=(l_v)_{v\in\Sigma_\infty}$ be a weight such that $l_v=k$ for all $v\in \Sigma_\infty$ with an even integer $k\geq 6$, and $\eta$ a quadratic idele class character of 
$F^\times$. Let $S$ be a finite subset of $\Sigma_\fin-S(\ff_\eta)$ and ${\bf J}=\{J_v\}_{v\in S}$ a family of closed subintervals of $(-2,2)$. Given a prime ideal $\fq$ prime to $S\cup S(\ff_\eta)$, 
there exist constants $C_{\fq}>0$ and
$N_{\gq, S, l, \eta, {\bf J}}>0$
with the following properties:
For any ideal $\fn\in \Ical_{S\cup S(\fq),\eta}^{+}$ with $\nr(\fn)>N_{\fq, S, l, \eta, {\bf J}}$,
there exists $\pi \in \Pi_{\rm{cus}}^*(l,\fn)$ such that
\begin{itemize}
\item[(i)] $L(1/2,\pi)\not=0$ and $L(1/2,\pi \otimes \eta)\not=0$, 
\item[(ii)] $[\Q(\pi):\Q]\geq C_{\fq} \,\sqrt{\log\log \nr(\fn)}$, and  
\item[(iii)] $q_v^{\nu_v(\pi)/2}+q_v^{-\nu_v(\pi)/2}\in J_v$ for all $v\in S$.
\end{itemize}
\end{thm}
We should note that this can be regarded as a refinement of \cite[Corollary 1.2]{SugiyamaTsuzuki}.

As for derivatives, we have a conditional result. 
\begin{thm} \label{MAIN-THM2.5} 
Let $l=(l_v)_{v\in\Sigma_\infty}$ and $\eta$ be the same as in Theorem~\ref{MAIN-THM2}. Suppose that 
for any ideal $\gn$ prime to $\ff_\eta$,
\begin{align}
\tfrac{\d}{\d s}|_{s=1/2}(L(s,\pi)L(s,\pi\otimes \eta))\geq 0 \quad \text{for all $\pi \in \Pi_{\rm{cus}}^*(l,\gn)$.}
\label{DLnonneg}
\end{align}
Let $S$ be a finite subset of $\Sigma_\fin-S(\ff_\eta)$ and ${\bf J}=\{J_v\}_{v\in S}$ a family of closed subintervals of $(-2,2)$. Given a prime ideal $\fq$ prime to $S\cup S(\ff_\eta)$ 
and a constant $M>1$, there exist constants $C_{\fq}>0$ and
$N_{\fq, S, l, \eta, {\bf J}, M}>0$
with the following properties:
 For any ideal $\fn \in \Ical_{S\cup S(\fq),\eta}^{-}$ with
$\nr(\fn)>N_{\fq, S, l, \eta, {\bf J}, M}$
and $\sum_{v\in S(\fn)}\frac{\log q_v}{q_v}\leq M$, there exists $\pi \in \Pi_{\rm{cus}}^*(l,\fn)$ such that
\begin{itemize}
\item[(i)] $\epsilon(1/2,\pi \otimes \eta)=-1$,
\item[(ii)] $L(1/2,\pi)\not=0$ and $L'(1/2,\pi \otimes \eta)\not=0$,
\item[(iii)] $[\Q(\pi):\Q]\geq C_{\fq}\,\sqrt{\log\log \nr(\fn)}$, and 
\item[(iv)] $q_v^{\nu_v(\pi)/2}+q_v^{-\nu_v(\pi)/2}\in J_v$ for all $v\in S$.
\end{itemize}
\end{thm}
We should note that the assumption \eqref{DLnonneg} is a consequence of the Riemann hypothesis for the $L$-function $L(s,\pi)L(s,\pi\otimes \eta)$. 
Theorem~\ref{MAIN-THM2} (Theorem~\ref{MAIN-THM2.5}) yields  Hilbert cuspforms in arbitrarily large level with arbitrary large degree of the field of rationality, such that the central value of $L$-function and the central value (derivative) of its prescribed quadratic twist are nonzero simultaneously. Although we 
can expect a similar result for parallel weight $2$ Hilbert cuspforms, 
our method does not work as it is for such low weight cases. In order to 
treat these interesting cases, the technique of Green's function as in
\cite{Tsuzuki} and \cite{Sugiyama2} may be useful.

\subsection{}
Let us describe a brief line of argument how we prove Theorems~\ref{MAIN-THM1}, \ref{MAIN-THM2} and \ref{MAIN-THM2.5}, explaining the organization of this paper. 
In our previous work \cite{SugiyamaTsuzuki}, we constructed the 
renormalized smoothed automorphic Green function $\hat{\bf \Psi}_{\rm{reg}}^{l}(\gn|\a)$ as the value at $\l=0$ of an entire extension of some 
Poincar\'{e} series $\hat{\mathbf \Psi}_{\b,\l}^{l}(\gn|\a)$ originally 
defined for $\Re(\l)>1$. Then we computed the period integral of $\hat{\bf \Psi}_{\rm{reg}}^{l}(\gn|\a)$ along the diagonal split torus $H$
adelically in a very explicit form. In the present work, instead of the period
integral, we introduce a certain integral transform $\partial P_{\b,\l}^{\eta}(\varphi)$
(see \S 2) for any continuous function $\varphi$ on $\PGL(2, F)\bsl \PGL(2, \A)$ and a quadratic idele class character $\eta$ of $F^{\times}$,
depending on a complex parameter $\l$ and a test function $\b$ for 
renormalization, whose constant term at $\l=0$ yields the derivative at 
$s=1/2$ of the period integral of $\varphi|\det|_\A^{s-1/2}$ along $H$. 
The main step to have the second formula in Theorem~\ref{MAIN-THM1} is to calculate $\partial P_{\b,\l}^{\eta}(\hat{\mathbf \Psi}_{\rm{reg}}^{l}(\gn|\a))$ and its constant term at $\l=0$ in two different  ways; the process is completely parallel to that in \cite{SugiyamaTsuzuki} for period integrals. In \S 2, after recalling the 
explicit formula of Hecke's zeta integrals for old forms (\cite{Sugiyama1}) and calculating their derivatives,
we prove a formula of ${\rm{CT}}_{\l=0}\,\partial P_{\b,\l}^{\eta}(\hat{\mathbf \Psi}_{\rm{reg}}^{l}(\gn|\a))$ 
written in terms of the spectral data of cuspidal representations in $\Pi_{\rm{cus}}^{*}(l,\gn)$ (Proposition~\ref{cuspidal part 1}). In \S 3, 
closely following \cite{SugiyamaTsuzuki}, we compute $\partial P_{\b,\l}^{\eta}(\hat{\mathbf \Psi}_{\rm{reg}}^{l}(\gn|\a))$ according to the $H(F)\times H(F)$-
double coset decomposition of $\GL(2, F)$. By equating the two expressions of ${\rm{CT}}_{\l=0}\,\partial P_{\b,\l}^{\eta}(\hat{\mathbf \Psi}_{\rm{reg}}^{l}(\gn|\a))$ 
obtained in \S 2 and \S 3, we get a kind of relative trace formula, 
which is stated in Theorem~\ref{DRTF}. The formula is not for our $\ADL_{-}^{*}(\fn)$ but for a similar average of $L$-values over all cuspidal 
representations $\pi \in \Pi_{\rm{cus}}(l,\gn)$. We need to sieve out 
information on an average of only those $\pi\in \Pi_{\rm{cus}}(l,\fn)$ 
with exact conductor $\fn$. For that purpose, we introduce a certain operation (see Definition~\ref{N-transformDef}), which we call the $\cN$-transform, for any arithmetic function defined on a set of ideals. The first subsection 
of \S 4 is devoted to the study of the $\cN$-transform. By applying the 
$\cN$-transform of each term occurring in the formula \eqref{DRTF-1}, we 
deduce yet another formula \eqref{henkeiDRTF-1}, which relates the 
average $\ADL^*(\fn)$ to the sum of following terms: (i) the
$\cN$-transforms of $\tilde \WW_{\rm{u}}^\eta(l,\gn|\a)$ and $\WW_{\rm{hyp}}^\eta(l,\gn|\a)$, both of them occurring in the geometric side of
\eqref{DRTF-1}, (ii) the $L$-value average $\AL^*(\fn)$ and (iii) the $\cN$-transform of a certain term $\AL^{\partial w}(\fn)$ arising from the spectral side of \eqref{DRTF-1}. In \S 6, we analyze these terms 
separately and obtain an exact evaluation of the $\cN$-transform of $\tilde \WW_{\rm{u}}^\eta(l,\gn|\a)$ and estimations of the remaining 
terms, which lead us to the proof of the second formula of Theorem~\ref{MAIN-THM1}. In the proceeding \S 5, by applying the relative trace 
formula \cite[Theorem 9.1]{SugiyamaTsuzuki} to the test function $\a_\fa$,
we deduce the first asymptotic formula in Theorem~\ref{MAIN-THM1}, 
which is necessary to prove
Theorem~\ref{MAIN-THM2}. 
In \S 7, we give the proof of Theorems~\ref{MAIN-THM2} and \ref{MAIN-THM2.5}. Actually, what we 
do there is to confirm that the argument of \cite{Royer} for the classical modular forms still works with a minor modification in our setting. The analysis performed in \S 6 relies 
on explicit formulas of local integrals arising from $\WW_{\rm{hyp}}^\eta(l,\gn|\a)$ and $\tilde\WW_{\rm{u}}^\eta(l,\gn|\a)$; the aim of \S 8 
is to provide them. In the final section \S 9, we study a certain 
lattice sum to use it in the error term estimate in \S 5 and \S 8.

\medskip
\noindent
{\bf Notation} : 
\begin{itemize}
\item Given a condition ${\rm P}$, $\delta({\rm P})$ is $1$ if ${\rm{P}}$ is true and $0$ otherwise. 
\item $\N$ denotes the set of positive integers and $\N_0=\N\cup \{0\}$. 
\item $F$ denotes a totally real number field. 
\begin{itemize}
\item $d_F=[F:\Q]$.
\item $\cO$ : the integer ring of $F$. 
\item $D_F$ : the absolute discriminant of $F$.
\end{itemize}
\item $\Sigma_\infty$ (resp. $\Sigma_\fin$) : the set of infinite places 
(resp. finite places) of $F$. We set $\Sigma_{F}= \Sigma_\infty \cup \Sigma_\fin$.
\item $F_v$ : the completion of $F$ at $v\in \Sigma_{F}$.
\item $|\,|_v$ : the normalized valuation of $F_v$ for $v\in \Sigma_{F}$.
\item For any $v\in \Sigma_\fin$, 
\begin{itemize} 
\item $\cO_v$ : the integer ring of $F_v$. 
\item $\varpi_v$ : a prime element of $\cO_v$. 
\item $\fp_v$ : the prime ideal of $\cO$ corresponding to $v$. 
\item $q_v$ : the cardinality of the residue field $\cO/\fp_v$ for $v$.
\end{itemize}
\item $\A$ : the adele ring of $F$.
\item For an ideal $\fm\subset \cO$, 
\begin{itemize}
\item $S(\fm)=\{v\in \Sigma_\fin|\,{\rm{ord}}_v(\fm)>0\}$. 
\item $S_k(\fm)=\{v\in S(\fm)|\,{\rm{ord}}_v(\fm)=k\,\}$ for $k\in \N$.  
\item $\nr(\fm)$ : the absolute norm of $\fm$. 
\end{itemize}
\item $\eta=\prod_{v}\eta_v$ always denotes a quadratic idele class character of $F^\times$.
\begin{itemize}
\item $\ff_\eta = \prod_{v \in \Sigma_{\fin}} \gp_{v}^{f(\eta_{v})}$ : the conductor of $\eta$.
\item $x_\eta^*$ :
the finite idele such that $x_{\eta, v}^{*} =\varpi_{v}^{-f(\eta_{v})}$
for all $v \in \Sigma_{\fin}$.
\item $x_\eta$ : the adele such that $x_{\eta,v}=0$ for all $v\in \Sigma_\infty$ and $x_{\eta,\fin}=x_\eta^*$.  
\item $\epsilon(\eta)$ : the number of $v\in \Sigma_\infty$ such that $\eta_v(-1)=-1$.
\item $\tilde \eta$ : the character of the ideal group relatively prime 
to $\ff_\eta$ defined by $\tilde \eta(\fp_v)=\eta_v(\varpi_v)$, $v\in \Sigma_\fin-S(\ff_\eta)$. 
\end{itemize}
\item $\fn$ always denotes an $\cO$-ideal prime to $\ff_\eta$. 
\item $H=\{\left[\begin{smallmatrix}t_1 & 0 \\ 0 & t_2 \end{smallmatrix} \right]|\,t_1,\,t_2\in \GL(1)\,\}$.
\item $\bK_\fin=\prod_{v\in \Sigma_\fin}\bK_v$ with $\bK_v=\GL(2, \cO_v)$ 
for $v\in \Sigma_\fin$.
\item $\bK_0(\fn)=\prod_{v\in \Sigma_\fin}\bK_0(\fn\cO_v)$ with
$\bK_0(\fn\cO_v)=\{\left[\begin{smallmatrix} a & b \\ c & d \end{smallmatrix}\right]\in \bK_v|\,c\in \fn\cO_v\,\}$ for $v\in \Sigma_\fin$.
\item $\pi$ : an irreducible cuspidal representation of $\GL(2, \A)$ with 
trivial central character.
\begin{itemize}
\item $\ff_\pi$ : the conductor of $\pi$.
\item $\{\pi_v\}_{v\in \Sigma_F}$: a family of irreducible admissible 
representations of $\GL(2, F_{v})$ such that $\pi \cong \otimes_{v} \pi_v$. 
\item $S_\pi=\{v\in \Sigma_\fin|\,\ord_v(\ff_\pi)\geq 2\,\}$. 
\item For a place $v\in \Sigma_\fin-S(\ff_\pi)$, the Satake parameter of $\pi_v$ is denoted by $A_v(\pi)=\diag(q_v^{\nu_v(\pi)/2}, q_v^{-\nu_v(\pi)/2})$ with $\nu_v(\pi) \in \C/4\pi i (\log q_v)^{-1}\Z$. We set $\lambda_v(\pi)=\tr A_v(\pi)$. 
\item $L^{S_\pi}(s,\pi;\Ad)$: the adjoint square $L$-function of $\pi$, whose local factors above $S_\pi$ are removed. 
\end{itemize}
\item $l=(l_v)_{v\in \Sigma_\infty}$ : an even weight, i.e., a system of 
positive even integers indexed by $\Sigma_\infty$. We set ${\underline l}=\inf_{v\in \Sigma_\infty} l_v$. 

\item $\Pi_{\rm{cus}}(l,\fn)$ : the set of all those $\pi\cong \otimes_v \pi_v$ such that $\pi_v$ is a discrete series representation 
of $\PGL(2, F_{v})$ of weight $l_v$ for any $v\in \Sigma_{\infty}$ and $\fn\subset \ff_\pi$.  
\item $\Pi_{\rm{cus}}^*(l,\fn)=\{\pi \in \Pi_{\rm{cus}}(l,\fn)|\,\ff_\pi =\fn\,\}.$ 
\item $S$ : a finite subset of $\Sigma_\fin-S(\ff_\eta\fn)$.
\begin{itemize}
\item ${\frak X}_S$ : the complex manifold $\prod_{v\in S} (\C/4\pi i(\log q_v)^{-1}\Z)$. 
\item $\ccA_S=\otimes_{v\in S} \ccA_v$, where for $v\in \Sigma_\fin$, $\ccA_v$ denotes the space of holomorphic functions $\a(s)$ in
$s\in \C/4 \pi i (\log q_v)^{-1}\Z$ such that $\a(-s)=\a(s)$.
\end{itemize}
\end{itemize}

\section{Spectral average of derivatives of $L$-series : the spectral side}
Let $\cB$ be the space of even entire functions $\beta(z)$ on $\C$ 
such that, for any finite interval $I\subset \R$ and for any $N>0$, 
$|\beta(\sigma+it)|\ll_{I,N} (1+|t|)^{-N}$ for $\sigma+it \in I+i\R$.
Given $\beta \in \cB$, $t>0$ and $\lambda \in \C$, we set
$$\beta_{\lambda}^{(1)}(t) = \frac{1}{2\pi i}\int_{L_{\s}}\frac{\b(z)}{(z+\l)^2}t^{z}dz,
$$
where $L_{\s}=\{ z \in \CC\ | \ \Re(z)=\s \}$.
The defining integral is independent of the choice of $\s>-\Re(\l)$. By the residue theorem,  
\begin{align}
{\rm CT}_{\l=0} \{\beta_\lambda^{(1)}(t)-\beta_\lambda^{(1)}(t^{-1})
\} = \b(0) \log t.
 \label{const of beta^1}
\end{align}
In the same way as \cite[Lemma 7.1]{Tsuzuki}, we have the estimate 
\begin{align}
|\b_{\l}^{(1)}(t)| \ll_{\s} \inf\{t^{\s}, t^{-\Re(\l)}\}\log t, \hspace{5mm} t>0,\,\s > -\Re(\l).
\label{estimate of beta^1}
\end{align}

\begin{defn} 
For a cusp form $\varphi$ on $\PGL(2,\A)$, set
$$\partial P^{\eta}_{\beta,\lambda}(\varphi)=
\int_{F^\times \bsl \A^\times} \varphi\left(\left[\begin{smallmatrix} t & 0 \\ 0 & 1 \end{smallmatrix}\right]\,
\left[\begin{smallmatrix} 1 & x_\eta \\ 0 & 1 \end{smallmatrix}\right]\right)
\,\eta(t x_\eta^{*})\{\beta_{\lambda}^{(1)}(|t|_\A)-\beta_{\lambda}^{(1)}(|t|_\A^{-1})\}\,\d^\times t, \quad \Re(\lambda)\gg 0. 
$$
\end{defn}
By \eqref{estimate of beta^1}, the integral $\partial P_{\beta,\lambda}^{\eta}(\varphi)$ is absolutely convergent for $\l \in \CC$ and the function $\l \mapsto \partial P_{\beta,\lambda}^{\eta}(\varphi)$ is entire on $\CC$. Therefore, \eqref{const of beta^1} gives us the formula
$${\rm CT}_{\l=0}\partial P_{\beta,\lambda}^{\eta}(\varphi) = \int_{F^{\times} \backslash \AA^{\times}}
\varphi\left(\left[\begin{smallmatrix} t & 0 \\ 0 & 1 \end{smallmatrix}\right]\,\left[\begin{smallmatrix} 1 & x_\eta \\ 0 & 1 \end{smallmatrix}\right]\right)
\,\eta(t x_\eta^{*}) \log|t|_{\AA} d^{\times}t \ \b(0)
=\frac{d}{ds}Z^{*}(s, \eta, \varphi) \bigg|_{s=1/2}\b(0).$$
Here $Z^{*}(s,\eta,\varphi)$ is the modified global zeta integral considered in
\cite[2.6.2]{Tsuzuki}, \cite[\S 4]{Sugiyama1}, \cite[\S 2.1]{Sugiyama2} and \cite[\S 6.3]{SugiyamaTsuzuki}.

\subsection{}
For $j\in \N_0$, a place $v\in \Sigma_\fin$, an irreducible admissible representation $\pi_v$ of $\PGL(2,F_v)$ and for a character $\eta_v$ of $F_v^\times$ such that $\eta_v^2=1$, we define a polynomial of $X$ by setting
$Q^{\pi_v}_{j,v}(\eta_v,X)=$
\begin{align}
\begin{cases}
1, \qquad &(j=0), \\
\eta_v(\varpi_v)X-Q(\pi_v), \qquad &(c(\pi_v)=0,\,j=1), \\
\eta_v(\varpi_v)^{j-1}X^{j-1}(\eta_v(\varpi_v)X-q_v^{-1}\chi_v(\varpi_v)^{-1}), \qquad& (c(\pi_v)=1,\,j\geq 1), \\
q_v^{-1}\eta_v(\varpi_v)^{j-2}X^{j-2}(a_vq_v^{1/2}\eta_v(\varpi_v)X-1)(a_v^{-1}q_v^{1/2}\eta_v(\varpi_v)X-1), \qquad &(c(\pi_v)=0,\,j\geq 2),\\
\eta_v(\varpi_v)^{j}X^{j}, \qquad &(c(\pi_v)\geq 2, \, j\geq 1),
\end{cases}
 \label{Q(eta,X)}
\end{align} 
({\it cf}. \cite[Corollary 19]{Sugiyama1}), where 
$$
Q(\pi_v)=(a_v+a_v^{-1})/(q_v^{1/2}+q_v^{-1/2})\quad{\text{ with $a_v^{\pm}$ the Satake parameter of $\pi_v$ if $c(\pi_v)=0$}},$$
and $\chi_v$ is the unramified character of $F_v^\times$ such that $\pi_v\cong \sigma(\chi_v|\,|_v^{1/2},\chi_v|\,|_v^{-1/2})$ if $c(\pi_v)=1$.
For $\pi \in \Pi_{\rm{cus}}(l,\fn)$, we set
$$Q_{\pi, \eta, \rho}(s)= \prod_{v \in S(\gn\gf_{\pi}^{-1})} Q_{\rho(v), v}^{\pi_{v}}(\eta_{v}, q_{v}^{1/2-s}), \quad \rho \in \Lambda_\pi(\fn),
$$
where $\Lambda_\pi(\fn)$ denotes the set $\{\rho : \Sigma_{\fin} \rightarrow \NN_{0} \ | \ 0\le \rho(v) \le \ord_{v}(\gn\gf_{\pi}^{-1}) \ (\forall v \in \Sigma_{\fin})\}$. We recall here an explicit formula of the modified zeta integral $Z^{*}(s, \eta, \varphi_{\pi,\rho})$ for the basis $\{\varphi_{\pi,\rho}\}$ of $V_{\pi}[\tau_{l}]^{\bfK_{0}(\gn)}$ (\cite[Proposition 20]{Sugiyama1} and \cite[Propoition 6.1]{SugiyamaTsuzuki}):  
\begin{align}
Z^{*}(s, \eta, \varphi_{\pi, \rho})= D_{F}^{s-1/2}(-1)^{\e(\eta)}\Gcal(\eta) Q_{\pi, \eta, \rho}(s) L(s, \pi \otimes \eta)
 \label{explicit of zeta(s)}
\end{align}
for any $\pi\in \Pi_{\rm{cus}}(l,\fn)$ and $\rho \in \Lambda_{\pi}(\gn)$. 
\subsection{}
Let $\pi\in \Pi_{\rm{cus}}(l,\fn)$ and $\rho\in \Lambda_\pi(\fn)$. For a complex parameter $z$, we set 
\begin{align}
w_\fn^\eta(\pi;z)&=\sum_{\rho\in \Lambda_\pi(\fn)} \prod_{v\in S(\fn\ff_\pi^{-1})} {\overline{Q_{\rho(v),v}^{\pi_v}({\bf 1}, 1)}\,Q_{\rho(v),v}^{\pi_v}(\eta_v,q_v^{1/2-z})}/{\tau_{\pi_v}(\rho(v),\rho(v))}
 \label{1}
\\
&=\prod_{v\in S(\fn\ff_\pi^{-1})}r^{(z)}(\pi_v,\eta_v)
 \label{2}
\end{align}
with 
$$
r^{(z)}(\pi_v,\eta_v)=\sum_{j=0}^{{\rm{ord}}_v(\fn\ff_\pi^{-1})} 
{\overline{Q_{j,v}^{\pi_v}({\bf 1}, 1)}\,Q_{j,v}^{\pi_v}(\eta_v,q_v^{1/2-z})}/{\tau_{\pi_v}(j,j)}
$$
Here $Q_{j,v}^{\pi_v}(\eta_v,X)$ is the polynomial defined by \eqref{Q(eta,X)}, and $\tau_{\pi_v}(j,j)$ is given by \cite[Corollary 12, Corollary 16 and Lemma 3]{Sugiyama1} as
\begin{align}
\tau_{\pi_v}(j,j)=
\begin{cases}
1, \qquad &(j=0\,\text{or}\,c(\pi_v)\geq 2), \\
1-Q(\pi_v)^2, \qquad &(c(\pi_v)=0,\,j=1), \\
1-q_v^{-2}, \qquad &(c(\pi_v)=1,\,j\geq 1), \\
(1-Q(\pi_v)^2)(1-q_v^{-2}), \qquad &(c(\pi_v)=0,\,j\geq 2).
\end{cases}
 \label{tau(j,j)}
\end{align}
Here is the explicit determination of $r^{(z)}(\pi_v,\eta_v)$. 

\begin{lem} \label{r^(z)(pi,eta)}
Let $v\in S(\fn\ff_\pi^{-1})$ and set $k_v={\ord}_v(\fn\ff_\pi^{-1})$, $X=q_v^{1/2-z}$. Suppose $\eta_v(\varpi_v)=-1$. Then we have  
\begin{align*}
r^{(z)}(\pi_v,\eta_v)
=
\begin{cases}
\frac{1-X}{1+Q(\pi_v)}+
\frac{(1+a_vq_v^{1/2}X)(1+a_v^{-1}q_v^{1/2}X)}{(q_v-1)(1+Q(\pi_v))}
\frac{1-(-X)^{k_{v}-1}}{1+X}, \qquad &(c(\pi_v)=0), 
\\
1-
\frac{X+q_v^{-1}\chi_v(\varpi_v)}{1+q_v^{-1}\chi_v(\varpi_v)}
\frac{1-(-1)^{k_v}X^{k_v}}{1+X}, \qquad &(c(\pi_v)=1), \\
\frac{1+(-1)^{k_v}X^{k_v+1}}{1+X}, \qquad &(c(\pi_v)\geq 2).
\end{cases}
\end{align*}
 Suppose $\eta_v(\varpi_v)=1$. Then we have
\begin{align*}
r^{(z)}(\pi_v,\eta_v)
=
\begin{cases}
\frac{1+X}{1+Q(\pi_v)}+
\frac{(1-a_vq_v^{1/2}X)(1-a_v^{-1}q_v^{1/2}X)}{(q_v-1)(1+Q(\pi_v))}\sum_{j=2}^{k_{v}}X^{j-2}, \qquad &(c(\pi_v)=0), 
\\
1+\frac{X-q_v^{-1}\chi_v(\varpi_v)}{1+q_v^{-1}\chi_v(\varpi_v)}
(\sum_{j=1}^{k_v}X^j), \qquad &(c(\pi_v)=1), \\
\sum_{j=0}^{k_v}X^j, \qquad &(c(\pi_v)\geq 2).
\end{cases}
\end{align*}
\end{lem}
\begin{proof}
From \eqref{Q(eta,X)} and \eqref{tau(j,j)}, we obtain the result by a direct computation.
\end{proof}

We abbreviate $r^{(1/2)}(\pi_v,\eta_v)$ to $r(\pi_v,\eta_v)$. Define 
\begin{align*}
w_\fn^\eta(\pi)&=w_\fn^\eta(\pi;1/2), \qquad 
\partial w_{\fn}^{\eta}(\pi)=\left(\frac{\d}{\d z}\right)_{z=1/2} w_{\fn}^{\eta}(\pi;z). 
\end{align*}
Note that the first quantity $w_\fn^\eta(\pi)$ is the same one as in \cite[Lemma 12]{Sugiyama2} and \cite[Lemma 6.2]{SugiyamaTsuzuki}. From Lemma~\ref{r^(z)(pi,eta)}, the second quantity $\partial w_\fn^\eta(\pi)$ is also evaluated explicitly.  

\begin{cor} Set $\partial r(\pi_v,\eta_v)=\frac{-1}{\log q_v} \left(\frac{\d}{\d z}\right)_{z=1/2} r^{(z)}(\pi_v,\eta_v)$. When $\eta_v(\varpi_v)=-1$, 
\begin{align*}
\partial r(\pi_v,\eta_v)=
\begin{cases}
\frac{-1}{1+Q(\pi_v)}+\dfrac{1+(-1)^{k_v}}{2}
\frac{2q_{v}+(q_{v}+1)Q(\pi_{v})}{(q_v-1)(1+Q(\pi_v))}
+\frac{(-1)^{k_{v}}(2k_{v}-3)-1}{4}
\frac{q_{v}+1}{q_{v}-1}, 
 \quad & (c(\pi_v)=0), \\
-\frac{1-(-1)^{k_v}}{2}\dfrac{1}{1+q_v^{-1}\chi_v(\varpi_v)}+
\frac{1+(-1)^{k_v}(2k_v-1)}{4}, \qquad & (c(\pi_v)=1), \\
\frac{(-1)^{k_v}(2k_v+1)-1}{4}, \qquad & (c(\pi_v)\geq 2).
\end{cases}
\end{align*}

When $\eta_v(\varpi_v)=1$, 
\begin{align*}
\partial r(\pi_v,\eta_v)=
\begin{cases}
\frac{1}{1+Q(\pi_v)}+(k_{v}-1)
\frac{2q_{v}-(q_{v}+1)Q(\pi_{v})}{(q_v-1)(1+Q(\pi_v))}
+ \frac{(k_{v}-2)(k_{v}-1)}{2}
\frac{(q_{v}+1)(1-Q(\pi_{v}))}{(q_{v}-1)(1+Q(\pi_{v}))}, 
\quad &(c(\pi_v)=0), \\
\frac{k_v}{1+q_v^{-1}\chi_v(\varpi_v)}+
\frac{1-q_v^{-1}\chi_v(\varpi_v)}{1+
q_v^{-1}\chi_v(\varpi_v)} 
\frac{k_v(k_v+1)}{2}, \qquad & (c(\pi_v)=1), \\
\frac{k_v(k_v+1)}{2}, \qquad & (c(\pi_v)\geq 2).
\end{cases}
\end{align*}

\end{cor}

\subsection{}
Depending on a function $\alpha\in \ccA_S$, we have constructed a cusp form denoted by $\hat{\mathbf\Psi}_{\rm reg}^{l}(\fn|\alpha)$ in \cite[6.5.3]{SugiyamaTsuzuki}. Recall that it has the expression 
\begin{align}
\hat{\bf \Psi}_{\rm reg}^{l}(\gn|\a; g) =
\frac{(-1)^{\#S}\{\prod_{v\in \Sigma_{\infty}}2^{l_{v}-1}\}C_{l}(0)
D_{F}^{-1/2}}{[\bfK_{\fin}:\bfK_{0}(\gn)]}
\sum_{\pi \in \Pi_{\rm cus}(l, \gn)}\sum_{\rho\in \Lambda_{\pi}(\gn)}
\a(\nu_{S}(\pi)) \frac{\overline{Z^*(1/2,{\bf 1},\varphi_{\pi,\rho})}}{\|\varphi_{\pi,\rho}\|^{2}}\, \varphi_{\pi,\rho}(g).
 \label{Psireg-SPECT}
\end{align}

\begin{prop}\label{cuspidal part 1}
We have
{\allowdisplaybreaks 
\begin{align*}
{\rm{CT}}_{\lambda=0}\partial P_{\beta,\lambda}^{\eta}(\hat{\mathbf\Psi}_{\rm reg}^{l}(\fn|\alpha)) 
= & (-1)^{\#S}\{\prod_{v \in \Sigma_{\infty}} 2^{l_{v}-1}\}C_{l}(0)D_{F}^{-1}[\bfK_{\fin} : \bfK_{0}(\gn)]^{-1}
(-1)^{\e(\eta)}\Gcal(\eta)\\
& \times \bigg[ \sum_{\pi \in \Pi_{\rm cus}(l, \gn)}
(\log D_{F})w_{\gn}^{\eta}(\pi)\frac{L(1/2, \pi )L(1/2, \pi \otimes \eta )}{\|\varphi_{\pi}^{\rm new}\|^{2}}\a(\nu_{S}(\pi))\\
& + \sum_{\pi \in \Pi_{\rm cus}(l, \gn)} \partial w_{\fn}^\eta(\pi)\,
\frac{L(1/2, \pi) L(1/2, \pi \otimes \eta)}{\|\varphi_{\pi}^{\rm new}\|^{2}}\a(\nu_{S}(\pi)) \\
& + \sum_{\pi \in \Pi_{\rm cus}(l, \gn)}w_{\gn}^{\eta}(\pi)\frac{L(1/2, \pi )L'(1/2, \pi \otimes \eta )}{\|\varphi_{\pi}^{\rm new}\|^{2}}\a(\nu_{S}(\pi)) \bigg] \b(0).
\end{align*}}
\end{prop}
\begin{proof}
Since the spectral expansion \eqref{Psireg-SPECT} is a finite sum, we have 
\begin{align*}
{\rm{CT}}_{\lambda=0}\partial P_{\beta,\lambda}^{\eta}(\hat{\mathbf\Psi}_{\rm reg}^{l}(\fn|\alpha))
= & (-1)^{\#S}\prod 2^{l_{v}-1}C_{l}(0)D_{F}^{-1/2}[\bfK_{\fin} : \bfK_{0}(\gn)]^{-1} \\
& \times \sum_{\pi \in \Pi_{\rm cus}(l, \gn)}\sum_{\rho\in \Lambda_{\pi}(\gn)}\a(\nu_{S}(\pi)) 
\frac{\overline{Z^*(1/2,{\bf 1},\varphi_{\pi,\rho})}}{\|\varphi_{\pi,\rho}\|^{2}}\,\frac{d}{ds}Z^{*}(s, \eta, \varphi_{\pi,\rho})\bigg|_{s=1/2}\b(0).
\end{align*}
By virtue of \cite[Proposition 6.1]{SugiyamaTsuzuki} and \eqref{explicit of zeta(s)}, we have
{\allowdisplaybreaks 
\begin{align*}
& \sum_{\rho \in \Lambda_{\pi}(\gn)} 
\frac{\overline{Z^*(1/2,{\bf 1},\varphi_{\pi,\rho}))}}{\|\varphi_{\pi,\rho}\|^{2}}\,\frac{d}{ds}Z^{*}(s, \eta, \varphi_{\pi,\rho}) \bigg|_{s=1/2} \\
= & \sum_{\rho \in \Lambda_{\pi}(\gn)}\frac{1}{\|\varphi_{\pi, \rho}\|^{2}}D_{F}^{-1/2}Q_{\pi, \1, \rho}(1/2)L(1/2, \pi)(\log D_{F}) \Gcal(\eta)Q_{\pi, \eta, \rho}(1/2)L(1/2, \pi \otimes \eta) \\
& + \sum_{\rho \in \Lambda_{\pi}(\gn)}\frac{1}{\|\varphi_{\pi, \rho}\|^{2}}D_{F}^{-1/2}Q_{\pi, \1, \rho}(1/2)L(1/2, \pi) \Gcal(\eta) (Q_{\pi, \eta, \rho})'(1/2)L(1/2, \pi \otimes \eta) \\
& + \sum_{\rho \in \Lambda_{\pi}(\gn)}\frac{1}{\|\varphi_{\pi, \rho}\|^{2}}D_{F}^{-1/2}Q_{\pi, \1, \rho}(1/2)L(1/2, \pi) \Gcal(\eta) Q_{\pi, \eta, \rho}(1/2)L'(1/2, \pi \otimes \eta) \\
= & (\log D_{F})D_{F}^{-1/2}\Gcal(\eta)w_{\gn}^{\eta}(\pi)\frac{L(1/2, \pi )L(1/2, \pi \otimes \eta )}{\|\varphi_{\pi}^{\rm new}\|^{2}} \\
& + \left\{\sum_{\rho \in \Lambda_{\pi}(\gn)}
\left( \prod_{v \in S(\gn \gf_{\pi}^{-1})}\frac{Q_{\rho(v), v}^{\pi_{v}}({\bf 1}_{v}, 1)}{\tau_{\pi_{v}}(\rho(v), \rho(v))}\right) (Q_{\pi, \eta, \rho})'(1/2)\right\}
D_{F}^{-1/2}\Gcal(\eta)\frac{L(1/2, \pi) L(1/2, \pi \otimes \eta)}{\|\varphi_{\pi}^{\rm new}\|^{2}} \\
& + D_{F}^{-1/2}\Gcal(\eta)w_{\gn}^{\eta}(\pi)\frac{L(1/2, \pi )L'(1/2, \pi \otimes \eta )}{\|\varphi_{\pi}^{\rm new}\|^{2}}.
\end{align*}
}
By the first expression \eqref{1} of $w^\eta_\gn(\pi,z)$, we have
\begin{align*}
\partial w^\eta_\gn(\pi)=\sum_{\rho\in \Lambda_\pi(\fn)} \left\{\prod_{v\in S(\fn\ff_\pi^{-1})} \frac{\overline{Q_{\rho(v),v}^{\pi_v}({\bf 1}, 1)}}{\tau_{\pi_v}(\rho(v),\rho(v))}\right\} \,(Q_{\pi,\eta, \rho})'(1/2).
\end{align*}
Thus we are done. 
\end{proof}

\section{Spectral average of derivatives of $L$-series: the geometric side}
Recall that the function $\hat{\mathbf\Psi}_{\rm reg}^{l}(\fn|\alpha)$ has another expansion coming from double cosets $H(F) \backslash \GL(2, F)/ H(F)$ (\cite[\S 7]{SugiyamaTsuzuki}):
\begin{align}
\hat{{\bf \Psi}}_{\rm reg}^{l}\left(\gn | \a; [\begin{smallmatrix}t&0\\0&1\end{smallmatrix}][\begin{smallmatrix}1&x_{\eta}\\0&1\end{smallmatrix}]\right)
= (1+i^{\tilde{l}}\delta(\gn=\go))J_{\rm id}(\a;t) + J_{\rm u}(\a;t)+J_{\bar{\rm u}}(\a;t) +J_{\rm hyp}(\a;t),\ \ \
t\in \A^\times,
 \label{GeomFormula}
\end{align}
where the terms in the right-hand side are defined in \cite[Lemma 7.1, Lemma 7.2 and Lemma 7.11]{SugiyamaTsuzuki}. For $\natural \in \{\rm id, u, \bar{u}, hyp \}$, we consider the ``orbital integrals''
$$\WW_{\natural}^{\eta}(\b, \l; \a) = \int_{F^{\times}\backslash \AA^{\times}} J_{\natural}(\a; t) \{\b_{\l}^{(1)}(|t|_{\AA}) - \b_{\l}^{(1)}(|t|_{\AA}^{-1})\}\eta(tx_{\eta}^{*})d^{\times}t$$
for $\a \in \Acal$, $\b \in \Bcal$ and $\l \in \CC$ such that $\Re(\l)>1$. We shall show that these integrals converge absolutely individually when $\Re(\l)>1$ and admit an analytic continuation in a neighborhood of $\l=0$.

\begin{lem}\label{int of beta^1}
Let $\l$ and $w$ be complex numbers such that $\Re(w) < \Re(\l)$. Let $\xi$ be an idele class character of $F^\times$. Then, we have
$$\int_{F^{\times}\backslash \AA^{\times}}\b_{\l}^{(1)}(|t|_\A)\xi(t)|t|_{\AA}^{w}d^{\times}t = \delta_{\xi, {\bf 1}}\vol(F^{\times}\backslash\AA^{1})\frac{\b(-w)}{(\l-w)^{2}}.$$
\end{lem}
\begin{proof}
The proof is given in the same way as \cite[Lemma 7.6]{Tsuzuki}.
\end{proof}

\begin{lem}
For $\Re(\l)>0$, the integral $\WW_{\rm id}^{\eta}(\b, \l; \a)$ converges absolutely and $\WW_{\rm id}^\eta(\b,\l;\a)= 0$.
\end{lem}
\begin{proof}
This follows immediately from Lemma \ref{int of beta^1},
since $J_{\rm id}(\a; t)$ is independent of the variable $t$ (\cite[Lemma 7.1]{SugiyamaTsuzuki}).
\end{proof}

Assume that $q(\Re(\bfs))>\Re(\l)> \s >1$.
Set
{\allowdisplaybreaks
\begin{align*}
V_{0, \eta}^{\pm}(\l; \bfs)&= \frac{1}{2\pi i}\int_{L_{\mp \s}}\frac{\beta(z)}{(z+\l)^2} \int_{\AA^{\times}}\Psi_{l}^{(0)}\left(\gn|\bfs; [\begin{smallmatrix}1&t^{-1}\\ 0&1\end{smallmatrix}][\begin{smallmatrix}1&x_{\eta}\\ 0&1\end{smallmatrix}]\right)\eta(tx_{\eta}^{*})|t|_\A^{\pm z}d^{\times}tdz, \\
V_{1, \eta}^{\pm}(\l; \bfs)&=\frac{1}{2\pi i}\int_{L_{\mp \s}}\frac{\beta(z)}{(z+\l)^2} \int_{\AA^{\times}}\Psi_{l}^{(0)}\left(\gn | \bfs; [\begin{smallmatrix}1&0\\ t^{-1}&1\end{smallmatrix}][\begin{smallmatrix}1&0\\ -x_{\eta}&1\end{smallmatrix}]w_{0}\right)\eta(tx_{\eta}^{*})|t|_\A^{\pm z}d^{\times}tdz
\end{align*}
}
and
\begin{align*}
\Upsilon_{S}^{\eta}(z; \bfs)&=\prod_{v \in S}(1-\eta_{v}(\varpi_{v})q_{v}^{-(z+(s_{v}+1)/2)})^{-1}(1-q_{v}^{(s_{v}+1)/2})^{-1}, \\
\Upsilon_{S, l}^{\eta}(z; \bfs)= & D_{F}^{-1/2}\{\# (\go/\gf_{\eta})^{\times}\}^{-1}
\{\prod_{v \in \Sigma_{\infty}}
\frac{2\Gamma(-z) \Gamma(l_{v}/2+z)}{\Gamma_{\RR}(-z + \e_{v})\Gamma(l_{v}/2)}i^{\e_{v}}\cos \left(\frac{\pi}{2}(-z+\e_{v})\right)\}\,
\Upsilon_{S}^{\eta}(z; \bfs).
\end{align*}

\begin{lem} \label{V-eta-pm}
The double integrals $V_{j,\eta}^{\pm}(\lambda;\bs)$ converge absolutely and 
$$V_{0,\eta}^{\pm}(\l; \bfs) = \frac{1}{2\pi i}\int_{L_{\s}}\frac{\b(z)}{(z+\l)^{2}}{\rm N}(\gf_{\eta})^{\mp z}L(\mp z, \eta)(-1)^{\e(\eta)}\Upsilon_{S, l}^{\eta}(\pm z; \bfs)dz$$
and
$$V_{1,\eta}^{\pm}(\l; \bfs) = \frac{1}{2\pi i}\int_{L_{\s}}\frac{\b(z)}{(z+\l)^{2}}{\rm N}(\gf_{\eta})^{\mp z}{\rm N}(\gn)^{\pm z}\tilde{\eta}(\gn)
\delta(\gn=\go)L(\mp z, \eta)i^{\tilde{l}}\Upsilon_{S, l}^{\eta}(\pm z; \bfs)dz.$$
\end{lem}
\begin{proof}
As in \cite[Lemma 8.2]{SugiyamaTsuzuki}, we exchange the order of integrals and compute the $t$-integrals first. Since $\eta\not=1$, the integrands in the remaining contour integrals in $z$ are holomorphic on $|\Re(z)|<\s$; thus we can shift the contour $L_{-\s}$ to $L_{\s}$ for $V_{0,\eta}^{+}$ and $V_{1,\eta}^{+}$. 
\end{proof}


\begin{lem} \label{WWu}
The integral $\WW_{\rm u}^{\eta}(\b, \l; \a)$ has an analytic continuation to $\CC$ as a function in $\l$.
The constant term of $\WW_{\rm u}^{\eta}(\b, \l; \a)$ at $\l=0$ equals
$\WW_{\rm u}^{\eta}(l, \gn|\a)\b(0)$ with 
$$\WW_{\rm u}^{\eta}(l, \gn|\a) =
(-1)^{\e(\eta)} \Gcal(\eta) D_{F}^{1/2}
(1+ (-1)^{\e(\eta)}\tilde{\eta}(\gn)i^{\tilde{l}}\delta(\gn=\go))
\left(\frac{1}{2\pi i}\right)^{\#S}\int_{\LL_{S}(\bfc)}
\frak{W}_{S}^{\eta}(l,\gn| \bfs)\a(\bfs)d\mu_{S}(\bfs),
$$
where $\Upsilon_{S}^{\eta}(\bfs)=\Upsilon_{S}^{\eta}(0;\bfs)$ and
\begin{align*}
\frak{W}_{S}^{\eta}(\bfs) = &
\pi^{\e(\eta)}\Upsilon_{S}^{\eta}(\bfs) L(1,\eta)
\bigg\{ \log D_{F}+\frac{L'(1, \eta)}{L(1,\eta)} \\
& + \sum_{v \in \Sigma_{\infty}}\left(
\sum_{k=1}^{l_{v}/2-1}\frac{1}{k}
- \frac{1}{2}\log \pi - \frac{1}{2}C_{\rm Euler} -\delta_{\e_{v},1}\log 2 \right)
+ \sum_{v \in S}\frac{\log q_{v}}{1-\eta_{v}(\varpi_{v})q_{v}^{(s_{v}+1)/2}}
\bigg\}.
\end{align*}
\end{lem}
\begin{proof} 
From \cite[Lemma 7.2]{SugiyamaTsuzuki}, we have the expression
$$\WW_{\rm u}^{\eta}(\b, \l; \a)=\left(\frac{1}{2\pi i}\right)^{\#S}\int_{\LL_{S}(\bfc)}\{ V_{0, \eta}^{+}(\l; \bfs) - V_{0, \eta}^{-}(\l; \bfs) + V_{1, \eta}^{+}(\l; \bfs) - V_{1, \eta}^{-}(\l; \bfs)\}\a(\bfs)d\mu_{S}(\bfs).$$
By Lemma~\ref{V-eta-pm}, the right-hand side becomes
\begin{align*}
&((-1)^{\e(\eta)} + i^{\tilde{l}}\delta(\gn=\go))\left(\frac{1}{2\pi i}\right)^{\#S}\int_{\LL_{S}(\bfc)}\frac{1}{2\pi i}\int_{L_{\s}}\frac{\b(z)}{(z+\l)^{2}}\{ {\rm N}(\gf_{\eta})^{-z}L(-z, \eta) \Upsilon_{S, l}^{\eta}(z; \bfs) \\
& - {\rm N}(\gf_{\eta})^{z}L(z, \eta) \Upsilon_{S, l}^{\eta}(-z; \bfs) \}dz \a(\bfs) d\mu_{S}(\bfs).
\end{align*}
which is holomorphic on $\Re(\l)>-\s$. Since $\s>1$ is arbitrary, this gives an analytic continuation of $\WW_{\rm u}^{\eta}(\b, \l; \a)$ to $\C$ and yields the equality
\begin{align*}
& {\rm CT}_{\l = 0}\WW_{\rm u}^{\eta}(\b, \l; \a) \\
= & \left(\frac{1}{2\pi i}\right)^{\#S}\int_{\LL_{S}(\bfc)}\left(\frac{1}{2\pi i}\int_{L_{\s}}\frac{\b(z)}{z^2} \{ f_{\rm u}(z) - f_{\rm u}(-z)\} dz\right)\a(\bfs)d\mu_{S}(\bfs)\\
= & ((-1)^{\e(\eta)}+ i^{\tilde{l}}\delta(\gn=\go)) \Res_{z=0}\left(\frac{\beta(z)}{z^2}f_{\rm u}(z)\right) \\
= & ((-1)^{\e(\eta)}+ i^{\tilde{l}}\delta(\gn=\go))({\rm CT}_{z=0}\frac{f_{\rm u}(z)}{z}\b(0) + \frac{1}{2}\Res_{z=0}f_{\rm u}(z) \b''(0)),
\end{align*}
where $f_{\rm u}(z) = {\rm N}(\gf_{\eta})^{-z}L(-z, \eta) \Upsilon_{S, l}^{\eta}(z; \bfs)$.
Since $\eta$ is nontrivial,
by the functional equation
$$
L(s,\eta)=i^{\epsilon(\eta)}D_F^{1-s}\nr(\ff_\eta)^{-s}\, \#((\cO/\ff_\eta)^\times)\,\cG(\eta)\,L(1-s,\eta),$$
$f_{\rm u}(z)$ is holomorphic at $z=0$. Thus, 
\begin{align*}
& {\rm CT}_{z=0}\frac{f_{\rm u}(z)}{z} = \lim_{z \rightarrow 0}\frac{f_{\rm u}(z)-f_{\rm u}(0)}{z} \\
= & -(\log{\rm N}(\gf_{\eta})) L(0, \eta) \Upsilon_{S, l}^{\eta}(0; \bfs) - L'(0, \eta)\Upsilon_{S, l}^{\eta}(0; \bfs)
+L(0, \eta)(\Upsilon_{S, l}^{\eta})'(0; \bfs) \\
= & i^{\e(\eta)}\Gcal(\eta)D_{F}^{1/2}\tilde{\Upsilon}_{S, l}^{\eta}(0; \bfs) \{ -L(1, \eta)\log {\rm  N}(\gf_{\eta})
+L(1, \eta)\log (D_{F}{\rm N}(\gf_{\eta})) +L'(1, \eta) + L(1, \eta)\frac{d}{dz}\log \tilde{\Upsilon}_{S,l}(z; \bfs)|_{z=0}\}\\
= & \Gcal(\eta)D_{F}^{1/2}\pi^{\e(\eta)}\Upsilon_{S}^{\eta}(\bfs)
\{ L(1, \eta)\log D_{F} +L'(1, \eta) + L(1, \eta)\frac{d}{dz}\log \tilde{\Upsilon}_{S,l}(z; \bfs)|_{z=0}\},
\end{align*}
where $\tilde{\Upsilon}_{S, l}^{\eta}(z; \bfs)= D_{F}^{1/2}\#((\cO/\ff_\eta)^\times) \Upsilon_{S, l}^{\eta}(z; \bfs)$.
Furthermore,
\begin{align*}
\frac{d}{dz}\log \tilde{\Upsilon}_{S,l}(z; \bfs)|_{z=0}
= & \sum_{v \in \Sigma_{\infty}}\bigg(\psi(l_{v}/2) - \frac{1}{2}\log \pi + \frac{1}{2}\psi\left(\frac{-z+\e_{v}}{2}\right)
-\psi(-z) +\frac{\pi}{2}\tan\frac{\pi}{2}(-z+\e_{v})\bigg)\bigg|_{z=0} \\
& +\sum_{v \in S}\frac{\log q_{v}}{1-\eta_{v}(\varpi_{v})q_{v}^{(s_{v}+1)/2}}.
\end{align*}
Here, by $\psi(1)=-C_{\rm Euler}$, $\psi(1/2)=-C_{\rm Euler}-2\log 2$ and
$\frac{d}{dt}\left(t \cot t\right)|_{t=0}=0,$
we have
$$\frac{1}{2}\psi\left(\frac{-z+\e_{v}}{2}\right)
-\psi(-z) +\frac{\pi}{2}\tan\frac{\pi}{2}(-z+\e_{v}) \bigg|_{z=0} =
\begin{cases}\displaystyle\frac{1}{2}C_{\rm Euler} & \text{($\e_{v}=0$),} \\
\ \\
\displaystyle\frac{1}{2}\psi\left(\frac{1}{2}\right)-\psi(1)=\frac{1}{2}C_{\rm Euler}-\log 2 & \text{($\e_{v}=1$).}
\end{cases}
$$
\end{proof}

Assume that $q(\Re(\bfs))>\Re(\l)> \s >1$. Set
{\allowdisplaybreaks
\begin{align*}
\tilde{V}_{1, \eta}^{\pm}(\l; \bfs)&= \frac{1}{2\pi i}\int_{L_{\pm \s}}\frac{\b(z)}{(z+\l )^2}\int_{\AA^{\times}}\Psi_{l}^{(0)}\left(\gn|\bfs; [\begin{smallmatrix}1&0 \\ t &1\end{smallmatrix}][\begin{smallmatrix}1&x_{\eta}\\ 0&1\end{smallmatrix}]\right)\eta(tx_{\eta}^{*})|t|_\A^{\pm z}d^{\times}tdz, \\
\tilde{V}_{0, \eta}^{\pm}(\l; \bfs)&= \frac{1}{2\pi i}\int_{L_{\pm \s}}\frac{\b(z)}{(z+\l )^2}\int_{\AA^{\times}}\Psi_{l}^{(0)}\left(\gn|\bfs; [\begin{smallmatrix}1&t \\ 0 &1\end{smallmatrix}][\begin{smallmatrix}1&0 \\ - x_{\eta}&1\end{smallmatrix}]w_{0}\right)\eta(tx_{\eta}^{*})|t|_\A^{\pm z}d^{\times}tdz.
\end{align*}
}In the same way as Lemma~\ref{V-eta-pm}, we obtain
\begin{lem} \label{tildeV-eta-pm}
The double integrals $\tilde V_{j,\eta}^{\pm}(\l;\bs)$ converge absolutely and \begin{align*}
\tilde{V}^{\pm}_{1,\eta}(\l; \bfs)&= \frac{1}{2\pi i}\int_{L_{\s}}
\frac{\b(z)}{(z+\l )^2}{\rm N}(\gf_{\eta})^{\mp z}{\rm N}(\gn)^{\mp z}
\tilde{\eta}(\gn)L(\pm z, \eta)
\Upsilon_{S, l}^{\eta}(\mp z; \bfs) dz, \\
\tilde{V}^{\pm}_{0,\eta}(\l; \bfs)&= \frac{1}{2\pi i}\int_{L_{\s}}
\frac{\b(z)}{(z+\l )^2}{\rm N}(\gf_{\eta})^{\mp z}\delta(\gn=\go)
L(\pm z, \eta) (-1)^{\e(\eta)}i^{\tilde{l}}\Upsilon_{S, l}^{\eta}(\mp z; \bfs) dz.
\end{align*}
\end{lem}


\begin{lem} \label{WWbu}
The integral $\WW_{\rm \bar u}^{\eta}(\b, \l; \a)$ converges absolutely on $\Re(\l)>1$ and has an
analytic continuation to $\CC$ as a function in $\l$.
The constant term of $\WW_{\rm \bar u}^{\eta}(\b, \l; \a)$ at $\l=0$ equals
$\WW_{\rm \bar u}^{\eta}(l, \gn|\a)\b(0)$ with
$$\WW_{\rm \bar u}^{\eta}(l, \gn|\a) =
(-1)^{\e(\eta)} \Gcal(\eta) D_{F}^{1/2}
((-1)^{\e(\eta)}\tilde{\eta}(\gn)+i^{\tilde{l}}\delta(\gn=\go))
\left(\frac{1}{2\pi i}\right)^{\#S}\int_{\LL_{S}(\bfc)}
\frak{W}_{S, {\rm \bar u}}^{\eta}(l, \gn |\bfs)\a(\bfs)d\mu_{S}(\bfs),
$$
where
$$\frak{W}_{S, {\rm \bar u}}^{\eta}(l, \gn |\bfs)
=-\pi^{\e(\eta)}\Upsilon_{S}^{\eta}(\bfs)
L(1, \eta)\log ({\rm N}(\gn){\rm N}(\gf_{\eta})^{2})
-\frak{W}_{S, {\rm u}}^{\eta}(l, \gn |\bfs).
$$
\end{lem}
\begin{proof}
From \cite[Lemma 7.2]{SugiyamaTsuzuki}, we have the expression
$$\WW_{\rm \bar u}^{\eta}(\b, \l; \a)=\left(\frac{1}{2\pi i}\right)^{\#S}\int_{\LL_{S}(\bfc)}\{ \tilde{V}_{0, \eta}^{+}(\l; \bfs) - \tilde{V}_{0, \eta}^{-}(\l; \bfs) + \tilde{V}_{1, \eta}^{+}(\l; \bfs) - \tilde{V}_{1, \eta}^{-}(\l; \bfs)\}\a(\bfs)d\mu_{S}(\bfs).$$
By Lemma~\ref{tildeV-eta-pm}, the right-hand side becomes
\begin{align*}
& (\tilde{\eta}(\gn) + (-1)^{\e(\eta)}i^{\tilde{l}}\delta(\gn=\go))
\left(\frac{1}{2\pi i}\right)^{\#S}\int_{\LL_{S}(\bfc)}\frac{1}{2\pi i}\int_{L_{\s}}\frac{\b(z)}{(z+\l)^{2}}
 \times \{
{\rm N}(\gf_{\eta})^{-z}{\rm N}(\gn)^{-z}L(z, \eta) \Upsilon_{S, l}^{\eta}(-z; \bfs) \\
& - {\rm N}(\gf_{\eta})^{z}{\rm N}(\gn)^{z}L(-z, \eta) \Upsilon_{S, l}^{\eta}(z; \bfs) \}dz \a(\bfs) d\mu_{S}(\bfs).
\end{align*}
As before, this gives an analytic continuation of $\WW_{\rm \bar u}^{\eta}(\b, \l; \a)$.
We set $f_{\rm \bar u}(z) = - {\rm N}(\gf_{\eta})^{2z}{\rm N}(\gn)^{z} f_{\rm u}(z)$. Then,
\begin{align*}
& {\rm CT}_{\l = 0}\WW_{\rm \bar u}^{\eta}(\b, \l; \a) \\
= & (\tilde{\eta}(\gn) + (-1)^{\e(\eta)}i^{\tilde{l}}\delta(\gn=\go))({\rm CT}_{z=0}\frac{f_{\rm \bar u}(z)}{z}\b(0) + \frac{1}{2}\Res_{z=0}f_{\rm \bar u}(z) \b''(0)).
\end{align*}
Since $\eta$ is supposed to be nontrivial, $f_{\rm \bar u}(z)$ is holomorphic at $z=0$ and 
\begin{align*}
{\rm CT}_{z=0}\frac{f_{\rm \bar u}(z)}{z} = & f_{\rm \bar u}'(0)
= -\log({\rm N}(\gn){\rm N}(\gf_{\eta}^2))f_{\rm u}(0) - f_{\rm u}'(0) \\
= & \Gcal(\eta)D_{F}^{1/2}\{
-\pi^{\e(\eta)}\Upsilon_{S}^{\eta}(\bfs)L(1,\eta) \log({\rm N}(\gn){\rm N}(\gf_{\eta}^2))
-\frak{W}_{S, {\rm u}}^{\eta}(l, \gn |\bfs)
\}.
\end{align*}
\end{proof}

\begin{lem} \label{WWhyp}
The integral $\WW_{\rm hyp}^{\eta}(\b, \l; \a)$ converges absolutely and has an analytic continuation to the region $\Re(\l)>-\e$ for some $\e>0$.
The constant term of $\WW_{\rm hyp}^{\eta}(\b, \l; \a)$ at $\l=0$ equals $\WW_{\rm hyp}^{\eta}(l, \gn|\a)\b(0)$.
Here
$$\WW_{\rm hyp}^{\eta}(l, \gn|\a) = \left(\frac{1}{2\pi i}\right)^{\#S}\int_{\LL_{S}(\bfc)}\frak{L}_{\eta}(l, \gn |\bfs)\a(\bfs)d\mu_{S}(\bfs)$$
with
$$\frak{L}_{\eta}(l, \gn |\bfs)=\sum_{b \in F-\{0, -1\}}\int_{\AA^{\times}}\Psi_{l}^{(0)}(\gn|\bfs, \delta_{b}
\left[\begin{smallmatrix}t & 0 \\ 0 & 1 \end{smallmatrix}\right]\left[\begin{smallmatrix}1 & x_{\eta} \\ 0 & 1\end{smallmatrix}\right])\eta(t x_{\eta}^{*}) \log |t|_{\AA} d^{\times}t.$$
\end{lem}
\begin{proof}
The absolute convergence and analytic continuation of $\WW_{\rm hyp}^{\eta}(\b, \l; \a)$ are given
in the same way as \cite[Lemma 8.5]{SugiyamaTsuzuki}.
We obtain the last assertion with the aid of (\ref{const of beta^1}).
\end{proof}

From the analysis so far, \eqref{GeomFormula} yields the formula:
\begin{align}
\partial P_{\b,\l}^{\eta}( 
\hat{{\bf \Psi}}_{\rm reg}^{l}(\gn | \a))=\WW_{{\rm {u}}}^\eta(\b,\l;\a)+\WW_{\bar{\rm {u}}}^\eta(\b,\l;\a)+\WW_{{\rm {hyp}}}^\eta(\b,\l;\a)
\label{GeomFormula1}
\end{align}
which is valid on a half-plane $\Re(\l)>-\e$ containing $\l=0$.

\subsection{The relative trace formula}
For any ideal $\fm\subset\cO$, set
\begin{align}
\iota(\fm)=[\bK_\fin:\bK_0(\fm)]=\prod_{v\in S(\fm)}(1+q_v)q_v^{\ord_{v}(\fm)-1}.
 \label{iota-fm}
\end{align}
Let $\Jcal_{S,\eta}$ be the monoid of ideals generated by prime ideals $\fp_v$ with $v\in \Sigma_\fin-S\cup S(\ff_\eta)$. We shall introduce several functionals in $\alpha \in \ccA_S$ depending on an ideal $\fm\in \Jcal_{S,\eta}$:  
\begin{align}
\AL^{w}(\fm;\alpha)&=C_l\, 
\sum_{\pi \in \Pi_{\rm{cus}}(l,\fm)} \frac{w_\fm^\eta(\pi)\,\iota(\ff_\pi)
}{\nr(\ff_\pi)\iota(\fm)}
\frac{L(1/2,\pi)\,L(1/2,\pi\otimes\eta)}{L^{S_{\pi}}(1,\pi, \Ad)}\,\alpha(\nu_S(\pi)), 
 \label{ALw}
\\
\AL^{\partial w}(\fm;\alpha)&=C_l\,\sum_{\pi \in \Pi_{\rm{cus}}(l,\fm)} 
\frac{\partial w_\fm^\eta(\pi)\,\iota(\ff_\pi)}{\nr(\ff_\pi)\,\iota(\fm)}\frac{L(1/2,\pi)\,L(1/2,\pi\otimes\eta)}{L^{S_{\pi}}(1,\pi, \Ad)
}\,\alpha(\nu_S(\pi)),
\label{ALpw}
\\
\ADL^{w}_{\pm }(\fm;\alpha)&=C_l\, 
\sum_{\substack{\pi \in \Pi_{\rm{cus}}(l,\fm) \\ \epsilon(1/2,\pi\otimes \eta)=\pm 1}} \frac{w_\fm^\eta(\pi)\,\iota(\ff_\pi)}{\nr(\ff_\pi)\,\iota(\fm)}
\frac{L(1/2,\pi)\,L'(1/2,\pi\otimes\eta)}{L^{S_{\pi}}(1,\pi, \Ad)}\,\alpha(\nu_S(\pi)).
\label{ADLwpm}
\end{align}

The derivative of $L$-functions in $\ADL^{w}_{+}$ is eliminated by the functional equation. 
\begin{prop} \label{ADL^w_+}
We have
\begin{align*}
\ADL_{+}^{w}(\fm;\alpha)&=C_l\, 
\sum_{\pi \in \Pi_{\rm{cus}}(l,\fm)} \log \{\nr(\ff_\pi\ff_\eta^2) D_{F}^{2}\}^{-1/2}
\,\frac{w_\fm^\eta(\pi)\,\iota(\ff_\pi)}{\nr(\ff_\pi)\,\iota(\fm)}\frac{L(1/2,\pi)\,L(1/2,\pi\otimes\eta)}{L^{S_{\pi}}(1,\pi, \Ad)}\,\alpha(\nu_S(\pi)).
\end{align*}
\end{prop}
\begin{proof}
By the functional equation, 
$$L'(1/2, \pi \otimes \eta) = \frac{\e'(1/2, \pi \otimes \eta)}{2}L(1/2, \pi \otimes \eta)$$
if $\e(1/2, \pi \otimes \eta)=1$. An explicit form of the $\e$-factor is given by
$\e(s, \pi\otimes \eta) = \e(1/2, \pi \otimes \eta)\{{\rm N}(\gf_{\pi \otimes \eta})D_{F}^{2}\}^{1/2-s}
= \e(1/2, \pi \otimes \eta)\{{\rm N}(\gf_{\pi}){\rm N}(\gf_{\eta})^2 D_{F}^{2}\}^{1/2-s}$. Hence we obtain the assertion immediately.
\end{proof}

The following is the main consequence of this section. 
\begin{thm} \label{DRTF}
For any ideal $\fn\in \Jcal_{S,\eta}$ and for any $\a\in \ccA_S$, 
\begin{align}
& 
2^{-1}(-1)^{\#S +\epsilon(\eta)}\cG(\eta)\,D_F^{-1}\,\{\ADL_{-}^{w}(\fn;\alpha)+\ADL_{+}^{w}(\fn;\alpha)+ (\log D_{F}) \AL^{w}(\fn;\alpha)+\AL^{\partial w}(\fn;\alpha)\}
 \label{DRTF-1}
\\
= & \tilde{\WW}_{\rm u}^{\eta}(l, \gn | \a)
+ \WW_{\rm hyp}^{\eta}(l, \gn | \a).
 \notag
\end{align}
Here
\begin{align}
\tilde{\WW}_{\rm u}^{\eta}(l, \gn | \a) & = 
(1-(-1)^{\e(\eta)}\tilde{\eta}(\gn))(-1)^{\e(\eta)} \Gcal(\eta) D_{F}^{1/2}
\{1 +(-1)^{\e(\eta)}\tilde\eta(\fn)i^{\tilde{l}}\delta(\gn=\go)\}
\label{DRTF-2}
\\
&\quad \times \left(\frac{1}{2\pi i}\right)^{\#S}\int_{\LL_{S}(\bfc)}
\tilde{\frak{W}}_{S}^{\eta}(l, \gn |\bfs)\a(\bfs)d\mu_{S}(\bfs)
\notag
\end{align}
with $\d \mu_{S}(\bfs)=\prod_{v\in S}2^{-1}\log q_v(q_v^{(1+s_v)/2}-q_v^{(1-s_v)/2})\,\d s_v$ and $\LL_S(\bfc)$ being the multidimensional contour $\prod_{v\in S}\{\Re(s_v)=c_v\}$ directed usually, 
\begin{align}
\tilde{\frak{W}}_{S}^{\eta}(l, \gn |\bfs) = &
\pi^{\e(\eta)}\Upsilon_{S}^{\eta}(\bfs)L(1,\eta)
\bigg\{
\log (\sqrt{{\rm N}(\gn)}{D_{F}{\rm N}(\gf_{\eta})}) +\frac{L'(1, \eta)}{L(1,\eta)}  + {\frak C}(l)
+ \sum_{v \in S}\frac{\log q_{v}}{1-\eta_{v}(\varpi_{v})q_{v}^{(s_{v}+1)/2}}
\bigg\},
 \label{DRTF-3}
\end{align}
\begin{align*}
\Upsilon_S^\eta(\bfs)&=\prod_{v\in S}(1-\eta_v(\varpi_v)q_v^{-(1+s_v)/2})^{-1}(1-q_v^{(1+s_v)/2})^{-1}, \\ 
{\frak C}(l)&=\sum_{v\in \Sigma_\infty} \left(\sum_{k=1}^{l_{v}/2-1}\frac{1}{k}
- \frac{1}{2}\log \pi - \frac{1}{2}C_{\rm Euler} -\delta_{\e_{v},1}\log 2 \right). 
\end{align*}
\end{thm}
\begin{proof}
From Proposition~\ref{cuspidal part 1} together with \cite[Lemma 6.4]{SugiyamaTsuzuki}, 
\begin{align*}
& {\rm{CT}}_{\lambda=0}\partial P_{\beta,\lambda}^{\eta}(\hat{\mathbf\Psi}_{\rm reg}^{l}(\fn|\alpha)) \\
&=2^{-1}(-1)^{\#S +\epsilon(\eta)}\cG(\eta)\,D_F^{-1}\,\{\ADL_{-}^{w}(\fn;\alpha)+\ADL_{+}^{w}(\fn;\alpha)+ (\log D_{F}) \AL^{w}(\fn;\alpha)+\AL^{\partial w}(\fn;\alpha)\}.\end{align*}
On the other hand, from the formula \eqref{GeomFormula1}, the same ${\rm{CT}}_{\lambda=0}\partial P_{\beta,\lambda}^{\eta}(\hat{\mathbf\Psi}_{\rm reg}^{l}(\fn|\alpha))$ is computed by Lemmas~\ref{WWu}, \ref{WWbu} and \ref{WWhyp}. 
\end{proof}

\section{Extraction of the new part : the totally inert case}

Let $\Ical_{S,\eta}$ be the monoid of ideals generated by prime ideals $\fp_v$ such that $v\in \Sigma_\fin-S\cup S(\ff_\eta)$ and $\tilde \eta(\fp_v)=-1$. Note that $\Ical_{S,\eta}$ is a submonoid of $\Jcal_{S,\eta}$ defined in 3.1. 

In this section, we separate the new part i.e., the contribution of those $\pi$ with $\ff_{\pi}=\fn$, from the total average $\ADL_{-}^{w}(\fn;\alpha)$ under the condition $\fn\in \Ical_{S,\eta}$.

\subsection{The $\cN$-transform}
For any ideal $\fc$ and a place $v\in \Sigma_\fin$, set
$$
\omega_v(\fc)=\begin{cases} 1, \qquad & (v\in S(\fc)), \\
 \dfrac{q_v+1}{q_v-1}, \qquad &(v\not\in S(\fc)).
\end{cases}
$$

For any pair of integral ideals $\fm$ and $\fb$, define 
$$
\omega(\fm,\fb)=
\delta(\fm\subset \fb)\,\prod_{v\in S(\fb)} \omega_v(\fm\fb^{-1}).
$$
Given an ideal $\fn$, let $\fn_0$ denote the largest square free integral ideal dividing $\fn$; thus, there exists the unique integral ideal $\fn_1$ such that 
$$ \fn=\fn_0\fn_1^2. $$
Let $\mathcal I$ be a set of integral ideals with the property that $\fn\subset \fm$, $\fn\in \mathcal I$ implies $\fm\in \mathcal I$.

\begin{prop} \label{INV-F}
Let $B(\fm)$ and $A(\fm)$ be two arithmetic functions defined for ideals $\fm\in\mathcal I$. Then, the following two conditions are equivalent:
\begin{itemize}
\item[(i)] For any $\fn\in \mathcal I$, 
\begin{align*}
B(\fn)&=\sum_{\fb|\fn_1} \omega(\fn,\fb^2)\,A(\fn\fb^{-2}).
\end{align*}
\item[(ii)] For any $\fn\in \mathcal I$, 
\begin{align*}
A(\fn)=\sum_{I\subset S(\fn_1)}(-1)^{\# I}\{\prod_{v\in I\cap S_1(\fn_1)}\omega_v(\fn_0)\}\,B(\fn\prod_{v\in I}\fp_v^{-2}).
\end{align*}
\end{itemize}
\end{prop}
\begin{proof}
We show that (i) implies (ii). By substituting (i), the right-hand side of (ii) becomes
\begin{align*}
&\sum_{I\subset S(\fn_1)}(-1)^{\# I}\{\prod_{v\in I\cap S_1(\fn_1)}\omega_v(\fn_0)\}\,\{\sum_{\fb|\fn_1\prod_{v\in I} \fp_v^{-1}}\omega(\fn\prod_{v\in I} \fp_v^{-2},\fb^2)\,A(\fn\fb^{-2}\prod_{v\in I}\fp_v^{-2})\}
\\
&=\sum_{\fb_1|\fn_1}\biggl\{\sum_{I\subset S(\fn_1\fb_1^{-1})}(-1)^{\# I} \omega\left(\fn\prod_{v\in I}\fp_v^{-2},\fn_1^2\fb_1^{-2}\prod_{v\in I}\fp_v^{-2}\right)\prod_{v\in I\cap S_1(\fn_1)\cap S(\fn_1\fb_1^{-1})}\omega_v(\fn_0)\biggr\}\,A(\fn_0\fb_1^2)
\end{align*}
Here to have the equality, we made the substitution $\fb_1=\fn_1\fb^{-1} \prod_{v\in I}\fp_v^{-1}$. If $\fb_1=\fn_1$, the term inside the bracket is $1$ obviously; otherwise it equals
\begin{align*}
&\sum_{I\subset S(\fn_1\fb_1^{-1})}(-1)^{\# I} \prod_{v\in S(\fn_1\fb_1^{-1}\prod_{v\in I} \fp_v^{-1})-S(\fn_0\fb_1^2)}\frac{q_v+1}{q_v-1}\,\prod_{v\in I \cap S(\fn_1\fb_1^{-1})\cap S_1(\fn_1)-S(\fn_0)}\frac{q_v+1}{q_v-1}\\
&=
\sum_{I\subset S(\fn_1\fb_1^{-1})}(-1)^{\# I} \prod_{v\in [(I-S_1(\fn_1\fb_1^{-1}))\cup (S(\fn_1\fb_1^{-1})-I)]-S(\fn_0\fb_1^2)}\frac{q_v+1}{q_v-1}\,\prod_{v\in I \cap S_1(\fn_1\fb_1^{-1})-S(\fn_0\fb_1^2)}\frac{q_v+1}{q_v-1}
\\
&=
\sum_{I\subset S(\fn_1\fb_1^{-1})}(-1)^{\# I} \prod_{v\in (S(\fn_1\fb_1^{-1})-I)-S(\fn_0\fb_1^2)}\frac{q_v+1}{q_v-1}\,\prod_{v\in I-S(\fn_0\fb_1^2)}\frac{q_v+1}{q_v-1}
\\
&=
\prod_{v\in S(\fn_1\fb_1^{-1})} (\omega_v(\fn_0\fb_1^{2})-\omega_v(\fn_0\fb_1^{2})),
\end{align*} 
which is zero by $S(\fn_1\fb_1^{-1})\not=\emp$. This completes the proof. 

\end{proof}

\begin{defn} \label{N-transformDef}
 For an arithmetic function $B: \mathcal I\rightarrow \C$, we define its $\cN$-transform $\cN[B]:\mathcal I \rightarrow \C$ by the formula 
\begin{align*}
\cN[B](\fn)&=\sum_{I\subset S(\fn_1)}(-1)^{\# I}\{\prod_{v\in I\cap S_1(\fn_1)}\omega_v(\fn_0)\}\, \frac{\iota(\fn \prod_{v\in I}\fp_v^{-2})}{\iota(\fn)}\,
B(\fn\prod_{v\in I}\fp_v^{-2}).
\end{align*}
\end{defn}

\begin{lem} \label{FIND}
For $t\in \C$, let $\nr^t$ be the arithmetic function $\fn\mapsto \nr(\fn)^t$ on $\mathcal I$. 
For any ideal $\gn$, we have
\begin{align*}
\cN[\nr^t](\fn)
&=\nr(\fn)^{t}\,\{\prod_{v\in S(\fn_{1})-S_2(\fn)}(1-q_v^{-2(1+t)})\}\,\{\prod_{v\in S_2(\fn)}(1-(1-q_{v}^{-1})^{-1}\,q_v^{-2(1+t)})\}.
\end{align*}
\end{lem}
\begin{proof}
 By \eqref{iota-fm}, we have
$$
\frac{\iota(\fn\prod_{v\in I}\fp_v^{-2})}{
\iota(\fn)}
=\prod_{v\in I}q_v^{-2}\prod_{v\in I \cap S_{2}(\gn)}(1+q_{v}^{-1})^{-1}.
$$
Therefore, 
\begin{align*}
&\sum_{I\subset S(\fn_{1})}(-1)^{\# I} \{\prod_{v\in I \cap S_1(\fn_{1})}\omega_v(\fn_{0})\}\,\frac{\iota(\fn \prod_{v\in I}\fp_v^{-2})}{\iota(\fn)}\,\nr(\fn \prod_{v\in I}\fp_v^{-2})^{t}
 \\
&=\nr(\fn)^{t}\,
\sum_{I\subset S(\fn_{1})}(-1)^{\# I} \{\prod_{v\in I \cap S_2(\fn)}
\frac{q_{v}+1}{q_{v}-1}\}
\,
\prod_{v \in I \cap S_{2}(\gn)}(1+q_{v}^{-1})^{-1}
\{\prod_{v\in I}
q_v^{-2(1+t)}\} \\
&=\nr(\fn)^{t}\,
\sum_{I\subset S(\fn_{1})}(-1)^{\# I}
\prod_{v \in I \cap S_{2}(\gn)}(1-q_{v}^{-1})^{-1}\{\prod_{v\in I} q_v^{-2t}\} \\
& = \nr(\fn)^{t} \{\prod_{v \in S(\fn_{1})-S_2(\fn)}(1-q_v^{-2(1+t)})\}\, \{\prod_{v\in S_2(\fn)}(1-(1-q_{v}^{-1})^{-1}\,q_v^{-2(1+t)})\}.
\end{align*}
\end{proof}

\begin{cor} \label{Findcor}
The $\cN$-transform of the arithmetic function $\log\nr(\fn)$ on $\mathcal I$ is given by 
\begin{align*}
\cN[\log \nr](\fn)=&\prod_{v\in S(\fn_1)-S_2(\fn)}(1-q_v^{-2})\,\prod_{v\in S_2(\fn)}(1-(q_{v}^{2}-q_{v})^{-1})
\\
&\quad \times \left(\log \nr(\fn)
+ \sum_{v\in S(\fn_1)-S_2(\fn)}\frac{2\log q_{v}}{q_{v}^{2}-1}
+ \sum_{v\in S_2(\fn)} \frac{2\log q_{v}}{q_{v}^{2} -q_{v} -1}\right).
\end{align*}
\end{cor}
\begin{proof}
Take the derivative at $t=0$ of the formula in Lemma~\ref{FIND}. 
\end{proof}

For any arithmetic function $B:\Ical\rightarrow \C$, we define another function $\cN^{+}[B]$ by setting
\begin{align*}
\cN^+[B](\fn)=&\sum_{I\subset S(\gn_{1})}\{\prod_{v\in I\cap S_1(\gn_{1})}\omega_v(\gn_{0})\}\,\frac{\iota(\fn\prod_{v\in I}\fp_v^{-2})}{\iota(\fn)}B(\fn\prod_{v\in I}\fp_v^{-2})
\end{align*} 
for $\fn=\gn_{0}\fn_1^2 \in \Ical$. In a similar way to Lemma~\ref{FIND}, we have 
\begin{align}
\cN^{+}[\nr^t]
&=\nr(\fn)^{t}\,\{\prod_{v\in S(\gn_{1})-S_2(\fn)}(1+q_v^{-2(t+1)})\} \,\{\prod_{v\in S_2(\fn)}(1+(1-q_v^{-1})^{-1}q_v^{-2(t+1)})\}
 \label{cN+f}
\end{align} 
for any $t\in \R$. 

\begin{lem} \label{cN+hyouka}
Let $c>0$. Then for any sufficiently small $\e>0$, we have 
\begin{align*}
\cN^{+}[\nr^{-c+\e}](\fn) \ll_{\e} \nr(\fn)^{-\inf(c,1)+\e}, \quad \fn \in \Ical
. 
\end{align*}
\end{lem}
\begin{proof}
From $\nr(\fn)^{-c+\e}\leq \nr(\fn)^{-\inf(c,1)+\e}$, we have $\cN^{+}[\nr^{-c+\e}](\fn)\leq \cN^{+}[\nr^{-\inf(c,1)+\e}](\gn)$ obviously. Let us set $t=-\inf(c,1)+\e$ and examine the right-hand side of the formula \eqref{cN+f}. We note that $t+1=1-\inf(c,1)+\epsilon \geq \epsilon>0$. The set $P(\epsilon)=\{v\in \Sigma_\fin|\,1-q_v^{-1}<q_v^{-\epsilon}\,\}$ is a finite set. For $v\in S_2(\fn)-P(\epsilon)$, we have $(1-q_v^{-1})^{-1}\leq q_v^{\epsilon}$ and $q_v^{-2(t+1)}\leq q_v^{-2\e}$; by these, the factor $1+(1-q_v^{-1})q_v^{-2(t+1)}$ is bounded by $1+q_v^{-\epsilon}$. For $v\in S(\gn_{1})-S_2(\fn)$ or $v\in S_2(\fn)\cap P(\epsilon)$, we simply apply $q_v^{-2(t+1)}\leq q_v^{-2\epsilon}$. Thus,
\begin{align}
\cN^{+}[\nr^{t}](\fn)\leq \nr(\fn)^{t}\{\prod_{v\in S(\gn_{1})-S_2(\fn)}(1+q_v^{-2\epsilon})\}\,
\{\prod_{v\in P(\e)}(1+(1-q_v^{-1})^{-1}q_v^{-2\e})\}
\,\{\prod_{v\in S_2(\fn)-P(\e)}(1+q_v^{-\e})\}.
 \label{C-P14-2}
\end{align}
In the right-hand side, the second factor is independent of $\fn$. The first and the last factors combined are estimated as
\begin{align*}
\{\prod_{v\in S(\gn_{1})-S_2(\fn)}(1+q_v^{-2\epsilon})\}\,
\{\prod_{v\in S_2(\fn)-P(\e)}(1+q_v^{-\e})\}
&\leq \{\prod_{v\in S(\fn)}(1+q_v^{-\epsilon})\}^2
 \\
&\ll_{\e} \{\prod_{v\in S(\fn)}q_v^{\epsilon}\}^2
\leq \nr(\fn)^{2\epsilon}.
\end{align*}
Hence there exists a constant $C(\epsilon)>0$ dependent of $\e$ such that \eqref{C-P14-2} is less than $C(\epsilon)\,\nr(\fn)^{-\inf(c,1)+3\e}$ for any $\fn \in \Ical$.
\end{proof}

\subsection{The totally inert case over $\fn$}
Set $\Ical=\mathcal I_{S,\eta}$. Fixing a test function $\alpha\in \ccA_S$ for a while, we study the arithmetic functions $\AL^*:\Ical\rightarrow \C$ and $\ADL_{-}^*:\Ical\rightarrow \C$ defined by the formulas \eqref{AL*} and \eqref{Intro0}, respectively. We relate these functions to the $\cN$-transforms of arithmetic functions $\AL^w$, $\ADL^w_{\pm}$ on $\Ical$ defined in 3.1. 

We remark that an ideal $\fn\in \Ical$ satisfies the condition
\begin{align}
\eta_v(\varpi_v)=-1,  \qquad v\in S(\fn). \label{TRN}
\end{align}
This means that the quadratic extension of $F$ corresponding to $\eta$ is inert over all places dividing $\fn$. Under this condition, the quantities $w_\fn^\eta(\pi)$ and $\partial w_\fn^\eta(\pi)$ turn out to be written explicitly in terms of the arithmetic function $\omega(\fm,\fb)$.

\begin{lem}\label{value-w} 
Let $\fn\in \Ical$. Then, for any $\pi\in \Pi_{\rm{cus}}(l,\fn)$, we have $w_\fn^\eta(\pi)=0$ unless $\fn\ff_\pi^{-1}=\fb^2$ for some integral ideal $\fb$, in which case 
$$
w_\fn^\eta(\pi)=\omega(\fn,\fn\ff_\pi^{-1}). 
$$
\end{lem}
\begin{proof} Let $v\in S(\fn\ff_\pi^{-1})$. From \cite[Lemma 6.2]{SugiyamaTsuzuki}, 
$$
r(\pi_v,\eta_v)=\dfrac{1+(-1)^{k_v}}{2}\times \begin{cases} 
1, \qquad &(c(\pi_v)\geq 1), \\
\dfrac{q_v+1}{q_v-1}, \qquad&(c(\pi_v)=0).
\end{cases}
$$
Thus $r(\pi_v,\eta_v)=0$ unless $k_v=\ord_v(\fn\ff_\pi^{-1})$ is even.

\end{proof}

\begin{lem} \label{value-delw} 
Let $\fn\in \Ical$. For any $\pi \in \Pi_{\rm{cus}}(l,\fn)$, we have the following.  
\begin{itemize}
\item[(i)] If $\fn\ff_\pi^{-1}=\fb^{2}$ with an integral ideal $\fb$, then 
$$
\partial w_\fn^\eta(\pi)=\omega(\fn,\fn\ff_\pi^{-1})\,
\sum_{v\in S(\fb)} (-\log q_v)\ord_v(\fb).
$$
\item[(ii)] If $\fn\ff_\pi^{-1}=\fb^2\fp_u$ with an integral ideal $\fb$ and a place $u\in S(\fn)$, then 
$$
\partial w_\fn^\eta(\pi)
=\omega(\fn,\fn\ff_\pi^{-1})
\,\log q_u\,\begin{cases} 
\ord_{u}(\gb) + \dfrac{q_u-1}{(1+a_u q_u^{1/2})(1+a_u^{-1}q_u^{1/2})}, \qquad &(c(\pi_u)=0), \\
\ord_u(\fb)+\dfrac{1}{1+q_u^{-1}\chi_u(\varpi_u)}, \qquad &(c(\pi_u)=1), \\
\ord_u(\fb)+1, \qquad &(c(\pi_u)\geq 2). 
\end{cases}
$$
\end{itemize}
Except the above two cases (i) and (ii), we have $\partial w_\fn^\eta(\pi)=0$.  
\end{lem}

\begin{lem} \label{AL-AL*}
 For any $\fn\in \Ical$,
\begin{align*}\AL^w(\fn;\alpha)&=\sum_{\fb} \omega(\fn,\fb^{2})\,\frac{
\iota(\fn\fb^{-2})}{\iota(\fn)}
\,\AL^*(\fn\fb^{-2};\alpha),\\
\ADL^w_{-}(\fn;\alpha)&=\sum_{\fb} \omega(\fn,\fb^{2})\,\frac{
\iota(\fn\fb^{-2})}{\iota(\fn)}
\,\ADL^*(\fn\fb^{-2};\alpha),\\
\ADL_{+}^{w}(\fn;\alpha)&=\sum_{\fb} \omega(\fn,\fb^{2})\,\frac{
\iota(\fn\fb^{-2})}{\iota(\fn)}\,\log(\nr(\fn\fb^{-2})^{-1/2}\,\nr(\ff_\eta)^{-1}D_{F}^{-1})\,\AL^*(\fn\fb^{-2};\alpha),
\end{align*}
where $\fb$ runs through all the integral ideals such that $\fn\subset \fb^{2}$.
\end{lem}
\begin{proof}
This follows immediately from Lemma~\ref{value-w}. To have the last formula, we also need Proposition~\ref{ADL^w_+}. 
\end{proof}

\begin{lem} \label{AL^*-AL}
For any $\fn\in \Ical$, 
\begin{align*}\AL^{*}(\fn)&=\cN[\AL^w](\fn), \\ 
\ADL^{*}(\fn)&=\cN[\ADL^w_{-}](\fn),  \\
-\log(\sqrt{\nr(\fn)}\nr(\ff_\eta)D_{F})\,\AL^{*}(\fn)&=\cN[\ADL^w_{+}](\fn)
\end{align*}
\end{lem}
\begin{proof} By Lemma~\ref{AL-AL*}, we obtain the first formula by applying Proposition~\ref{INV-F} with $B(\fm)=\iota(\fm)\,\AL^{w}(\fm;\alpha)$ and $A(\fm)=\iota(\fm)\,\AL^{*}(\fm;\alpha)$ both defined for $\fm\in \mathcal I$. The remaining two formulas are proved in the same way.
\end{proof}

The formula \eqref{DRTF-1} in Theorem~\ref{DRTF} can be applied to an arbitrary ideal $\fm\in \Ical$. In the right-hand side of the formula, we have two terms $\tilde \WW_{\rm{u}}^\eta(l,\fm|\alpha)$ and $\WW_{\rm{hyp}}^\eta(l,\fm|\a)$, which we regared as arithmetic functions in $\fm$ for a while and consider their $\cN$-transforms $\cN[\tilde \WW_{\rm{u}}^\eta]$ and $\cN[\WW_{\rm{hyp}}^\eta]$. The following is the main result of this section. 

\begin{prop} \label{henkeiDRTF}
For any $\fn \in \Ical$, we have the identity among linear functionals in $\alpha \in \ccA_S$:
\begin{align}
{\rm{ADL}}^*(\fn) = &
\,2(-1)^{\# S+\e(\eta)}\,\cG(\eta)^{-1}\,D_F\{\cN[\tilde\WW_{\rm{u}}^\eta](\fn)
+\cN[\WW_{\rm{hyp}}^\eta](\fn)\}
 \label{henkeiDRTF-1}
\\
&+\log(\nr(\fn)^{1/2}\nr(\ff_{\eta}))\,{\rm{AL}}^*(\fn)
-\cN[{\rm{AL}}^{\partial w}](\fn).
 \notag
\end{align}
\end{prop}
\begin{proof}
We take the $\cN$-transform of both sides of the formula \eqref{DRTF-1} regarding it as an identity among arithmetic functions on $\Ical$. Then apply Lemma~\ref{AL^*-AL}.
\end{proof}

\section{An error term estimate for averaged $L$-values}
In this section we prove the first asymptotic formula of Theorem~\ref{MAIN-THM1}. 
Recall the sets $\Ical_{S,\eta}^{\pm}$, to which $\fn$ should belong. We note that, by the sign of the functional equation, $L(1/2,\pi)L(1/2,\pi\otimes \eta)=0$ if $\pi\in \Pi_{\rm{cus}}^*(l,\fn)$ unless $\fn \in \Ical_{S,\eta}^{+}$. Thus we restrict ourselves to those levels $\fn$ belonging to $\Ical_{S,\eta}^{+}$, for otherwise $\AL^*(\fn;\alpha)=0$. 

%
We have the following asymptotic result, whose proof is given in the next subsection. 
\begin{prop} \label{C-P9} Suppose ${\underline l}=\inf_{v\in \Sigma_\infty}l _v\geq 6$. For any ideal $\fm \in \Ical_{S(\fa),\eta}^{+}$, we have 
\begin{align*}
\AL^{w}(\fm;\alpha_{\fa})&=4D_F^{3/2}\,L_\fin(1,\eta)\,\nr(\fa)^{-1/2}\delta_{\square}(\fa_\eta^{-})d_1(\fa_\eta^{+})+{\mathcal O}_{\epsilon,l,\eta}(\nr(\fa)^{c+2+\epsilon}\nr(\fm)^{-c+\e})
\end{align*}
for any ideal $\fa$ primes to $\ff_\eta$, where $c=d_F^{-1}({\underline l}/2-1).$
\end{prop}

From this, we can deduce the asymptotic formula for the primitive part $\AL^*(\fn;\a_\fa)$ stated in Theorem~\ref{MAIN-THM1}. Indeed, we apply the first formula of Lemma~\ref{AL^*-AL} substituting the expression of $\AL^{w}$ given in Proposition~\ref{C-P9}. The main and the error terms are computed by Lemmas \ref{FIND} and \ref{cN+hyouka}, respectively. This completes the proof.

\subsection{The proof of Proposition~\ref{C-P9}}
For any place $v\in \Sigma_\fin$, we define a function $\Lambda_v:F_v-\{0,-1\}\rightarrow \Z$ by setting 
\begin{align*}
\Lambda_v(b)=\delta(b\in \cO_v)\{\ord_v(b(b+1))+1\}.
\end{align*}
For an $\cO$-ideal $\fb$, we set
$$
\tau^{S(\fb)}(b)=\{\prod_{v\in \Sigma_\fin-S(\fb)}\Lambda_v(b)\}\,\prod_{v\in S(\fb)}\delta(b\in \fb^{-1}\cO_v), \quad b\in F-\{0,-1\}.
$$
For an even integer $k\,(\geq 4)$ and a real valued character $\varepsilon$ of $\R^\times$, let $J^{\varepsilon}(k;b)$ ($b\in \R-\{0,-1\}$) be the integral studied in \cite[10.3]{SugiyamaTsuzuki}; they are evaluated explicitly in \cite[Lemma 10.15]{SugiyamaTsuzuki} as
\begin{align*}
J^{{\bf 1}}(k;b)&=\begin{cases} 
(1+b)^{-k/2}\,\frac{2\Gamma(k/2)^2}{\Gamma(k)}{}_2F_1(k/2,k/2;k;(b+1)^{-1}), \quad& (b(b+1)>0), \\
2\log|(b+1)/b|\,P_{k/2-1}(2b+1)-\sum_{m=1}^{[k/4]}\frac{8(k-4m+1)}{(2m-1)(k-2m)}\,P_{k/2-2m}(2b+1), \quad &(b(b+1)<0),
\end{cases}
\\
J^{{\rm{sgn}}}(k;b)&=
\begin{cases} 0 , \quad &(b(b+1)>0), \\
 2\pi i \,P_{k/2-1}(2b+1), \quad &(b(b+1)<0),
\end{cases}
\end{align*}
where $P_{n}(x)$ is the Legendre polynomial of degree $n$. 

\begin{lem} \label{ArchIntEst}
Let $k$ be an even integer greater than $2$ and $\eta_v$ a real valued character of $\R^{\times}$. Then, for any $\epsilon>0$, we have the estimation
\begin{align*}
|b(b+1)|^{\epsilon}\,|J^{\eta_v}(k;b)| \ll_{\epsilon,k} (1+|b|)^{-k/2+2\e}, \quad b\in \R-\{0,-1\} 
\end{align*}
with the implied constant depending on $k$ and $\epsilon$. 
\end{lem} 
\begin{proof} For $J^{\rm{sgn}}(k;b)$ the estimation is obvious. As for $J^{\bf 1}(k;b)$, we only have to note that the estimation ${}_2F_{1}(k/2,k/2;k;(b+1)^{-1})=O(|\log b|)$ for small $b>0$ (\cite[p.49]{MOS}) and the functional equation $J^{\bf 1}(k;b)=(-1)^{k/2}J^{\bf 1}(k;-b-1)$ $(b<-1)$ proved in \cite[Lemma 10.16]{SugiyamaTsuzuki}.
\end{proof}

Given relatively prime $\cO$-ideals $\fn$ and $\fb$ and for $\e\geq 0$, we set
$$
\fI^{\eta}_\e(l,\fn,\fb)=\sum_{b\in \fn\fb^{-1}-\{0,-1\}}\tau^{S(\fb)}(b)^{2}\,
|\nr(b(b+1))|^{\e}\prod_{v\in \Sigma_\infty} |J^{\eta_v}(l_v;b)|.
$$

\begin{prop}  \label{fI-EST} 
Suppose $\underline l \geq 6$. Let $\fb$ and $\fn$ be relatively prime ideals. For any $\e \ge 0$ and $\epsilon'>0$, we have
$$
\fI^\eta_\e(l,\fn,\fb)\ll_{\epsilon, \e', l}\nr(\fb)^{1+c+ \e'}\,\nr(\fn)^{-c+\e+\e'}$$
with the implied constant independent of $\fb$ and $\fn$.
\end{prop}
\begin{proof}
By \cite[Lemma 12.4]{SugiyamaTsuzuki} and Lemma~\ref{ArchIntEst}, we have
$$
\fI_\e^\eta(l,\fn,\fb)\ll_{\epsilon, \e',l}\nr(\fb)^{4\e'} \sum_{b\in \fn\fb^{-1}-\{0\}} \prod_{v\in \Sigma_\infty}(1+|b_v|)^{-l_v/2+\e+2\e'}
=\nr(\fb)^{4\e'}\theta(\fn\fb^{-1})
$$
for any $\e\ge 0$ and any $\e'>0$, where we regard the fractional ideal $\fn\fb^{-1}$ as a $\Z$-lattice in the Euclidean space $F_\infty=F\otimes_{\QQ} \R$ and $\theta(\Lambda)$ is constructed for $\{l_{v}-2\e-4\e'\}_{v\in \Sigma_\infty}$ in place of $l$ (see \S 9). If $\e\ge 0$
and $\e'>0$ are small enough, then we can apply the theory in \S 9 to this $\theta(\Lambda)$. The desired estimation follows if we apply Theorem~\ref{theta-est} with $\Lambda=\fn\fb^{-1}$ and $\Lambda_0=\fb^{-1}$ noting $D(\fn\fb^{-1})=\nr(\fn)\nr(\fb)^{-1}$, $D(\fb^{-1})=\nr(\fb)^{-1}$ and $r(\fb^{-1})\leq r(\cO)$.
\end{proof}

\begin{prop} \label{HYPERBOLIC-EST}
 Suppose ${\underline{l}}\geq 6$. Given $\cO$-ideals $\fn$ and $\fa=\prod_{v\in S(\fa)}\fp_v^{n_v}$ relatively prime to each other, for any $\epsilon>0$, we have 
$$
|\JJ_{\rm{hyp}}^{\eta}(l,\fn|\alpha_\fa)|\ll_{\epsilon,l, \eta}\nr(\ga)^{c+2+\e}\nr(\gn)^{-c+\e}
$$
with the implied constant independent of $\fa$ and $\fn$.
\end{prop}

\begin{proof} Let $v\in S(\fa)$ and $n\in \N_0$. By \eqref{Tcheby}, 
\begin{align}
\alpha_{\fp_v^{n}}(\nu)=\frac{z^{n+1}-z^{-(n+1)}}{z-z^{-1}}=\sum_{m=0}^{[n/2]}\alpha^{(n-2m)}_v(\nu) - \delta(n \in 2\NN_{0})
 \label{HYPERBOLIC-EST-1}
\end{align}
with $\alpha^{(m)}_v(\nu)=z^{m}+z^{-m}$, $z=q_v^{\nu/2}$. By \cite[Lemma 10.3]{SugiyamaTsuzuki}, we have 
\begin{align}
|J_v^{\eta_{v}}(b,\alpha_v^{(m)})| \ll & (1+m)^{2} \delta(|b|_v\leq q_v^{m})
q_v^{\delta(m>0)-m/2}
\{1+\Lambda_{v}(b)\},
\quad b \in F^{\times}-\{-1\}
\notag
\end{align}
with the implied constant independent of $m\ge 0$ and $v$. Hence if $n>0$, 
\begin{align*}
|J_{v}^{\eta_{v}}(b,\alpha_{\fp_v^n})|&\ll \delta(|b|_v\leq q_v^{n})\{\sum_{m=0}^{n}(1+m)^{2}q_v^{1-m/2}\} \{ 1 + \Lambda_{v}(b)\}
\\
&\leq \delta(|b|_v\leq q_v^{n})q_v\,\left(\sum_{m=0}^{\infty} (1+m)^{2} 2^{-m/2}\right)
\{ 1 + \Lambda_{v}(b)\}.
\end{align*}
Thus we have a constant $C$ independent of $v \in S(\ga)$ and $n \in \N_{0}$ such that
\begin{align}
|J_{v}^{\eta_{v}}(b,\alpha_{\fp_v^n})|\le C\, q_v^{\delta(n>0)}\,\delta(|b|_v\leq q_v^{n})\, \{ 1+\Lambda_{v}(b)\},
\quad b\in F^\times-\{0,-1\}.
\label{HYPERBOLIC-EST-2}
\end{align}
Combining (\ref{HYPERBOLIC-EST-2}) with Proposition~\ref{fI-EST} and \cite[Lemmas 10.4, 10.5 and Corollary 10.11]{SugiyamaTsuzuki}, we obtain
\begin{align*}
|\JJ_{\rm{hyp}}^{\eta}(l,\fn|\alpha_{\fa})|
\leq & \,C^{\#S(\fa)}\,\{\prod_{v\in S(\fa)}q_v^{\delta(n_v>0)}\}\,
\sum_{I \subset S(\ga)}\sum_{b \in \gn(\prod_{v \in I}\gp_{v}^{n_{v}})^{-1}\gf_{\eta}^{-1}}\tau^{S(\prod_{v \in I}\gp_v^{n_{v}})}(b)\prod_{v \in \Sigma_{\infty}}|J_{v}^{\eta_{v}}(l_{v}; b)| \\
\le & \,C^{\#S(\fa)}\,\nr(\ga)
\sum_{I\subset S(\ga)}\fI_{0}(l, \gn, \prod_{v \in I}\gp_v^{n_{v}}\gf_{\eta}) \\\ll_{\e,l}& \,C^{\#S(\fa)}\,\nr(\ga)\sum_{I \subset S(\ga)}\nr(\prod_{v \in I}\gp_v^{n_{v}}\gf_{\eta})^{1+c+\e}\nr(\gn)^{-c+\e} \\
\ll_{\e, l, \eta} & \,C^{\#S(\ga)}\nr(\ga) \times 2^{\#S(\ga)}\nr(\ga)^{1+c+\e}\nr(\gn)^{-c+\e}.
\end{align*}
By the estimate $(2\,C)^{\# S(\fa)}\ll_{\e, \eta} \nr(\fa)^{\e}$, we are done.
\end{proof}

\begin{lem} \label{TCH-UNIP}
Set $\Upsilon_v^{\eta_v}(s)=(1-\eta_v(\varpi_v)q_v^{-(1+s)/2})^{-1}(1-q_v^{(1+s)/2})^{-1}$. For $n\in \N_0$, 
\begin{align*}
\frac{1}{2\pi i}\,\int_{\s-i\infty}^{\s+i\infty} \Upsilon_{v}^{\eta_v}(s)\,\a_{\fp_v^{n}}(s)\,\d\mu_{v}(s)&=-q_v^{-n/2}
\begin{cases}
\delta(n\in 2\N_0), \quad &(\eta_v(\varpi_v)=-1), \\
n+1, \quad &(\eta_v(\varpi_v)=+1), 
\end{cases}
\\
\frac{\log q_v}{2\pi i}\,\int_{\s-i\infty}^{\s+i\infty} \frac{\Upsilon_{v}^{\eta_v}(s)\,\a_{\fp_v^{n}}(s)}{1-\eta_v(\varpi_v)q_v^{(s+1)/2}}\,\d\mu_{v}(s)&=q_v^{-n/2}\log q_v\begin{cases}
(-1)^{n}\,\left[\frac{n+1}{2}\right], \quad & (\eta_v(\varpi_v)=-1), \\
\frac{n(n+1)}{2}, \quad &(\eta_v(\varpi_v)=+1).
\end{cases}
\end{align*}
\end{lem}
\begin{proof}
The second integral is $\tilde U_v^{\eta_v}(\a_{\fp_v^{n}})$ defined by \eqref{Unipterm-2}. Then we have the second formula using \eqref{HYPERBOLIC-EST-1} and Lemma~\ref{DUNIP} by a direct computation. The first formula is confirmed in the same way by using \cite[Proposition 11.1]{SugiyamaTsuzuki}.
\end{proof}

To show Proposition~\ref{C-P9}, we apply \cite[Theorem 9.1]{SugiyamaTsuzuki} taking $S=S(\fa)$. From the first formula of Lemma~\ref{TCH-UNIP}, 
\begin{align*}
\JJ_{\rm{u}}^\eta(l,\fn|\a_\fa)&=(-1)^{\# S(\fa)} \prod_{v\in S(\fa_\eta^{-})}  q_v^{-n_v/2}\delta(n_v\in 2\N_0)\,\prod_{v\in S(\fa_\eta^{+})}q_v^{-n_v/2}(n_v+1)\\
&=(-1)^{\# S(\fa)}\nr(\fa)^{-1/2}\delta_{\square}(\fa_\eta^{-})d_1(\fa_\eta^{+}).\end{align*}
We use Proposition~\ref{HYPERBOLIC-EST} to estimate $\JJ_{\rm{hyp}}^\eta(l,\fn|\a_\fa)$, which yields the error term. This completes the proof.

\section{An error term estimate for averaged derivative of $L$-values}
 Let $\fa=\prod_{v\in S(\fa)}\fp_v^{n_v}$ and $\Ical_{S,\eta}^{\pm}$ be as in \S 1. 
In this section we prove the asymptotic formula of $\ADL_{-}^{*}(\fn;\a_\fa)$ for $\fn\in \Ical_{S(\fa),\eta}^{-}$ stated in Theorem~\ref{MAIN-THM1}. We remark that $\ADL_{-}^{*}(\fn)=0$ if $\fn\in \Ical_{S(\fa),\eta}^{+}$. Indeed, for such $\fn$, $\e(1/2,\pi)\e(1/2,\pi\otimes \eta)=+1$ for all $\pi\in \Pi_{\rm{cus}}^*(l,\fn)$, which means $\e(1/2,\pi)=-1$ and hence $L(1/2,\pi)=0$ for all $\pi$ occurring in the sum $\ADL_{-}^{*}(\fn)$.
 
Starting from the formula \eqref{henkeiDRTF-1} with $\a$ specialized to $\a_\fa$, we examine the 4 terms in the right-hand side separately. Here is the highlights in the analysis for each term. 
\begin{itemize}
\item[(i)] We compute the term $\cN[\tilde \WW_{\rm{u}}^{\eta}](\fn)$ explicitly by using Lemma~\ref{DUNIP}, Lemma~\ref{FIND} and Corollary~\ref{Findcor}, which yields the main term of the formula (modulo a part of the error term); see 6.1 for detail.
\item[(ii)] We prove 
\begin{align*}
\cN[\WW_{\rm{hyp}}^{\eta}](\fn)=\Ocal_\e(\nr(\fa)^{c+2+\e}\nr(\fn)^{-\inf(1,c)+\e})
\end{align*}
 by using the explicit formula of local terms given in \S 8; see 6.2 for detail.  
\item[(iii)] Since $\fn\in \Ical_{S(\fa),\eta}^{-}$, the term $\AL^{*}(\fn)$ vanishes by the reason of the sign of the functional equations. 
\item[(iv)] We prove 
\begin{align*}
\cN[\AL^{\partial w}](\fn)=\Ocal_{\e}\left(\nr(\fa)^{-1/2+\e}\,X(\fn)+\nr(\fa)^{c+2}\nr(\fn)^{-\inf(1,c)+\e}\right).
\end{align*} 
This part is most subtle and the term $X(\fn)$ arises from this stage; see 6.3 for detail.  
\end{itemize}
Combining these considerations, we obtain the second formula in Theorem~\ref{MAIN-THM1} immediately.

\subsection{Computation of $\cN[\tilde \WW_{\rm{u}}^{\eta}](\fn)$ }
Let us describe the procedure (i). We take $\alpha$ to be the function $\alpha_\fa$. Set $S=S(\fa)$. From \eqref{DRTF-2}, we have that $\cN[\tilde \WW_{\rm{u}}^{\eta}](\fn)$ is
the sum of the following two integrals: 
\begin{align}
&2(-1)^{\e(\eta)} \Gcal(\eta) D_{F}^{1/2}
\left(\frac{1}{2\pi i}\right)^{\#S}\,\int_{\LL_{S}(\bfc)}
\cN[\tilde{\frak{W}}_{S}(-|\bfs)](\fn) \a_{\fa}(\bs)\,\d \mu_{S}(\bfs), 
 \label{Mainterm-1}
\\
&2(-1)^{\e(\eta)} \Gcal(\eta) D_{F}^{1/2}
\left(\frac{1}{2\pi i}\right)^{\# S}\,\int_{\LL_{S}(\bfc)}\cN[D\,\tilde{\frak{W}}_{S}(-|\bfs)](\fn)
\,\a_{\fa}(\bfs)d\mu_{S}(\bfs),
 \label{Mainterm-2}
\end{align}
where $\tilde{\frak{W}}_{S}(-|\bfs)$ is the quantity \eqref{DRTF-3} viewed as an arithmetic function in $\fn$ and $D$ is an arithmetic function given by $D(\fn)=(-1)^{\e(\eta)}\tilde\eta(\fn)\delta(\fn=\cO)i^{\tilde{l}}$. By the formula \eqref{DRTF-3}, 
\begin{align*}
\cN[\tilde{\frak{W}}_{S}(-|\bfs)](\fn)
&=\pi^{\e(\eta)}\Upsilon_{S}^\eta(\bfs)\,L(1,\eta)\Biggl\{2^{-1}\cN[\log \nr](\fn)
\\&+\left(\log(D_F\nr(\ff_\eta))+\frac{L'}{L}(1,\eta)+{\frak C}(l)+\sum_{v\in S}\frac{\log q_v}{1-\eta_v(\varpi_v)q_v^{(s_v+1)/2}}\right)\,\cN[1](\fn)\Biggr\}.
\end{align*}
By Lemma~\ref{FIND} and Corollary~\ref{Findcor}, we have formulas of $\cN[\log \nr](\fn)$ and of $\cN[1](\fn)$; substituting these, and by using Lemma~\ref{TCH-UNIP}, we complete the evaluation of the integral \eqref{Mainterm-1}.

The evaluation of the integral \eqref{Mainterm-2} is similar; instead of $\cN[\log \nr]$ and $\cN[1]$, we need $\cN[D\log \nr]$ and $\cN[D]$, which are much easier. Indeed, in the expression 
\begin{align*}
\cN[D\log\nr](\fn)&= (-1)^{\e(\eta)}i^{\tilde{l}} \sum_{I\subset S(\fa)}(-1)^{\# I} \{\prod_{v\in I\cap S_1(\fa)}\omega_{v}(\gn_{0})\}\,\frac{\iota(\fn\prod_{v\in I}\fp_v^{-2})}{\iota(\fn)}\\
&\quad \times \tilde\eta(\fn\prod_{v\in I}\fp_v^{-2})\,
\delta(\fn\prod_{v\in I}\fp_v^{-2}=\cO)\log\nr(\fn\prod_{v\in I}\fp_v^{-2}),
\end{align*}
the sum survives only if $\fn=\prod_{v\in S(\fn)}\fp_v^{2}$ and $I=S(\fn)$. A similar remark is applied to $\cN[D](\fn)$. Hence, 
\begin{align*}
\cN[D\log\nr](\fn)&=\delta(S(\fn)=S_2(\fn))\{\prod_{v\in S(\gn)}\frac{q_{v}+1}{q_{v}-1}\}\,(-1)^{\e(\eta)}i^{\tilde{l}}\frac{(-1)^{\# S(\fn)}}{\iota(\fn)}\log\nr(\cO)=0, \\
\cN[D](\fn)&=\delta(S(\fn)=S_2(\fn))\{\prod_{v \in S(\gn)}\frac{q_{v}+1}{q_{v}-1}\}\,(-1)^{\e(\eta)}i^{\tilde{l}}\frac{(-1)^{\# S(\fn)}}{\iota(\fn)}.
\end{align*}
Since $\iota(\fn)^{-1}=\Ocal(\nr(\fn)^{-2})$, the integral \eqref{Mainterm-2} amounts at most to $\nr(\fn)^{-2+\e}\nr(\ga)^{-1/2+\e}$.

\subsection{Estimation of the term $\cN[\WW_{\rm{hyp}}^{\eta}](\fn)$}
Let us describe the procedure (ii). We need the following estimation, which we prove in 8.4.

\begin{prop} \label{WWhypglobal}
For any small $\e>0$, 
\begin{align*}
|\WW_{\rm{hyp}}^{\eta}(l,\fn;\a_{\fa})|\ll_{\e, l, \eta} \nr(\fa)^{c+2+\e}\nr(\fn)^{-c+\e}, \quad \fn\in \Ical^{-}_{S(\fa),\eta},
\end{align*}
where the implied constant is independent of the ideal $\fa$. 
\end{prop}

From this proposition and Lemma~\ref{cN+hyouka}, 
\begin{align*}
|\cN[\WW_{\rm{hyp}}^{\eta}](\fn)|&\leq \cN^{+}[|\WW_{\rm{hyp}}^{\eta}|](\fn)
\ll_{\e} \nr(\fa)^{c+2+\e}\,\cN^{+}[\nr^{-c+\e}](\fn)\ll_{\e}\nr(\fa)^{c+2+\e}\,\nr(\fn)^{-\inf(c,1)+2\e}.  
\end{align*}

\subsection{Estimation of the term $\cN[\AL^{\partial w}](\fn)$}
Let us describe the procedure (iv).

\begin{lem} \label{ALpw-est}
Let $\alpha\in \ccA_S$. Then for any $\fn \in \Ical_{S(\fa),\eta}^{-}$, we have the inequality
\begin{align*}
|\AL^{\partial w} (\fn;\alpha)| &\leq 
\sum_{(\fb,u)} D(\fn;\fb,u)\,\frac{\iota(\fn\fb^{-2}\fp_u^{-1})}{\iota(\fn)}\,|\AL^*(\fn\fb^{-2}\fp_u^{-1};\a)|,
\end{align*}
where $(\fb,u)$ runs through all the pairs of an integral ideal $\fb$ and a place $u$ such that $\fn\subset \fb^2\fp_u$. For such $(\fb,u)$, we set
\begin{align*}
D(\fn;\fb,u)&=\omega(\fn,\fb^{2}\fp_u)\,(\log q_u)\,\left(\ord_u(\fb)+\frac{q_u^{1/2}+1}{q_u^{1/2}-1}\right).
\end{align*}
\end{lem}
\begin{proof}
By Lemma~\ref{value-delw}, the $\pi$-summand of $\AL^{\partial w}(\fn; \a)$ vanishes unless the conductor $\ff_\pi$ satisfies either (i) $\fn\ff_\pi^{-1}=\fb^2$ with some $\fn \subset \fb$, or (ii) $\fn\ff_\pi^{-1}=\fb^2\fp_u$ with some $\fn\subset \fb$ and $u\in S(\fn)$. In the case (i), the $\pi$-summand vanishes. Indeed, $\ff_\pi$ belongs to $\Ical_{S(\fa),\eta}^{-}$ and thus $L(1/2,\pi)L(1/2,\pi\otimes \eta)=0$ by the functional equation.  In the second case (ii), by the Ramanujan bound $|a_v|=1$ and the obvious relation $|\chi_v(\varpi_v)|=1$, we have \begin{align*}
|\partial w_\fn^\eta(\pi)|&\leq \omega(\fn,\fb^2\fp_u)\,\log q_u\,
\begin{cases}
\ord_{u}(\fb) +\frac{q_u-1}{(1-q_u^{1/2})^2}, \qquad &(c(\pi_u)=0), \\
\ord_u(\fb)+\frac{1}{1-q_u^{-1}}, \qquad &(c(\pi_u)=1), \\
\ord_u(\fb)+1, \qquad &(c(\pi_u)\geq 2)
\end{cases}
\\
&\leq \omega(\fn,\fb^2\fp_u)\,(\log q_u)\,\left(\dfrac{q_u^{1/2}+1}{q_u^{1/2}-1}+\ord_v(\fb)\right)=D(\fn;\fb,u).
\end{align*}
Here, we used $\frac{1}{1-q_u^{-1}}<\frac{q_u-1}{(1-q_u^{1/2})^2}=\frac{q_u^{1/2}+1}{q_u^{1/2}-1}$ to have the second inequality.
\end{proof}

\begin{lem} \label{C-L12}
For any small $\e\in (0,1)$, we have
\begin{align}
&\sum_{(\fb,u)} \nr(\fb^2\fp_u)^{\epsilon} \frac{\iota(\fn\fb^{-2}\fp_u^{-1})}{\iota(\fn)}\,\nr(\fn\fb^{-2}\fp_u^{-1})^{-\inf(c,1) +\epsilon}\ll_{\epsilon} \nr(\fn)^{-\inf(c, 1)+2\epsilon},
\label{C-L12-2}
\\ 
&\sum_{(\fb,u)} \nr(\fb)^{\epsilon}\left(\frac{q_u+1}{q_u-1}\right)^2\,(\log q_u) \frac{\iota(\fn\fb^{-2}\fp_u^{-1})}{\iota(\fn)}\ll_{\e} X(\fn),
\label{C-L12-3}
\end{align}
where $(\fb,u)$ runs through the same range as in Lemma~\ref{ALpw-est}.
\end{lem}

\begin{proof}
Let us show the second estimate. By the inequality $\iota(\fn\fb^{-2}\fp_u^{-1})/\iota(\fn)\leq \nr(\fb^{-2}\fp_u^{-1})$, 
\begin{align*}
\sum_{(\fb,u)} \nr(\fb)^{\epsilon}\left(\frac{q_u+1}{q_u-1}\right)^2\,(\log q_u) \frac{\iota(\fn\fb^{-2}\fp_u^{-1})}{\iota(\fn)}
&\leq \sum_{(\fb,u)} \nr(\fb)^{-2+\e}\,\left(\frac{q_u+1}{q_u-1}\right)^2\,\frac{\log q_u}{q_u}
\\
&\leq \{\sum_{\fb\subset \cO} \nr(\fb)^{-2+\e}\}\,\{\sum_{u\in S(\fn)} \left(\frac{q_u+1}{q_u-1}\right)^2\,\frac{\log q_u}{q_u}\}
\\
&=\zeta_{F,\fin}(2-\e)\,\{
\sum_{u\in S(\fn)} \frac{\log q_u}{q_u}+\sum_{u\in S(\fn)} \frac{4\log q_u}{(q_u-1)^2}\}.
\end{align*}
Since $\zeta_{F,\fin}(2-\e)$
is convergent, we are done. 
\end{proof}

\begin{prop}\label{C-P13.1}
For any sufficiently small $\e>0$,
\begin{align*}
|\AL^{\partial w}(\fn;\alpha_\fa)| \ll_{\epsilon, l, \eta} \nr(\fa)^{-1/2}d_1(\fa_\eta^{+})\delta_\square(\fa_\eta^{-})\,X(\fn)
+\nr(\fa)^{c+2+\e}\nr(\gn)^{-\inf(c,1)+\e}, \quad \fn \in \Ical_{S(\fa),\eta}^{-}.
\end{align*}
\end{prop}
\begin{proof}
Let $\epsilon>0$. From $\frac{x+1}{x-1}\ll_{\epsilon} x^{\epsilon}$ for $x\geq 2$, we have
$$\omega(\fn,\fb^{2}\gp_{u}) \leq \left(\prod_{v\in S(\fb)}\frac{q_v+1}{q_v-1}\right) \frac{q_u+1}{q_u-1} \ll _{\epsilon} \nr(\fb)^\epsilon\, \frac{q_u+1}{q_u-1}$$
with the implied constant independent of $\fn$ and $(\fb, u)$. 
By this,
\begin{align*}
 D(\fn;\fb,u)\ll_{\epsilon}\nr(\fb)^{\epsilon}(\log q_u) \left(\frac{q_u+1}{q_u-1}\right)^2
\end{align*}
with the implied constant independent of $\fn$ and $(\fb,u)$. Using these estimates, we have the desired bound by \eqref{MAIN-THM1-1} and Lemmas~\ref{ALpw-est} and \ref{C-L12}.
\end{proof}

\begin{prop} \label{C-P14}
For any sufficiently small $\epsilon>0$,
\begin{align}
&\left|\cN[\AL^{\partial w}](\fn;\a_\fa)\right| 
 \ll_{\epsilon, l, \eta} \nr(\fa)^{-1/2}d_1(\fa_\eta^{+})\delta_{\square}(\fa_\eta^{-})\,X(\fn)+\nr(\fa)^{c+2+\e}\,\nr(\fn)^{-\inf(1,c)+\epsilon}, \quad \fn\in \Ical_{S(\fa),\eta}^{-}. 
 \label{CP14-1}
\end{align}
\end{prop}
\begin{proof} Let $\fn=\fn_1^2\fn_0$. From Proposition~\ref{C-P13.1}, we have 
\begin{align*}
|\cN[\AL^{\partial w}](\fn;\a_\fa)|\ll_\e \nr(\fa)^{-1/2}d_1(\fa_\eta^{+})\delta_{\square}(\fa_\eta^{-}) \cN^{+}[X](\fn)+\nr(\fa)^{c+2+\e}\,\cN^+[\nr^{-\inf(1,c)+\e}](\fn)
\end{align*}
for all $\fn$. Since $X(\fm)\leq X(\fn)$ if $\fn\subset \fm\subset \cO$, we have\begin{align*}
\cN^{+}[X](\fn) &\leq X(\fn)\,\cN^{+}[1](\fn)
\\
&=X(\fn)\,\{\prod_{v\in S(\fn_1)-S_2(\fn)}(1+q_v^{-2})\} \,\{\prod_{v\in S_2(\fn)}(1+(1-q_v^{-1})^{-1}q_v^{-2})\}
\\
&\leq X(\fn)\{\prod_{v\in \Sigma_\fin}(1+q_v^{-2})\} \,\{\prod_{v\in \Sigma_\fin}(1+(1-q_v^{-1})^{-1}q_v^{-2})\}
\ll X(\fn),
\end{align*}
because the Euler products occurring are convergent. 

From the proof of Lemma~\ref{cN+hyouka}, we have $\cN^{+}[\nr^{-\inf(c,1)+\e}](\fn)\ll_{\e} \nr(\fn)^{-\inf(c,1)+3\e}$. Consequently, for any sufficiently small $\e\in (0,1)$, we obtain the estimate
\begin{align*}
|\cN[\AL^{\partial w}](\fn;\a_\fa)|\ll_{\e} \nr(\fa)^{-1/2}d_1(\fa_\eta^{+})\delta_{\square}(\fa_\eta^{-}) X(\fn) +\nr(\fa)^{c+2+\e}\nr(\fn)^{-\inf(1,c)+3\e}
\end{align*}
with the implied constant independent of $\fn$ and $\fa$. Since $\e$ is arbitrary, we are done. 
\end{proof}

\section{An estimation of number of cusp forms}


Recall that we set $c=d_F^{-1}({\underline l}/2-1)$. Suppose that for each ideal $\fa\subset \cO$, we are given a set $\Jcal_{\fa}$ consisting of ideals prime to $\ff_\eta\fa$ in such a way that $\Jcal_{\fa}\subset \Jcal_{\fa'}$ for any $\fa\subset \fa'$, and a family of real numbers $\{\omega_{\fn}(\pi)|\,\pi \in \Pi_{\rm{cus}}^*(l,\fn)\}$ for each $\fn \in \Jcal_{\fa}$ which satisfies the following estimate
for any $\e>0$: 
\begin{align}
\left|\sum_{\pi \in \Pi_{\rm{cus}}^*(l,\fn)}\omega_{\fn}(\pi)
\prod_{v\in S(\fa)}X_{n_v}(\lambda_{v}(\pi))-\prod_{v\in S(\fa)}\mu_{v, \eta_v}(X_{n_v})\right|\ll_{\e, l, \eta}
\frac{\nr(\fa)^{-1/2+\e}}{\log \nr(\fn)}+\nr(\fa)^{c+2+\e}\nr(\fn)^{-\inf(c,1)+\e},
 \label{omega-asymp}
\end{align}
with the implied constant independent of $\fa$ and $\fn\in \Jcal_{\fa}$. Moreover we impose the non-negativity condition:
\begin{align}
\omega_{\fn}(\pi)\geq 0 \qquad {\text{for all $\pi \in \Pi_{\rm{cus}}^{*}(l,\fn)$ and $\fn\in \Jcal_{\fa}$. }}
\label{non-NEG}
\end{align}

Let $\fq$ be a prime ideal relatively prime to $\gf_{\eta}$.
In what follows, we abuse the symbol $\fq$ to denote the corresponding place $v_\fq$ of $F$; for example, we write $\nu_{\fq}(\pi)$, $\lambda_{\fq}(\pi)$ in place of $\nu_{v_\fq}(\pi)$, $\lambda_{v_\fq}(\pi)$ etc. Let $S=\{v_1,\dots,v_r\}$ be a finite subset of $\Sigma_\fin-S(\ff_\eta\fq)$ and set $\fa_{S}=\prod_{v\in S} \fp_v$. Let ${\bf J}=\{J_j\}_{j=1}^{r}$ a family of closed subintervals of
$(-2,2)$. For each $J_j$, we choose an open interval $J_{j}'$ such that ${\overline {J_j'}} \subset J_j^\circ$ and $C^\infty$-function $\chi_{j}:\R\rightarrow [0,\infty)$ with the following properties:
\begin{itemize}
\item $\chi_j(x)\not=0$ for all $x\in J_j'$. 
\item ${\rm{supp}}(\chi_j)\subset J_j$. 
\item $\int_{-2}^{2}\chi_j(x) \d \mu_{v,\eta_v}(x)=1$, where
\begin{align*}
\d\mu_{v,\eta_v}(x)=
\begin{cases}
\dfrac{q_v-1}{(q_v^{1/2}+q_v^{-1/2}-x)^2}\,d\mu^{\rm ST}(x),
\quad &(\eta_v(\varpi_v)=+1), \\
\dfrac{q_v+1}{(q_v^{1/2}+q_v^{-1/2})^2-x^2}\,\d\mu^{\rm{ST}}(x),
\quad &(\eta_v(\varpi_v)=-1).
\end{cases}
\end{align*}
\end{itemize}
Here $d\mu^{\rm ST}(x) = (2\pi)^{-1}\sqrt{4-x^{2}}dx$.
Fixing such a family of functions $\{\chi_{j}\}$, we set
$$
\Omega_{\fn}(\pi)=\omega_{\fn}(\pi)\,\prod_{j=1}^{r}\chi_{j}(\lambda_{v_j}(\pi)), \quad \pi \in \Pi_{\rm{cus}}^*(l,\fn), \,\fn \in \Jcal_{\fq\fa_S}.
$$
\begin{lem} \label{Omeganoseishitsu}
For any sufficiently small $\e>0$, there exists $N_{\e, S, l}>0$ such that 
\begin{align}
\left|\sum_{\pi \in \Pi_{\rm{cus}}^*(l,\fn)}\Omega_{\fn}(\pi)X_n(\lambda_\fq(\pi))
-\mu_{\fq, \eta_{\gq}}(X_n)\right|\ll_{\e, l, \eta, S, {\bf J}} \frac{n+1}{(\log\nr(\fn))^3}+\frac{\nr(\fq^n)^{-1/2+\e}}{\log \nr(\fn)}+\nr(\fq^n)^{2+c+\e}\nr(\fn)^{-\inf(c, 1)+\e}
 \label{Omega-asymp}
\end{align}
for $n\in \N_0$ and $\fn\in \Jcal_{\fq\fa_S}$ with $\nr(\gn) > N_{\e, S, l}$.
Here the implied constant is independent of $n$ and $\fn$. Moreover,
\begin{align}
\Omega_{\fn}(\pi)\geq 0 \qquad {\text{for all $\pi \in \Pi_{\rm{cus}}^{*}(l,\fn)$ and $\fn\in \Jcal_{\fq\fa_S}$. }}
\label{Non-NEG}
\end{align}
\end{lem}
\begin{proof}
Given an integer $M>1$, define $\chi_{j}^{M}(x)=\sum_{n=0}^{M}\hat\chi_{j}(n)\,X_n(x)$ for $x\in [-2,2]$
with $\hat{\chi}_{j}(n) = \int_{-2}^{2}\chi_{j}(x)X_{n}(x)d\mu^{\rm ST}(x)$
and set 
$$
\chi(\bx)=\prod_{j=1}^{r}\chi_{j}(x_j), \quad \chi^{M}(\bx)=\prod_{j=1}^{r}\chi_{j}^M(x_j)  
$$
for $\bx=\{x_j\}_{j=1}^{r}$ in the product space $[-2,2]^{r}$. Let $\fn\in \Jcal_{\fq\fa_S}$. By the triangle inequality, the left-hand side of \eqref{Omega-asymp} is no greater than the sum of the following three terms :
\begin{align}
&\biggl|\sum_{\pi \in \Pi_{\rm{cus}}^*(l,\fn)}\omega_{\fn}(\pi)X_n(\lambda_\fq(\pi))\{\chi(\lambda_S(\pi))-\chi^M(\lambda_S(\pi))\}\biggr|,
 \label{Omeganoseishitsu1}
\\
&\biggl|\sum_{\pi \in \Pi_{\rm{cus}}^*(l,\fn)}\omega_{\fn}(\pi)X_n(\lambda_\fq(\pi))\chi^M(\lambda_S(\pi))-\mu_{\fq, \eta_{\gq}}(X_n)\,\mu_{S, \eta}(\chi^M)\biggr|,
 \label{Omeganoseishitsu2}
\\
&|\{\mu_{S, \eta}(\chi^M)-\mu_{S, \eta}(\chi)\}\,\mu_{\fq, \eta_{\gq}}(X_n)|,
 \label{Omeganoseishitsu3}
\end{align}
where $\lambda_{S}(\pi) = (\lambda_{v}(\pi))_{v \in S}$ and $\mu_{S, \eta}=\otimes_{v\in S}\,\mu_{v, \eta_v}$. Note $\mu_{S, \eta}(\chi)=1$. We shall estimate these quantities. Since $|\hat\chi_j(n)|\ll_{\chi_{j}} n^{-5}$ for any $n>0$ by integration-by-parts and by $\max_{[-2,2]}|X_n|\ll n+1$, we have
$$
|\chi_j^M(x)|\leq \sum_{n\leq M} |\hat\chi_j(n)|\,|X_n(x)|\ll_{\chi_j} \sum_{n\leq M}n^{-4}\leq \zeta(4)
$$
and 
$$\max_{x\in [-2,2]}|\chi_j(x)-\chi_j^M(x)|\leq \sum_{n>M}|\hat\chi_j(n)|\max_{[-2, 2]}|X_n|\ll_{\chi_j} \sum_{n>M}n^{-4}\ll M^{-3}.
$$
By these, 
\begin{align}
\max_{[-2,2]^r}|\chi(\bx)-\chi^M(\bx)|&\leq \max_{[-2,2]^r}\biggl(\sum_{j=1}^{r}|\prod_{h=1}^{j-1}\chi_h^{M}(x_h)|\,|\chi_j(x_j)-\chi_j^M(x_j)|\biggr)
\ll_{S, \chi} M^{-3}.
 \label{Omeganoseishitsu4}
\end{align}
From \eqref{omega-asymp} for $\fa=\cO$, noting $\fn\in \Jcal_{\gq \fa_{S}}\subset \Jcal_{\cO}$, we have the estimate $|\sum_{\pi \in \Pi_{\rm{cus}}^{*}(l,\fn)}\omega_{\fn}(\pi)-1|\ll_{\e, l, \eta}(\log\nr(\fn))^{-1}+\nr(\fn)^{-\inf(c, 1)+\e}$. Hence \eqref{Omeganoseishitsu1} is majorized by 
\begin{align*}
\{\max_{[-2,2]}|X_n|\}\,\{\max_{\bx\in [-2,2]^r}|\chi(\bx)-\chi^M(\bx)|\}\,\sum_{\pi \in \Pi_{\rm{cus}}^{*}(l,\fn)}\omega_{\fn}(\pi)\ll_{\e, l, \eta, S, \chi} (n+1)M^{-3}(1+\nr(\fn)^{-\inf(c, 1)+\e}).
\end{align*}
By \eqref{Omeganoseishitsu4}, the quantity \eqref{Omeganoseishitsu3} is majorized by $\mu_{\fq, \eta_{\gq}}(X_n)M^{-3}$, which amounts at most to $(n+1)M^{-3}$. Let us estimate \eqref{Omeganoseishitsu2}. By expanding the product, $\chi^{M}(\bx)$ is expressed as a sum of the terms $\prod_{j=1}^{r}\hat \chi_j(n_j)\prod_{j=1}^{r}X_{n_j}(x_j)$ over all ${\bf n}=(n_j)_{j=1}^{M}\in [0,M]^{r}$. Hence by using \eqref{omega-asymp}, we can majorize \eqref{Omeganoseishitsu2} from above by 
\begin{align*}
&\sum_{{\bf n}} \left|\sum_{\pi \in \Pi_{\rm{cus}}^*(l,\fn)}\omega_{\fn}(\pi)
\prod_{j}X_{n_j}(\lambda_{v_j}(\pi))-\mu_{S, \eta}(\prod_{j}X_{n_j})\right|
\\
&\ll_{\e, l, \eta, S, \chi}
\frac{\nr(\fa_{S}^{M}\gq^{n})^{-1/2+\e}}{\log \nr(\fn)}+\nr(\fa_{S}^{M}\gq^{n})^{2+c+\e}\nr(\fn)^{-\inf(c, 1)+\e}.
\end{align*}
Combining the estimations made so far, we have that the left-hand side of \eqref{Omega-asymp} is majorized by 
\begin{align}
(n+1)M^{-3}(1+\nr(\fn)^{-\inf(c,1)+\e})+
\frac{\nr(\fa_{S}^{M}\gq^{n})^{-1/2+\e}}{\log \nr(\fn)}+\nr(\fa_S^{M}\gq^{n})^{2+c+\e}\nr(\fn)^{-\inf(c,1)+\e}.
 \label{Omeganoseishitsu5}
\end{align}
Now take
$$
M=\left[\frac{\e}{2+c+\e}\frac{\log\nr(\fn)}{\log\nr(\fa_S)}\right].
$$
Then $\nr(\fa_S)^{M(2+c+\e)}\leq \nr(\fn)^{\e}$, and also $\nr(\fa_S)^{M(-1/2+\e)}\leq 1$ evidently. By these, \eqref{Omeganoseishitsu5} is majorized by 
\begin{align*}
&(n+1)(\log\nr(\fn))^{-3}\log\nr(\fa_S)^3(1+\nr(\fn)^{-\inf(c, 1)+\e})
+\frac{\nr(\fq^n)^{-1/2+\e}}{\log \nr(\fn)}+
\nr(\fq^{n})^{2+c+\e}\nr(\fn)^{-\inf(c, 1)+2\e}
\\
\ll_{\e,S}& \,(n+1)(\log\nr(\fn))^{-3}+\frac{\nr(\fq^n)^{-1/2+\e}}{\log \nr(\fn)}
+\nr(\fq^{n})^{2+c+\e} \nr(\fn)^{-\inf(c, 1)+2\e}.
\end{align*} 
\end{proof}

\begin{lem} \label{shortIntEST}
Let $I\subset [-2,2]$ be an open interval disjoint from the set $\{\lambda_\fq(\pi)|\,\pi\in \Pi_{\rm{cus}}^{*}(l,\fn), \
\Omega_{\fn}(\pi) \neq 0 \}$. Then for any small $\epsilon>0$, there exists a constant $N_{\e, l, \eta, S, \gq}>0$ such that for any ideal $\gn \in \Jcal_{\fq\fa_S}$ with $\nr(\gn)>N_{\e, l, \eta, S, \gq}$,
$$
\mu_{\fq, \eta_{\gq}}(I)\ll_{\epsilon, l, \eta, S, {\bf J}} \nr(\fq)^{\epsilon}\,(\log \nr(\fn))^{-1+\epsilon}  
$$
holds with the implied constant independent of $I$, $\fn$ and $\fq$.
\end{lem}
\begin{proof} The proof of \cite[Proposition 5.1 and Lemma 5.2]{Royer} goes through as it is with a small modification. We reproduce the argument for convenience. 

Let $\Delta>0$ be a parameter to be specified below and $K$ a closed subinterval of $I$ such that
\begin{itemize}
\item[(i)] $\mu_{\fq, \eta_{\fq}}(I-K)\leq \Delta$.
\end{itemize}
Depending on $\Delta$ and $K$, we choose a $C^\infty$-function $\tF$ on $\R$ such that 
\begin{itemize}
\item[(ii)] ${\rm{supp}}(\tF)\subset \bar I$,
\item[(iii)] $\tF(x)=1$ if $x\in K$ and $0\leq \tF(x)\leq 1$ for $x\in \R$,  
\item[(iv)] $|\tF^{(k)}(x)|\ll_{k} \Delta^{-k}$ for $k\in \N_0$.
\end{itemize} 
Since $I$ does not contain the relevant $\lambda_\fq(\pi)$'s, from (ii) we have $\Omega_{\fn}(\pi)\tF(\lambda_\fq(\pi))=0$ for all $\pi \in \Pi_{\rm{cus}}^*(l,\fn)$. Using this, from (i), and (iii), we have the inequalities
\begin{align} 
\mu_{\fq, \eta_{\fq}}(I)\leq \mu_{\fq, \eta_{\fq}}(K)+\Delta
&\leq \int_{-2}^{2} \tF\,\d\mu_{\fq, \eta_{\fq}}+\Delta
 \notag
\\
&\leq\biggl|\sum_{\pi\in  \Pi_{\rm{cus}}^*(l,\fn)}\Omega_{\fn}(\pi)\,\tF(\lambda_\fq(\pi))-\int_{-2}^{2} \tF\,\d\mu_{\fq, \eta_{\fq}}\biggr|+\Delta.
 \label{shortIntEST-1}
\end{align}
If we set $\tF_{M}(x)=\sum_{n=0}^{M} \hat \tF(n)\,X_n(x)$, then the first term of \eqref{shortIntEST-1} is bounded by the sum of the following three terms
\begin{align}
&\biggl(\sum_{\pi \in \Pi_{\rm{cus}}^*(l,\gn)}|\Omega_{\fn}(\pi)|\biggr)\cdot \max_{[-2,2]}|\tF-\tF_M|, \label{shortIntEST-2}
\\
&\int_{-2}^{2}\max_{[-2,2]}|\tF-\tF_M|\,\d\mu_{\fq, \eta_{\fq}}, \label{shortIntEST-3}
\\
&\biggl|\sum_{\pi\in  \Pi_{\rm{cus}}^*(l,\fn)}\Omega_{\fn}(\pi)\,\tF_M(\lambda_\fq(\pi))-\int_{-2}^{2} \tF_M\,\d\mu_{\fq, \eta_{\fq}}\biggr|. \label{shortIntEST-4}
\end{align}
We remark that by the non-negativity of $\Omega_{\fn}(\pi)$, the absolute value in \eqref{shortIntEST-2} can be deleted. Then by the estimate $|\hat \tF(n)|\ll_{k} n^{-k}\Delta^{-k}$ which follows from (iv) by integration by parts, and by $\max_{[-2,2]}|X_n|\ll n+1$, we have 
$$\max_{[-2,2]}|\tF-\tF_M| \leq \sum_{n>M}|\hat \tF(n)|\,\max_{[-2,2]}|X_n|
\ll_{k} \sum_{n>M}n^{-k}\Delta^{-k} n \ll M^{2-k}\Delta^{-k}
$$
with $k \ge 3$. From \eqref{Omega-asymp} applied with $n=0$, noting $\mu_{\fq, \eta_{\fq}}(X_0)=1$, we have the estimate $|\sum_{\pi\in \Pi_{\rm{cus}}^*(l,\fn)}\Omega_{\fn}(\pi)-1|\ll_{\e, l, \eta, S, {\bf J}}(\log \nr(\fn))^{-1}+\nr(\fn)^{-\inf(c, 1)+\e}$. Hence the sum of \eqref{shortIntEST-2} and \eqref{shortIntEST-3} is majorized by
$$
\Delta^{-k}M^{2-k}\,(1+(\log \nr(\fn))^{-1}+\nr(\fn)^{-\inf(c,1)+\e})\ll {\Delta^{-k}M^{2-k}} 
$$
with the implied constant independent of $\Delta$, $M$, $\fq$ and $\fn$. By \eqref{Omega-asymp} and by $|\hat \tF(n)|\ll 1$, the term \eqref{shortIntEST-4} is majorazed by 
\begin{align*}
&\sum_{n=0}^{M}|\hat \tF(n)|\,\left|\sum_{\pi\in  \Pi_{\rm{cus}}^*(l,\fn)}\Omega_{\fn}(\pi)\,X_n(\lambda_\fq(\pi))-\mu_{\fq, \eta_{\fq}}(X_n)\right|
\\
& \ll_{\e, l, \eta, S, {\bf J}} \sum_{n=0}^{M}\left(\frac{n+1}{(\log\nr(\fn))^3}+\frac{\nr(\fq^n)^{-1/2+\e}}{\log \nr(\fn)}+\nr(\fq^{n})^{2+c+\e}\,\nr(\fn)^{-\inf(c, 1)+\e}\right)
\\
&\ll_\e \frac{M^2}{(\log\nr(\fn))^3}+\frac{1}{\log \nr(\fn)}+\nr(\fq)^{c'M}\,\nr(\fn)^{-\inf(c, 1)+\e}, 
\end{align*}
where $c'=2+c+\e$. Putting all relevant estimations together, we obtain 
\begin{align*}
\mu_{\fq, \eta_{\fq}}(I)\ll_{k,\e, l, \eta, S, {\bf J}} \Delta+{\Delta^{-k}M^{2-k}}+\frac{1}{\log \nr(\fn)}+\frac{M^2}{(\log\nr(\fn))^3}+\nr(\fq)^{c'M}\nr(\fn)^{-\inf(c, 1)+\e}
\end{align*}
with the implied constant independent of $I$, $\Delta$, $M$, $\fq$ and $\fn$. By setting $M=\left[\frac{\inf(c, 1)}{2c'}\frac{\log \nr(\fn)}{\log \nr(\fq)}\right]$, this yields the estimate
$$
\mu_{\fq, \eta_{\fq}}(I)\ll_{k, \e, l, \eta, S, {\bf J}} \Delta +\Delta^{-k}(\log\nr(\fq))^{k-2}(\log \nr(\fn))^{2-k}+(\log \nr(\fn))^{-1}+\nr(\fn)^{-c_2/2+\e}. 
$$
Let $\epsilon>0$ and we let $\Delta$ vary so that it satisfies $\Delta^{-k}(\log\nr(\fn))^{2-k} \asymp_{k} (\log\nr(\fn))^{-1+\epsilon}$, or equivalently 
$$
\Delta \asymp_{k} (\log \nr(\fn))^{-1+(3-\epsilon)/k}.
$$
By taking $k=[3/\epsilon]+1$, we have $(\log \nr(\fn))^{-1+\epsilon/2}\ll_{\e} \Delta \ll_{\e} (\log\nr(\fn))^{-1+\epsilon}$. Hence, 
\begin{align*}
\mu_{\fq, \eta_{\fq}}(I) &\ll_{\epsilon, l, \eta, S, {\bf J}} (\log\nr(\fn))^{-1+\epsilon}+(\log\nr(\fn))^{-1+\e}(\log\nr(\fq))^{k-2}+(\log\nr(\fn))^{-1}+\nr(\fn)^{-\inf(c, 1)/2+\e}
\\
&\ll_{\epsilon} \nr(\fq)^{\epsilon}\,(\log \nr(\fn))^{-1+\epsilon}.
\end{align*}
This completes the proof.
\end{proof}

\begin{lem} \label{LoweboundCuspforms}
Given $\epsilon>0$, there exists a positive number $N_{\epsilon, l, \eta, S, \gq, {\bf J}}$ such that for any ideal $\fn \in \Jcal_{\gq\ga_S}$ with $\nr(\fn)>N_{\epsilon, l, \eta, S, \gq, {\bf J}}$, the inequality 
\begin{align*}
\#\{\l_{\gq}(\pi)|\, \pi \in \Pi_{\rm{cus}}^*(l,\fn), \,\Omega_{\fn}(\pi)\not=0\,\}\geq \nr(\fq)^{-\epsilon}(\log \nr(\fn))^{1-\epsilon}
\end{align*}
holds.
\end{lem}
\begin{proof}
 \cite[Lemma 5.3]{Royer}. 
\end{proof}

\subsection{} 
Let $\Gamma={\rm{Aut}}(\C/\Q)$. We let the group $\Gamma$ act on the set of even weights by the rule ${}^{\sigma}l=(l_{\sigma^{-1}\circ v})_{v\in \Sigma_\infty}$ for $l=(l_v)_{v\in \Sigma_\infty}$ and $\sigma\in \Gamma$, regarding $\Sigma_\infty={\rm{Hom}}(F,\C)$. Let $\Q(l)$ be the fixed field of ${\rm{Stab}}_\Gamma(l)$, which is a finite extension of $\Q$. 
From \cite{Shimura} (see \cite{Raghuram-Tanabe} also), the Satake parameter $A_v(\pi)$ belongs to $\GL(2, \bar \Q)$ for any $v\in \Sigma_\fin-S(\fn)$ and the set $\Pi_{\rm{cus}}(l,\fn)$ has a natural action of the Galois group ${\rm{Gal}}(\bar \Q/\Q(l))$ in such a way that $({}^\sigma \pi)_v\cong \pi_{\sigma^{-1}\circ v}$ for all $v\in \Sigma_\infty$ and 
\begin{align}
q_v^{1/2}\,A_v({}^\sigma \pi)=\sigma(q_v^{1/2}\,A_v(\pi))\quad{\text{for all $v\in \Sigma_\fin-S(\fn)$}}.
\label{GaloisAction}
\end{align}
The field of rationality of $\pi \in \Pi_{\rm{cus}}(l,\fn)$, to be denoted by $\Q(\pi)$, is defined as the fixed field of the group
$$\{\sigma\in {\rm{Gal}}(\bar \Q/\Q(l))|\, ^{\sigma}\pi=\pi\,\}.$$  
From \eqref{GaloisAction}, by the strong multiplicity one theorem for $\GL(2)$, we have $$\Q(\pi)=\Q(l)(q_v^{1/2}\l_v(\pi)|\,v\in \Sigma_\fin-S(\fn)).$$ 


\begin{prop} \label{HeckeFieldDegreeEST}
Suppose $l$ is a parallel weight, i.e., there exists $k\in 2\N$ such that $l_v=k$ for all $v\in \Sigma_\infty$. Let $S$ be a finite subset of $\Sigma_\fin-S(\ff_\eta)$ and ${\bf J}=\{J_v\}_{v\in S}$ a family of closed subintervals of $(-2,2)$. Given a sufficiently small $\epsilon>0$
and a prime ideal $\gq$ prime to $S \cup S(\gf_{\eta})$,
there exists a positive integer $N_{\epsilon, l, \eta, S, \gq, {\bf J}}$ such that for any $\fn \in \Jcal_{\gq\ga_S}$ with $\nr(\fn)>N_{\epsilon, l, \eta, S, \gq, {\bf J}}$, there exists $\pi\in \Pi_{\rm{cus}}^*(l,\fn)$ such that $\omega_{\fn}(\pi)\not=0$, $\lambda_v(\pi)\in J_v$ for all $v\in S$, and
$$
[\Q(\pi):\Q] \geq \sqrt{\max\left\{
\frac{(1-\epsilon)\log\log\nr(\fn)}{\log(16\sqrt{\nr(\fq)})}-2\epsilon, 0
\right\} }.
$$
\end{prop}
\begin{proof}
By choosing $C^\infty$-functions $\{\chi_v\}$ as above, we construct the weight function $\Omega_{\fn}(\pi)$. We follow the proof of \cite[Proposition 7.3]{Royer}. Let $d(\fn,\Omega)$ denote the maximal degree of algebraic numbers $\lambda_{\fq}(\pi)$ ($\pi \in \Pi_{\rm{cus}}^*(l,\fn),\,\Omega_{\fn}(\pi)\not=0)$. Then,
\begin{align*}
d(\fn,\Omega)&\leq \max\{ \,[ \Q(\pi) : \Q] \,|\,\pi\in \Pi_{\rm{cus}}^{*}(l,\fn),\,\Omega_\fn(\pi)\not=0\,\}
\\
&\leq \max\{ \,[ \Q(\pi) : \Q] \,|\,\pi\in \Pi_{\rm{cus}}^{*}(l,\fn),\,\omega_\fn(\pi)\not=0,\,\lambda_v(\pi)\in J_v\,(\forall v\in S)\,\}.
\end{align*}
Let ${\mathcal E}(M,d)$ denote the set of algebraic integers which, together with its conjugates, have the absolute values at most $M$ and the absolute degrees at most $d$. From the parallel weight assumption, the Heck eigenvalues $\nr(\fq)^{1/2}\lambda_\fq(\pi)$ are known to be algebraic integers (\cite[Proposition 2.2]{Shimura}). Since $\sigma(\nr(\fq)^{1/2}\lambda_\fq(\pi))=\nr(\fq)^{1/2}\lambda_\fq({}^\sigma\pi)$ from \eqref{GaloisAction}, by the Ramanujan bound, we have $\nr(\fq)^{1/2}\lambda_\fq(\pi)\in {\mathcal E}(2\nr(\fq)^{1/2},d(\fn,\Omega))$. Then the cardinality of the set $\{\nr(\fq)^{1/2}\l_{\gq}(\pi)|\, \pi \in \Pi_{\rm{cus}}^*(l,\fn), \,\Omega_{\fn,\eta}(\pi)\not=0\,\}$ is bounded from above by $\#{\mathcal E}(2\nr(\fq)^{1/2}, d(\fn,\Omega))$ which in turn is no greater than $(16\nr(\fq)^{1/2})^{d(\fn,\Omega)^2}$ by \cite[Lemma 6.2]{Royer}. Combining this with the lower bound provided by Lemma~\ref{LoweboundCuspforms}, we have
$$
\nr(\fq)^{-\e}(\log\nr(\fn))^{1-\e}\leq (16\nr(\fq)^{1/2})^{d(\fn,\Omega)^2}.   
$$
By taking logarithm, we are done. 
\end{proof}

\smallskip
\noindent
{\bf Remark} : The parallel weight assumption can be removed if the integrality of the Hecke eigenvalues $q_v^{1/2}\lambda_v(\pi)$ for all $v\in \Sigma_\fin-S(\ff_\pi)$ is known in a broader generality.

\subsection{The proof of Theorem~\ref{MAIN-THM2}}
Theorem~\ref{MAIN-THM1} means the numbers
\begin{align*}
\omega_{\fn}(\pi)=\frac{C_l}{4D_F^{3/2}L_\fin(1,\eta)\nu(\gn)} \frac{1}{\nr(\fn)}\frac{L(1/2,\pi)L(1/2,\pi\otimes\eta)}{L^{S_\pi}(1,\pi;\Ad)}, \quad \pi \in \Pi_{\rm{cus}}^{*}(l,\fn),\,\fn\in \Ical_{S\cup S(\fq),\eta}^{+}
\end{align*}
satisfy our first assumption \eqref{omega-asymp}. The second assumption \eqref{non-NEG} follows from \cite{Jacquet-Chen}. Thus Theorem~\ref{MAIN-THM2} is a corollary of Proposition~\ref{HeckeFieldDegreeEST} with this particular $\{\omega_{\fn}(\pi)\}$. 

\subsection{The proof of Theorem~\ref{MAIN-THM2.5}}
For any $M>1$, let $\Ical_{S\cup S(\gq), \eta}^{-}[M]$ be the set of $\fn\in \Ical_{S\cup S(\fq),\eta}^{-}$ such that $\sum_{v\in S(\gn)}\frac{\log q_{v}}{q_{v}}\leq M$. Theorem~\ref{MAIN-THM1} means 
\begin{align*}
\omega_{\fn}(\pi)=\frac{C_l}{4D_F^{3/2}L_\fin(1,\eta)\,\nu(\fn)\log \sqrt{\nr(\fn)}} \frac{1}{\nr(\fn)}\frac{L(1/2,\pi)L'(1/2,\pi\otimes\eta)}{L^{S_\pi}(1,\pi;\Ad)}, \quad \pi \in \Pi_{\rm{cus}}^{*}(l,\fn), \, \fn \in \Ical_{S\cup S(\fq),\eta}^{-}[M]
\end{align*}
satisfy our first assumption \eqref{omega-asymp}. By our non-negativity assumption \eqref{DLnonneg}, the second assumption \eqref{non-NEG} is also available. Thus Theorem~\ref{MAIN-THM2.5} follows from Proposition~\ref{HeckeFieldDegreeEST}.

\smallskip
\noindent
{\bf Remark} : In the parallel weight two case (i.e., $l_v=2$ for all $v\in \Sigma_\infty$) with totally imaginary condition on $\eta$,
the assumption \eqref{DLnonneg} follows from \cite[Theorem 6.1]{SWZhan3} due to the non-negativity of the Neron-Tate height pairing. Similar results may be expected in the parallel higher weight case (\cite{SWZhan}).

\section{Computations of local terms}
Let $\a=\otimes_{v\in S} \a_v$ be a decomposable element of $\ccA_S$. We examine the term $\WW_{\rm hyp}^{\eta}(l, \gn|\a)$ appearing in the formula \eqref{DRTF-1}, which is given in Lemma~\ref{WWhyp}. Recall that the function $\Psi_l^{(0)}(\gn|\a,g)$ in adele points $g=\{g_v\}$ is a product of functions $\Psi_v(g_v)$ on local groups $\GL(2, F_{v})$ such that $\Psi_v(g_v)=\Psi_{v}^{(0)}(l_{v}; g_v)$ for $v\in \Sigma_\infty$, $\Psi_v(g_v)=\hat \Psi_{v}^{(0)}(\a_v; g_v)$ for $v\in S$, and $\Psi_v(g_v)=\Phi_{\gn, v}^{(0)}(g_v)$ for $v\in \Sigma_\fin-S$ (see \cite[\S 5]{SugiyamaTsuzuki}). From Lemma~\ref{WWhyp}, by exchanging the order of integrals, we have the first equality of the formula
\begin{align}
\WW_{\rm{hyp}}^{\eta}(l,\fn|\a)&=\sum_{b\in F-\{0,-1\}} \int_{\A^\times} \hat\Psi_l^{(0)}(\gn|\a,
\delta_{b}
\left[\begin{smallmatrix} t & 0 \\ 0 & 1 \end{smallmatrix}\right]\left[\begin{smallmatrix}1 & x_{\eta} \\
0 & 1\end{smallmatrix}\right])\eta(t x_{\eta}^{*}) \log |t|_{\AA} d^{\times}t 
 \label{WWhypglobal-series}
\\
&=\sum_{b\in F-\{0,-1\}} \sum_{w\in \Sigma_F} \{\prod_{v\in \Sigma_F-\{w\}}J_v(b)\}\,W_w(b),
 \notag 
\end{align}
where
\begin{align*}
J_v(b)&=
\int_{F_{v}^{\times}}\Psi_{v}( \delta_{b}
\left[\begin{smallmatrix}t_{v} & 0 \\ 0 & 1 \end{smallmatrix}\right]\left[\begin{smallmatrix}1 & x_{\eta, v} \\ 0 & 1\end{smallmatrix}\right])\eta_{v}(t_{v} x_{\eta, v}^{*}) d^{\times}t_{v}, 
\\
W_w(b)&=
 \int_{F_{w}^{\times}}\Psi_{w}(\delta_{b}
\left[\begin{smallmatrix}t_{w} & 0 \\ 0 & 1 \end{smallmatrix}\right]\left[\begin{smallmatrix}1 & x_{\eta, w} \\ 0 & 1\end{smallmatrix}\right])\eta_w(t_{w} x_{\eta, w}^{*}) \log |t_{w}|_{w} d^{\times}t_{w} 
\end{align*}
for $b\in F_v-\{0,-1\}$. The second equality of \eqref{WWhypglobal-series} is justified by $\sum_{b}\sum_{w} \{\prod_{v\not=w}|J_v(b)|\}|W_w(b)|<\infty$, which results from the analysis to be made in 8.4. The integrals $J_v(b)$ are studied and their explicit evaluations are obtained in \cite[\S 10]{SugiyamaTsuzuki}. In what follows, we examine the integral $W_w(b)$ separating cases $w\in S$, $w\in \Sigma_\fin-S$ and $w\in \Sigma_\infty$. 


\subsection{}
Let $v\in S$. Then the integral $W_v(b)$ depends on the test function $\a_v\in \ccA_v$ and the character $\eta_v$ of $F_v^\times$; we write $W_v^{\eta_v}(b;\a_v)$ in place of $W_v(b)$ in this subsection. We have 
$$W_{v}^{\eta_{v}}(b, \a_{v}) = \frac{1}{2\pi i}\int_{L_{v}(c)}\{ \int_{F_{v}^{\times}}
\Psi_{v}^{(0)}(s_{v}; \delta_{b}
\left[\begin{smallmatrix} t & 0 \\ 0 & 1 \end{smallmatrix}\right])
\eta_{v}(t) \log |t|_{v}d^{\times}t \} \a_{v}(s_{v}) d\mu_{v}(s_{v}).
$$

\begin{lem} \label{DorbitalIntHeckeFnt}
Let $v\in S$. Let $\alpha_v^{(m)}(s_v)=q_{v}^{ms_{v}/2}+q_{v}^{-ms_{v}/2}$ with $m\in \N_{0}$. Then, for any $b\in F_v-\{0,-1\}$, 
$$
W_v^{\eta_v}(b;\alpha_v^{(m)})=\tilde{I}^{+}_v(m;\,b)+ \eta_{v}(\varpi_{v})\{
(\log q_{v}) I^{+}_{v}(m;\,\varpi_v^{-1}(b+1)) -\tilde{I}_{v}^{+}(m; \varpi_{v}^{-1}(b+1)) \}
$$
with
$I_{v}^{+}(m; -)$ defined in \cite[Lemma 10.2]{SugiyamaTsuzuki}
and
\begin{align*}
\tilde{I}_v^{+}(m;\,b)&=\vol(\go_{v}^{\times})(\log q_{v})\,2^{\delta(m=0)}\,\biggl(
-q_v^{-m/2}\,\tilde{\delta}_m^{\eta_v}(b)
\\
&\qquad +\sum_{l=\sup(0, 1-{\ord}_v(b))}^{m-1}\{(m-l-1)q_v^{1-m/2}-(m-l+1)q_v^{-m/2}\}\tilde{\delta}^{\eta_v}_l(b)\biggr),
\end{align*}
where we set
$$
\tilde{\delta}^{\eta_v}_n(b)=\delta(|b|_v < q_v^{n})\, \eta_{v}(\varpi_{v}^{n})\eta_{v}(b)(-n-\ord_{v}(b))
$$
for $n \in \NN$ and
$$\tilde{\delta}^{\eta_v}_0(b)= \delta(|b|_{v}<1)
\begin{cases} -2^{-1} \ord_{v}(b)(\ord_{v}(b)+1), \quad (\eta_v(\varpi_v)=1), \\  4^{-1}(\eta_{v}(b)-1) + 2^{-1}\ord_{v}(b) \eta_{v}(b), \quad (\eta_v(\varpi_v)=-1).
\end{cases}
$$
When $m=0$,
$$W_{v}^{\eta_{v}}(b; \a_{v}^{(0)}) = -2\vol(\go_{v}^{\times})(\log q_{v})
(\tilde{\delta}_{0}^{\eta_{v}}(b)
+\eta_{v}(\varpi_{v})\delta_{0}^{\eta_{v}}(\varpi_{v}^{-1}(b+1))
-\eta_{v}(\varpi_{v})\tilde{\delta}_{0}^{\eta_{v}}(\varpi_{v}^{-1}(b+1))$$
with $\delta_{0}^{\eta_{v}}$ defined in \cite[Lemma 10.2]{SugiyamaTsuzuki}.
\end{lem}
\begin{proof}
This is proved in a similar way to \cite[Lemma 10.2]{SugiyamaTsuzuki}.
We decompose the integral into the sum $W_{v}(b; \a_{v}^{(m)}) = \tilde{I}_{v}^{+}(m; b) + \tilde{I}_{v}^{-}(m; b)$, where
$\tilde{I}_{v}^{+}(m; b) = \int_{t\in F_{v}^{\times}, |t| \le 1}\hat{\Phi}_{vm}
(\delta_{b} \left[\begin{smallmatrix} t & 0 \\ 0 & 1 \end{smallmatrix}\right])\eta_{v}(t)\log |t|_{v}d^{\times}t$ with $\hat\Phi_{vm}(g_v)$ the integral computed in \cite[Lemma 10.1]{SugiyamaTsuzuki}. We consider the case $m>0$. By \cite[Lemma 10.1]{SugiyamaTsuzuki}, 
\begin{align*}
\tilde{I}_{v}^{+}(m; b) = & \int_{|t| \le 1, \sup(1, |t|_{v}^{-1}|b|_{v})=q_{v}^{m}}
(-q_{v}^{-m/2})\eta_{v}(t)\log |t|_{v}d^{\times}t \\
& + \sum_{l=0}^{m-1}
\int_{|t| \le 1, \sup(1, |t|_{v}^{-1}|b|_{v})=q_{v}^{l}}
\{ (m-l-1)q_{v}^{1-m/2}-(m-l+1)q_{v}^{-m/2} \}
\eta_{v}(t)\log |t|_{v}d^{\times}t.
\end{align*}
We have the following three equalities:
\begin{itemize}
\item If $l=0$ and $\eta_{v}(\varpi_{v})=1$,
\begin{align*}
\int_{|t| \le 1, \sup(1, |t|_{v}^{-1}|b|_{v})=q_{v}^{l}}
\eta_{v}(t)\log |t|_{v}d^{\times}t
= \delta( |b|_{v} < 1)\vol(\go_{v}^{\times})\log q_{v} \frac{-\ord_{v}(b)(\ord_{v}(b) +1)}{2}.
\end{align*}

\item If $l=0$ and $\eta_{v}(\varpi_{v})=-1$,
\begin{align*}
\int_{|t| \le 1, \sup(1, |t|_{v}^{-1}|b|_{v})=q_{v}^{l}}
\eta_{v}(t)\log |t|_{v}d^{\times}t
= \delta( |b|_{v} < 1)\vol(\go_{v}^{\times})\log q_{v}
(\frac{\eta_{v}(b)-1}{4} + \frac{\ord_{v}(b)\eta_{v}(b)}{2}).
\end{align*}

\item If $l>0$,
\begin{align*}
\int_{|t| \le 1, \sup(1, |t|_{v}^{-1}|b|_{v})=q_{v}^{l}}
\eta_{v}(t)\log |t|_{v}d^{\times}t
&= \delta( |b| < q_{v}^{l})\int_{|t|_{v} = q_{v}^{-l}|b|_{v}} \eta_{v}(t)\log |t|_{v}d^{\times}t
\\
&=-\delta(|b|_v\leq q_v^{l})\vol(\cO_v^\times)(\log q_v)\eta_v(\varpi_v^lb)(l+\ord_v(b)). 
\end{align*}
\end{itemize}

Furthermore, $\tilde{I}_{v}^{-}(m; b)$ is transformed into
\begin{align*}
\tilde{I}_{v}^{-}(m; b) = & \int_{|t|_{v}>1}\hat{\Phi}_{vm}(\delta_{b}\left[\begin{smallmatrix} t & 0 \\ 0 & 1 \end{smallmatrix}\right]) \eta_{v}(t)\log |t|_{v}d^{\times}t \\
= & \int_{|y|_{v}< 1}\hat{\Phi}_{vm}(\delta_{b}\left[\begin{smallmatrix} \varpi_{v}^{-1} y^{-1} & 0 \\ 0 & 1 \end{smallmatrix}\right])
\eta_{v}(\varpi_{v}^{-1}y^{-1})\log |\varpi_{v}^{-1} y^{-1}|_{v}d^{\times}t \\
= & \eta_{v}(\varpi_{v}^{-1})\int_{|y|_{v}\le 1}\hat{\Phi}_{vm}
(\delta_{b}\left[\begin{smallmatrix} \varpi_{v}^{-1} y^{-1} & 0 \\ 0 & 1 \end{smallmatrix}\right])
\eta_{v}(y)(\log q_{v} - \log |y|_{v})d^{\times}y \\
= & \eta_{v}(\varpi_{v})\{ (\log q_{v})I_{v}^{+}(m; \varpi_{v}^{-1}(b+1)) - \tilde{I}_{v}^{+}(m; \varpi_{v}^{-1}(b+1)) \}.
\end{align*}
From the results above, we have the lemma for $m>0$. The case $m=0$ is similar.
\end{proof}

\begin{lem}\label{esti of W_gq}
For $m \in \NN$,
$$|W_{v}^{\eta_{v}}(b; \a_{v}^{(m)})| \ll(\log q_{v})\delta(|b|_{v}\le q_{v}^{m-1})
q_{v}^{1-m/2}m(2m+\ord_{v}(b(b+1)))^{2}, \qquad b \in F_{v}^{\times}-\{-1\}.$$
When $m=0$,
$$|W_{v}^{\eta_{v}}(b; \a_{v}^{(0)})| \ll
(\log q_{v})\delta(|b|_{v}\le1)(\ord_{v}(b(b+1))+1)^{2}, \qquad b \in F_{v}^{\times}-\{-1\}.$$
Here the implied constants independent of $v$, $m$ and $b$.
Moreover, for $n \in \NN_{0}$,
$$|W_{v}^{\eta_{v}}(b; \a_{\gp_{v}^{n}})| \ll (\log q_{v}) q_{v}\delta(|b|_{v}\le q_{v}^{n})\{\ord_{v}(b(b+1))+2n+1)\}^{2}, \qquad b \in F_{v}^{\times}-\{-1\}$$
with the implied constant independent of $v$, $n$ and $b$.
\end{lem}
\begin{proof}
Noting (\ref{HYPERBOLIC-EST-1}), by the first and second estimates in the lemma,
the last estimate is given in the same way as in the proof of Proposition \ref{HYPERBOLIC-EST}.
We only prove the first estimate.
Suppose $m\ge 1$.
By \cite[Lemma 10.2]{SugiyamaTsuzuki}, $I_{v}^{+}(m, \varpi_{v}^{-1}(b+1))$ is estimated as
\begin{align*}
|I_{v}^{+}(m, \varpi_{v}^{-1}(b+1))| \ll \delta(|b|_{v} \le q_{v}^{m-1})(m+1)^{2}q_{v}^{1-m/2}.
\end{align*}
Next we examine $\tilde{I}_{v}^{+}(m; b)$.
From the definition of $\tilde\delta_m^{\eta_v}$ (in Lemma~\ref{DorbitalIntHeckeFnt}), we have $|\tilde{\delta}_{0}^{\eta_{v}}(b)|\le \delta(|b|_{v}<1) 2^{-1}(\ord_{v}(b)+1)^{2}.$
By using this, 
\begin{align*}
& \sum_{l=\sup(0, 1-\ord_{v}(b))}^{m-1}(m-l-1)q_{v}^{1-m/2}
|\tilde{\delta}_{l}^{\eta_{v}}(b)| \\
\le & \delta(m \ge1, |b|_{v}\le q_{v}^{m-2})q_{v}^{1-m/2}\{ \sum_{l=1}^{m-1}(m-l-1)
|\tilde{\delta}_{l}^{\eta_{v}}(b)| +(m-1)|\tilde{\delta}_{0}^{\eta_{v}}(b)| \} \\
\le & \delta(m \ge1, |b|_{v}\le q_{v}^{m-2})q_{v}^{1-m/2}\{ \sum_{l=1}^{m-1}(m-l-1)
(l+\ord_{v}(b)) +(m-1)|\tilde{\delta}_{0}^{\eta_{v}}(b)| \} \\
= & \delta(m \ge1,  |b|_{v}\le q_{v}^{m-2})q_{v}^{1-m/2}
(m-1)\{ 6^{-1}(m-2)m + 2^{-1}(m-2)\ord_{v}(b) +|\tilde{\delta}_{0}^{\eta_{v}}(b)| \}\\
\ll & \delta(m \ge 2, |b|_{v}\le q_{v}^{m-2})q_{v}^{1-m/2}m( m^{2} + m\ord_{v}(b) +(\ord_{v}(b)+1)^{2}) \\
\ll & \delta(m \ge 2, |b|_{v}\le q_{v}^{m-2} )q_{v}^{1-m/2}m(m +\ord_{v}(b))^{2}
\end{align*}
Similarly, 
\begin{align*}
\sum_{l=\sup(0, 1-\ord_{v}(b))}^{m-1}(m-l+1)q_{v}^{-m/2}
|\tilde{\delta}_{l}^{\eta_{v}}(b)|
\ll \delta(m \ge 1, |b|_{v}\le q_{v}^{m-2} )q_{v}^{-m/2}m(m+\ord_{v}(b)+1)^{2}.
\end{align*}
Hence, we obtain
\begin{align*}
|\tilde{I}_{v}^{+}(m; b)| \ll
(\log q_{v})
\delta(|b|_{v}\le q_{v}^{m-1})q_{v}^{1-m/2}m(m+\ord_{v}(b))^{2}, \qquad m \in \NN, \ b \in F_{v}^{\times}-\{-1\}.
\end{align*}
Furthermore,
\begin{align*}
& |\tilde{I}_{v}^{+}(m; \varpi_{v}^{-1}(b+1))| \\
\ll & (\log q_{v})
\delta(|b+1|_{v}\le q_{v}^{m-1})q_{v}^{1-m/2}m(m+\ord_{v}(b+1))^{2},
\qquad m \in \NN, \ b \in F_{v}^{\times}-\{-1\}.
\end{align*}
As a consequence, we have the lemma.
\end{proof}

\subsection{}
Let $v\in \Sigma_{\fin}-S$. There are three cases to be considered: $v\in \Sigma_\fin-S(\gn\gf_{\eta})$, $v\in S(\gn)$ and $v\in S(\gf_{\eta})$. 

\begin{lem}
\label{esti of W_rest}
Let $v\in \Sigma_\fin-(S\cup S(\gn\gf_{\eta}))$. For $b \in F_{v}^{\times}-\{-1\}$, we have
$$W_{v}^{\eta_{v}}(b) = \int_{F_{v}^{\times}}\Phi_{v, 0}^{(0)}
(\delta_{b}[\begin{smallmatrix} t & 0 \\ 0 & 1 \end{smallmatrix}])\eta_{v}(t)\log |t|_{v} \ d^{\times}t = \vol(\go_{v}^{\times})(\log q_{v}) \tilde{\Lambda}_{v}^{\eta_{v}}(b),$$
where
$$\tilde{\Lambda}_{v}^{\eta_{v}}(b)
= \delta(|b|_{v}\le 1)
\begin{cases}
\tilde{\delta}_{0}^{\eta_{v}}(b), \quad (|b|_{v}<1), \\
-\tilde{\delta}_{0}^{\eta_{v}}(b+1), \quad (|b+1|_{v}<1), \\
0, \quad (|b|_{v}=|b+1|_{v}=1).
\end{cases}$$
In particular, $|W_{v}^{\eta_{v}}(b)| \ll (\log q_{v})\delta(|b(b+1)|_{v}<1)(\ord_{v}(b(b+1))+1)^{2}, \ b \in F_{v}^{\times}-\{-1\}.$
\end{lem}
\begin{proof}
It follows immediately from the following computation:
\begin{align*}
\int_{F_{v}^{\times}}\Phi_{v, 0}^{(0)}
(\delta_{b}[\begin{smallmatrix} t & 0 \\ 0 & 1 \end{smallmatrix}])\eta_{v}(t)\log |t|_{v} \ d^{\times}t
= & \int_{|b|_{v}\le |t|_{v}<1} \eta_{v}(t)\log |t|_{v} d^{\times}t
+ \int_{|b+1|_{v}^{-1}\ge |t|_{v}>1} \eta_{v}(t)\log |t|_{v} d^{\times}t \\
= & \vol(\cO_v^\times)\,(\log q_v)\,\{\tilde{\delta}_{0}^{\eta_{v}}(b)-\tilde{\delta}_{0}^{\eta_{v}}(b+1)\}.
\end{align*}
\end{proof}

\begin{lem}\label{esti of W_n}
Let $v \in S(\gn)$. If $\eta_{v}(\varpi_{v})=1$, we have
\begin{align*}
W_{v}^{\eta_{v}}(b) = & \int_{F_{v}^{\times}}\Phi_{v, \gn}^{(0)}
(\delta_{b}[\begin{smallmatrix} t & 0 \\ 0 & 1 \end{smallmatrix}])\eta_{v}(t)\log |t|_{v} \ d^{\times}t \\
= & \vol(\go_{v}^{\times})(-\log q_{v})
\delta(b \in \gn\go_{v}) \ 2^{-1}(\ord_{v}(b)+\ord_{v}(\gn))(\ord_{v}(b)-\ord_{v}(\gn)+1).
\end{align*}
If $\eta_{v}(\varpi_{v})=-1$, then
\begin{align*}
& W_{v}^{\eta_{v}}(b) \\
= & \vol(\go_{v}^{\times})(-\log q_{v})
\delta(b \in \gn\go_{v})
[ 2^{-1}\{\ord_{v}(\gn)\eta_{v}(\varpi_{v}^{\ord_{v}(\gn)})+\ord_{v}(b)\eta_{v}(b)\} +4^{-1}\{\eta_{v}(b)-\eta_{v}(\varpi_{v}^{\ord_{v}(\gn)})\}].
\end{align*}
In particular,
\begin{align*}
|W_v^{\eta_v}(b)|\leq \delta(b\in \fn\cO_v)(\log q_v)(\ord_{v}(b)+\ord_{v}(\gn)+1)^{2}, \qquad b\in F_v^\times-\{-1\}.
\end{align*}
\end{lem}
\begin{proof}
It follows immediately from the following computation:
\begin{align*}
& \int_{F_{v}^{\times}}\Phi_{v, \gn}^{(0)}
(\delta_{b}[\begin{smallmatrix} t & 0 \\ 0 & 1 \end{smallmatrix}])\eta_{v}(t)\log |t|_{v} \ d^{\times}t
=  \int_{|b|_{v}\le |t|_{v}<1} \delta(t \in \gn\go_{v})\eta_{v}(t)\log |t|_{v} d^{\times}t \\
= & \delta(b\in \gn\go_{v})\sum_{n=\ord_{v}(\gn)}^{\ord_{v}(b)}\int_{\go_{v}^{\times}}
\eta_{v}(\varpi_{v}^{n}u)\log |\varpi_{v}^{n}u|_{v}d^{\times}u
= \delta(b\in \gn\go_{v})\vol(\go_{v}^{\times})(-\log q_{v})
\sum_{n=\ord_{v}(\gn)}^{\ord_{v}(b)}
\eta_{v}(\varpi_{v}^{n})n.
\end{align*}
\end{proof}

\begin{lem}
\label{esti of W_eta}
 Let $v \in S(\gf_{\eta})$ and put $f=f(\eta_{v}) \in \NN$. 
For $b \in F_{v}^{\times}-\{-1\}$, 
\begin{align*}
W_{v}^{\eta_{v}}(b) = & \delta(b \in \gp_{v}^{-f})
\eta_{v}(-1)(1-q_{v}^{-1})^{-1}q_{v}^{-f-d_{v}/2}(\log q_{v}) \times
[ -f+ \\
& \eta_{v}(b(b+1))
\{ \delta( b \in \gp_{v})(-f-\ord_{v}(b))
+\delta( b \in \go_{v}^{\times})(-f+\ord_{v}(b+1))
+ \delta(b \notin \go_{v})(-f)q_{v}^{\ord_{v}(b)}\}].
\end{align*}
In particular,
\begin{align*}
|W_v^{\eta_v}(b)|\leq 6(\log q_v)q_v^{-f}\delta(|b|_v\leq q_v^{f})\{f+\delta(|b|_v\leq 1)\ord_v(b(b+1))\}, \quad b\in F_v^\times
-\{-1\}.
\end{align*}
\end{lem}
\begin{proof}
We have the expression
$W_{v}^{\eta_{v}}(b)= \delta(b \in \gp_{v}^{-f})(W_{v, 1}^{\eta_{v}}(b)
+W_{v, 2}^{\eta_{v}}(b))$
with
$$W_{v, 1}^{\eta_{v}}(b)= \int_{\substack{-t \in \varpi_{v}^{f}U_{v}(f) \\
|t|_{v}||b+1|_{v}\le 1}}\eta_{v}(t\varpi_{v}^{-f})
\log |t|_{v} d^{\times}t= \eta_{v}(-1)(-f \log q_{v})q_{v}^{-f-d_{v}/2}(1-q_{v}^{-1})^{-1}
$$
and
$$W_{v, 2}^{\eta_{v}}(b) = \int_{
\substack{-t \in F_{v}^{\times}-\varpi_{v}^{f}U_{v}(f) \\
|1+ t\varpi_{v}^{-f}|_{v}|b+t\varpi_{v}^{-f}(b+1)|_{v}\le |t|_{v}}}\eta_{v}(t\varpi_{v}^{-f})
\log |t|_{v} d^{\times}t.
$$
The integration domain of $W_{v,2}^{\eta_v}(b)$ is a disjoint union of the sets $D_{l}(b)$ ($l\in \Z)$ defined in \cite[10.2]{SugiyamaTsuzuki}. By \cite[Lemmas 10.6, 10.7 and 10.8]{SugiyamaTsuzuki},\begin{align*}
& W_{v, 2}^{\eta_{v}}(b)= \sum_{l \in \ZZ}(-l \log q_{v})\int_{D_{l}(b)}\eta_{v}(t\varpi_{v}^{-f})d^{\times}t \\
= & \delta( |b|_{v}<1=|b+1|_{v})\{(-f+\ord_{v}(b+1)-\ord_{v}(b))\log q_{v}\}
\eta_{v}\left(\frac{-b}{b+1}\right)(1-q_{v}^{-1})^{-1}q_{v}^{-f-d_{v}/2} \\
& + \delta(|b|_{v}=|b+1_{v}|\ge 1)
(-f \log q_{v})\eta_{v}\left(\frac{-b}{b+1}\right)(1-q_{v}^{-1})^{-1}q_{v}^{-f+\ord_{v}(b)-d_{v}/2}\\
& + \delta( |b+1|_{v}<1=|b|_{v})\{(-f+\ord_{v}(b+1)-\ord_{v}(b))\log q_{v}\}
\eta_{v}\left(\frac{-b}{b+1}\right)(1-q_{v}^{-1})^{-1}q_{v}^{-f-d_{v}/2} \\
= & \eta_{v}\left(\frac{-b}{b+1}\right)(1-q_{v}^{-1})^{-1}q_{v}^{-f-d_{v}/2}(\log q_{v})
\{\delta(|b|_{v} <1=|b+1|_{v})(-f-\ord_{v}(b)) \\
& +\delta(|b|_{v}=|b+1|_{v}\ge 1)(-f)q_{v}^{\ord_{v}(b)}
+ \delta(|b+1|_{v}<1=|b|_{v})(-f+\ord_{v}(b+1))\} \\
= & \eta_{v}\left(\frac{-b}{b+1}\right)(1-q_{v}^{-1})^{-1}q_{v}^{-f-d_{v}/2}(\log q_{v})
\{\delta( b \in \gp_{v})(-f-\ord_{v}(b)) \\
& +\delta( b \in \go_{v}^{\times})(-f+\ord_{v}(b+1))
+ \delta(b \notin \go_{v})(-f)q_{v}^{\ord_{v}(b)}\}.
\end{align*}
This completes the proof.
\end{proof}


\subsection{} 
Let $v\in \Sigma_\infty$ and fix an identification $F_v\cong\R$. In this paragraph, we abbreviate $l_v$ to $l$ omitting the subscript $v$. Let $\varepsilon:\R^\times \rightarrow \{\pm 1\}$ be a character; thus $\varepsilon$ is the sign character or the trivial one. 
From the proof of \cite[Lemma 10.12]{SugiyamaTsuzuki}, we have
\begin{align*}
W_{v}^{\varepsilon}(b) & = \int_{\RR^{\times}}
\left(\frac{1+it}{\sqrt{t^{2}+1}}\right)^{l}\{ 1+i(bt^{-1}+t(b+1))\}^{-l/2}
\varepsilon(t) \log |t|_{v}d^{\times}t \\
& = \int_{\RR^{\times}}
(1-it)^{-l/2}(1+b+t^{-1}bi)^{-l/2}
\varepsilon(t) \log |t|_{v}d^{\times}t
\\
& = W_{+}(b) + \varepsilon(-1) \overline{W_{+}(b)},
\end{align*}
where we set 
$$W_{+}(b) = i^{l/2}(1+b)^{-l/2}\int_{0}^{\infty}
(t+i)^{-l/2}\left(t+\frac{bi}{b+1}\right)^{-l/2}t^{l/2-1}
\log t \ dt.
$$

Here is an explicit formula of $W_{+}(b)$.

\begin{lem} \label{EF-W+}
Suppose $l\ge 4$. Then, for $b \in \RR^{\times}-\{-1\}$, we have
\begin{align*}
W_{+}(b)=-\pi i\,J_{+}(l;b)-A(b)-i\,B(b), \\
\end{align*}
where
\begin{align*}
A(b)=&\sum_{k=0}^{l/2-1}\tbinom{l/2+k-1}{k}\tbinom{l/2-1}{k} \left\{
\tfrac{b^{k}}{2}
(\log|\tfrac{b}{b+1}|)^{2}-\tfrac{\theta(b)^2}{2}\,b^{k}-\tfrac{9\pi^2}{8}(-1)^{k+l/2}(b+1)^{k}\right\} \\
&+\sum_{k=0}^{l/2-1}\tbinom{l/2+k-1}{k}\sum_{j=1}^{l/2-k-1}\tbinom{l/2-1}{k+j} \tfrac{(-1)^{j}}{j} 
\left(\sum_{m=1}^{j-1}\tfrac{1}{m}\{ b^{k}+(-1)^{k+l/2}(b+1)^{k}\}-b^{k}\log|\tfrac{b}{b+1}|\right), \\
B(b)=&\sum_{k=0}^{l/2-1}\tbinom{l/2+k-1}{k}
\tbinom{l/2-1}{k}\,b^{k}\log|\tfrac{b}{b+1}|\,\theta(b)
\\
&-\sum_{k=0}^{l/2-1}\tbinom{l/2+k-1}{k}\sum_{j=1}^{l/2-k-1}\tbinom{l/2-1}{k+j}\tfrac{(-1)^{j}}{j}
\{\tfrac{3\pi}{2}(-1)^{k+l/2}(b+1)^{k}+b^{k}\theta(b)\},
\end{align*}
 $\theta(b)=\pi/2$ if $b(b+1)<0$, $\theta(b)=3\pi/2$ if $b(b+1)>0$
and $J_{+}(l; b)$ is the function defined in \cite[Lemma 10.13]{SugiyamaTsuzuki}.
\end{lem}
\begin{proof}
For $b \in \RR^{\times}-\{-1\}$,
put $g(z) = i^{l/2}(1+b)^{-l/2}(z+i)^{-l/2}\left(z+\frac{bi}{b+1}\right)^{-l/2}z^{l/2-1}(\log z)^{2},$
where
$\log z = \log|z| + i \arg(z)$ with $\arg(z) \in [0, 2\pi)$.
Then, $g(z)$ is holomorphic on $\CC -(\RR_{\ge 0}\cup \{-i, \frac{-bi}{b+1}\})$.
We note $\frac{-bi}{b+1} \in i \RR-\{0, -i\}$.
By Cauchy's integral theorem, we have
\begin{align*}
2\pi i\{\Res_{z=-i}+\Res_{z=\frac{-bi}{b+1}}\} g(z) =
\int_{\e}^{R}g(t)dt
+\oint_{|z|=R}g(z) dz -\int_{\e}^{R}g(te^{2\pi i}) -\oint_{|z|=\e}g(z)dz.
\end{align*}
with $R$ sufficiently large and $\e>0$ sufficiently small.
By $\lim_{R \rightarrow \infty}\oint_{|z|=R}g(z) dz = 0$,
$\lim_{\e \rightarrow +0}\oint_{|z|=\e}g(z) dz = 0$ and
$(\log t +2\pi i)^{2}=(\log t)^{2}+4\pi i\log t -4\pi^{2}$,
we also have
\begin{align*}
2\pi i\{\Res_{z=-i}+\Res_{z=\frac{-bi}{b+1}}\} g(z) =
-4\pi iW_{+}(b) +4\pi^{2}J_{+}(l; b).
\end{align*}
Hence, we obtain
\begin{align*}
W_{+}(b)=
-\frac{1}{2}\{\Res_{z=-i}+\Res_{z=\frac{-bi}{b+1}}\} g(z)
-\pi i\, J_{+}(l; b).
\end{align*}
Furthermore, a direct computation gives us
%
\begin{align*}
\Res_{z=-i}g(z)=&\sum_{k=0}^{l/2-1}\tbinom{l/2+k-1}{k} (-1)^{k+l/2}(b+1)^{k}\\
&\times 
\biggl\{\tbinom{l/2-1}{k} \tfrac{-9\pi^2}{4}+
2\sum_{j=1}^{l/2-k-1}\tbinom{l/2-1}{k+j} \tfrac{(-1)^{j}}{j} \biggl(\sum_{m=1}^{j-1}\tfrac{1}{m}-\tfrac{3\pi}{2}i\biggr)\biggr\}
\end{align*}
and
\begin{align*}
\Res_{z=\frac{ib}{b+1}}g(z)
=& \sum_{k=0}^{l/2-1}\tbinom{l/2+k-1}{k} b^{k} \times 
\biggl\{\tbinom{l/2-1}{k}\left(\log|\tfrac{b}{b+1}|+\theta(b)i\right)^2 \\
& + 2\sum_{j=1}^{l/2-k-1}\tbinom{l/2-1}{k+j} \tfrac{(-1)^{j}}{j} \biggl(\sum_{m=1}^{j-1}\tfrac{1}{m}-\log|\tfrac{b}{b+1}|-\theta(b)i\biggr)\biggr\}.
\end{align*}
This completes the proof.
\end{proof}

\begin{lem} \label{Est-arci-log}
Suppose $l>4$. For any $\epsilon>0$, we have
\begin{align*}
|b(b+1)|^{\e}\,|W_{v}^{\e}(b)|\ll_{\e, l} (1+|b|)^{-l/2+2\e}, \quad b\in \R-\{0,-1\}.
\end{align*}
\end{lem}
\begin{proof}
%
%
%
From Lemmas \ref{ArchIntEst} and \ref{EF-W+}, for any $\e>0$, $|b(b+1)|^{\e}W_{+}(b)$ is locally bounded around the points $b=0,-1$. For $b$ away from the set $\{0,-1\}$, we have
\begin{align*}
|W_+(b)|\leq |2b(b+1)|^{-l/4}\int_{0}^{\infty}t^{l/4}(t^{2}+1)^{-l/4}|\log t|\tfrac{\d t}{t}
\end{align*}
by $t^{2}(b+1)^2+b^{2}\geq 2|t||b(b+1)|$. Since $l>4$, the last integral is convergent; hence the above inequality gives us $|b(b+1)|^{\e}|W_+(b)|\ll_{\e, l} (1+|b|)^{-l/2+2\e}$ for large $|b|$.  
\end{proof}

%

\subsection{The proof of Proposition~\ref{WWhypglobal}}
We start from the formula \eqref{WWhypglobal-series} taking $\a$ to be $\a_\fa$ defined by \eqref{HeckefunctAA}. If we set
\begin{align}
\WW(T)=\sum_{b\in F-\{0,-1\}} \sum_{w\in T} \{\prod_{v\in \Sigma_F-\{w\}}J_v(b)\}\,W_w(b)
 \label{WWT}
\end{align}
for any subset $T\subset \Sigma_F$, then \eqref{WWhypglobal-series} can be written in the form
$$
\WW_{\rm{hyp}}^{\eta}(l,\fn|\a_\fa)
=\WW(\Sigma_\infty)+\WW(S(\fa))+\WW(S(\fn))+\WW(S(\gf_{\eta}))+\WW(\Sigma_\fin-
S(\fn\fa\gf_{\eta})). 
$$
We shall estimate each term in the right-hand side of this equality explicating  the dependence on $\fn$ and $\fa=\prod_{v\in S(\fa)}\fp_v^{n_v}$. Set $c=(\underline{l}/2-1)/d_{F}$. For convenience, we collect here all the estimates used below (other than these, we also need Lemma~\ref{Est-arci-log}). Let $w_1\in S(\fa)$, $w_2\in S(\fn)$, $w_3\in S(\ff_\eta)$, $w_4\in \Sigma_\fin-S(\fa\ff_\eta)$, and $w_5\in \Sigma_\infty$. Let $\e>0$ be a small number. Then,
\begin{align}
&|J_{w_1} (b)|\ll\delta(b\in \fa^{-1}\cO_{w_1})\,q_{w_1}\{1+\Lambda_{w_{1}}(b)\},
\qquad &
&|J_{w_2}(b)|\leq \delta(b\in \fn\cO_{w_2}),
 \label{Jhyouka2}
\\
&|J_{w_3}(b)| \ll \delta(b\in \ff_\eta^{-1}\cO_{w_3}),
\qquad &
&|J_{w_4}(b)|\leq \delta(b\in \cO_{w_4})\,\Lambda_{w_4}(b), \label{Jhyouka4} \\
&|b(b+1)|_{w_5}^{\e}|J_{w_5}(b)| \ll_{\e, l_{w_{5}}} (1+|b|_{w_5})^{-l_{w_5}/2+2\e}
\label{Jhyouka5}
\end{align}
(note the difference of $\ll$ and $\leq$ ), and  
\begin{align}
|W_{w_1}(b)|&\ll (\log q_{w_1})\,q_{w_1}\,\delta(b\in \fa^{-1}\cO_{w_1})
\{ 2n_{w_{1}} +\ord_{w_1}(b(b+1))+1\}^{2}, 
 \label{Whyouka1}
\\
|W_{w_2}(b)|&\ll (\log q_{w_2})\,\delta(b\in \fn\cO_{w_2})\,\{\ord_{w_2}(b)+\ord_{w_2}(\fn)+1\}^2, 
 \label{Whyouka2}
\\
|W_{w_3}(b)|&\ll (\log q_{w_3})\,\delta(b\in \ff_\eta^{-1}\cO_{w_3})\,\{ 2f(\eta_{w_{3}})
+\ord_{w_3}(b(b+1))+1 \} ,
 \label{Whyouka3}
\\
|W_{w_4}(b)|&\ll (\log q_{w_4})\,\delta(|b(b+1)|_{w_4}<1)\,\Lambda_{w_4}(b)^2
 \label{Whyouka4}
\end{align}
for $b\in F^\times$, where all the constants implied by the Vinogradov symbol are independent of the ideals $\fn$, $\fa$ and the places $w_i$ $(1\leq i\leq 5)$. Indeed, the second estimate in \eqref{Jhyouka2} and the both estimates of \eqref{Jhyouka4} follow from \cite[Lemmas 10.5, 10.4 and Corollary 10.11]{SugiyamaTsuzuki} immediately. The estimate \eqref{Jhyouka5} is from Lemma~\ref{ArchIntEst}. The first estimate in \eqref{Jhyouka2} is obtained in the proof of Proposition~\ref{HYPERBOLIC-EST}. The estimate \eqref{Whyouka1} follows from Lemma~\ref{esti of W_gq}, \eqref{Whyouka2} is from Lemma~\ref{esti of W_n}, \eqref{Whyouka3} is from Lemma~\ref{esti of W_eta}, and \eqref{Whyouka4} is from Lemma~\ref{esti of W_rest}. 

In the remaining part of this section, all the constants implied by Vinogradov symbol are independent of $\fn$ and $\fa$ (but may depend on $l$, $\eta$ and a given small number $\e>0$).
 
\begin{lem}\label{WWhyp-archi-part}
We have
\begin{align*}
|\WW(\Sigma_\infty)|\ll_{\e, l, \eta}
 \nr(\fa)^{c+2+\e}\,\nr(\fn)^{-c+\e}.
\end{align*}
\end{lem}
\begin{proof}
Similarly to the proof of Proposition \ref{HYPERBOLIC-EST},
by Lemma \ref{Est-arci-log}, we have
$|\WW(\Sigma_\infty)|\ll_{\e, l, \eta} C^{\#S(\ga)}\nr(\ga) \sum_{I \subset S(\ga)}\fI^{\eta}_{0}(l, \gn, \gf_{\eta}\prod_{v \in I}\gp_{v}^{n_{v}})$.
Then, the desired estimate is given by Proposition \ref{fI-EST}.
\end{proof}

\begin{lem}\label{WWhyp-q-part}
We have
$$|\WW(S(\fa))| \ll_{\e, l, \eta} \nr(\fa)^{c+2+\e}\,\nr(\fn)^{-c+\e}.$$
\end{lem}
\begin{proof}
By the estimates recalled above, the range of $b$ in the summation \eqref{WWT} with $T=S(\fa)$ can be restricted to $\gn\ga^{-1}\gf_{\eta}^{-1}-\{0,-1\}$. If $b \in \gn\fa^{-1}\ff_{\eta}^{-1}$, then $b(b+1)\fa^2\ff_{\eta}^2$ is an ideal of $\go$. From this, noting that $\eta$ is unramified over $S(\fa)$, we have the equality $\ord_{w}(b(b+1)\fa^2\ff_{\eta}^2) = 2n_w+\ord_{w}(b(b+1))$ for any $w\in S(\fa)$; by taking summation over $w\in S(\fa)$,  
\begin{align*}
\sum_{w\in S(\ga)}\{2n_w+\ord_{w}(b(b+1)) + 1\}\, \log q_w 
&\le \log \nr\bigl(b(b+1)\fa^2\ff_{\eta}^2\bigr) + \log \nr(\ga)
\ll_{\e, \eta} |\nr(b(b+1))|^{\e/2} \nr(\fa)^{\e}.
\end{align*}
Using this, from \eqref{Whyouka1}, \eqref{Jhyouka2}, and \eqref{Jhyouka4}, we obtain
\begin{align*}
\WW(S(\fa)) 
\ll & \sum_{b\in \gn\gf_{\eta}^{-1}\fa^{-1}-\{0,-1\}}
\sum_{w_{1}\in S(\ga)} \{\prod_{v\in \Sigma_{F}-\{w_1\}}|J_v(b)|\} \, (\log q_{w_1})\,q_{w_1}\,\{\ord_{w_1}(b(b+1))+2n_{w_1}+1\}^{2}
\\
\ll_{\e, \eta} & \,C^{\# S(\fa)}\,
\nr(\fa)^{1+2\e}
\sum_{b \in \gn\gf_{\eta}^{-1}\fa^{-1}} |\nr(b(b+1))|^{\e}\,\prod_{v\in \Sigma_{\infty}}|J_{v}(b)|\prod_{v\in \Sigma_{\fin}-S(\ga\gf_{\eta})}\Lambda_{v}(b)
\prod_{v \in S(\ga)}\{1+\Lambda_{v}(b)\}
\\ \leq & C^{\# S(\fa)}\nr(\fa)^{2\e+1} \, \sum_{I\subset S(\ga)}\fI^{\eta}_{\e}(l, \gn, \gf_{\eta}\prod_{v \in I}\gp_{v}^{n_{v}}),
\end{align*}
where $C$ is the implied constant in the first estimate of \eqref{Jhyouka2}
and (\ref{Whyouka1}). Noting $C^{\# S(\fa)}\ll_{\e} \nr(\fa)^{\e}$, we obtain the assertion by Proposition \ref{fI-EST}.  
\end{proof}

\begin{lem} \label{WWhyp-n-part}
We have
$$|\WW(S(\gn))| \ll_{\e, l, \eta} 
\nr(\fa)^{c+2+\e}\,\nr(\fn)^{-c+\e}.$$
\end{lem}
\begin{proof}
From the estimations recalled above, 
\begin{align*}
& |\WW(S(\gn))|\\
\ll_{\e, \eta} & C^{\# S(\fa)} \nr(\fa)\,\sum_{b \in \gn\gf_{\eta}^{-1}\ga^{-1}-\{0, -1\}}
\prod_{v \in \Sigma_{\infty}}|J_{v}(b)|\prod_{v \in \Sigma_{\fin}-S(\ga\gn\gf_{\eta})}\Lambda_{v}(b) \prod_{v \in S(\ga)}\{1+\Lambda_{v}(b)\}
\sum_{w_2 \in S(\gn)}|W_{w_2}(b)|,
\end{align*}
where $C$ is the implied constant in the first estimate of \eqref{Jhyouka2}. By \eqref{Whyouka2},
\begin{align*}
\sum_{w_2 \in S(\gn)} |W_{w}^{\eta_{w}}(b)|
\ll & \sum_{w_2\in S(\gn)} (\log q_{w_2})(\ord_{w_2}(\gn)+\ord_{w_2}(b)+1)^{2}
\\
\ll & \sum_{w_2 \in S(\gn)}\ord_{w_2}(\gn)^{2}(\log q_{w_2})
 +\sum_{w_2 \in S(\gn)}(\log q_{w_2}) \Lambda_{w_2}(b)^{2}
 \ll_{\e}  \nr(\gn)^{\e} \prod_{v\in S(\fn)}\Lambda_v(b)^2
\end{align*}
for $b\in \gn\gf_{\eta}^{-1}\ga^{-1}$. From this, we obtain
\begin{align*}
|\WW(S(\gn))|
\ll_{\e, \eta} & C^{\# S(\fa)}\nr(\ga)\nr(\gn)^{\e} \sum_{b \in \gn\gf_{\eta}^{-1}\ga^{-1}-\{0, -1\}}\prod_{v \in \Sigma_{\infty}}|J_{v}(b)|
\prod_{v \in \Sigma_{\fin}-S(\ga\gf_{\eta})}\Lambda_{v}(b)^{2}\prod_{v\in S(\ga)}\{1+\Lambda_{v}(b)\} \\
=& C^{\# S(\fa)}\nr(\ga)\nr(\gn)^{\e}\, \sum_{I \subset S(\ga)}\fI^{\eta}_{0}(l, \gn, \gf_{\eta}\prod_{v \in I}\gp_{v}^{n_{v}}).
\end{align*}
Then, the desired estimate is given by Proposition \ref{fI-EST}.
\end{proof}

\begin{lem}\label{WWhyp-eta-part}
We have
$$\WW(S(\gf_{\eta})) \ll_{\e, l, \eta} 
 \nr(\fa)^{c+2+\e}\,\nr(\fn)^{-c+\e}.$$
\end{lem}

\begin{proof}
By the same argument as in the proof of Lemma \ref{WWhyp-q-part}, we have
\begin{align*}
\sum_{w\in S(\gf_{\eta})}\{2f(\eta_{w})+\ord_{w}(b(b+1))+1\}\, \log q_w 
&\le \log \nr\bigl(b(b+1)\fa^2\ff_{\eta}^2\bigr) + \log\nr(\gf_{\eta})\\
&\ll_{\e, \eta} |\nr(b(b+1))|^{\e} \nr(\fa)^{2\e}
\end{align*}
for $b \in \gn\ga^{-1}\gf_{\eta}^{-1}$. From the estimations recalled as above, we obtain
\begin{align*}
\WW(S(\gf_{\eta})) \le & \sum_{b \in \gn\gf_{\eta}^{-1}\ga^{-1}}
\sum_{w_3 \in S(\gf_{\eta})}
\{\prod_{v \in \Sigma_{F}-\{ w_3 \}}|J_{v}(b)|\}\, (\log q_{w_3})\{2f(\eta_{w_{3}})+\ord_{w_{3}}(b(b+1))+1\} \\
\ll_{\e, l, \eta}\, & C^{\#S(\fa)}\nr(\fa)\,
\sum_{b\in \gn\gf_{\eta}^{-1}\ga^{-1}}|\nr(b(b+1))|^{\e} \nr(\ga)^{2\e}
\prod_{v \in \Sigma_{\infty}}|J_{v}(b)|
\prod_{v\in \Sigma_{\fin}-S(\fa\ff_\eta)}\Lambda_{v}(b) \prod_{v \in S(\ga)}\{1+\Lambda_{v}(b)\}\\
\ll_{\e, l, \eta} &  C^{\#S(\fa)}\nr(\fa)^{1+2\e}\sum_{I\subset S(\ga)}\fI^{\eta}_{\e}(l, \gn, \gf_{\eta}\prod_{v \in I}\gp_{v}^{n_{v}}).
\end{align*}
Then, the desired estimate is given by Proposition \ref{fI-EST}.
\end{proof}

\begin{lem}\label{WWhyp-rest-part}
We have
$$\WW(\Sigma_{\fin}-S(\ga\gn\gf_{\eta})) \ll_{\e, l, \eta}
\nr(\fa)^{c+2+\e}\,\nr(\fn)^{-c+\e}.$$
\end{lem}
\begin{proof} In the summation on the left hand side of \eqref{WWT} with $T=\Sigma_\fin-S(\ga\gn\gf_\eta)$, the range of $(b,w)$ is restricted to $b \in \gn\gf_{\eta}^{-1}\ga^{-1}$ and $w \in S(b(b+1)\go \cap \go)\cap T$, due to the estimations recalled above. Thus,
\begin{align*}
&\WW(\Sigma_{\fin}-S(\ga\gn\gf_{\eta})) 
 \\
= &
\sum_{b\in \gn\gf_{\eta}^{-1}\ga^{-1}-\{0, -1\}}\sum_{w \in S(b(b+1)\go\cap\go)-S(\ga\gn \gf_{\eta})}
\{\prod_{v \in \Sigma_{F}-\{w_4\}}|J_{v}(b)|\}\, |W_{w_4}(b)|\\
\ll_{\e, \eta}\, & C^{\# S(\fa)} \nr(\fa)\,
\sum_{b\in \gn\gf_{\eta}^{-1}\ga^{-1}-\{0, -1\}}
\prod_{v \in \Sigma_{\infty}}|J_{v}(b)|\prod_{v \in S(\ga)}\{1+\Lambda_{v}(b)\} \\
& \times \sum_{w_4 \in S(b(b+1)\go \cap \go)-S(\ga\gn \gf_{\eta})}
\{\prod_{\substack{v \in \Sigma_{\fin}-S(\ga\gf_{\eta})\\ v \neq w_4}}\Lambda_{v}(b)\}\,
(\log q_{w_4})\Lambda_{w_4}(b)^2 
\\
\ll_{\e, \eta}&\, C^{\# S(\fa)} \nr(\fa)
\sum_{b \in \gn\gf_{\eta}^{-1}\ga^{-1}-\{0, -1\}}\prod_{v \in \Sigma_{\infty}}|J_{v}(b)|
\prod_{v \in S(\ga)}\{1+\Lambda_{v}(b)\} \\
& \times \{\sum_{w_4\in S(b(b+1)\go \cap \go)-S(\ga\gn \gf_{\eta})}\log q_{w_4}\}
\prod_{v \in \Sigma_{\fin}-S(\ga\gf_{\eta})}\Lambda_{v}(b)^{2} \\
\ll_{\e, \eta}\, &  C^{\# S(\fa)} \nr(\fa)
\sum_{b \in \gn\gf_{\eta}^{-1}\ga^{-1}-\{0, -1\}}\prod_{v \in \Sigma_{\infty}}|J_{v}(b)|
\prod_{v \in S(\ga)}\{1+\Lambda_{v}(b)\} \\
& \times\prod_{v \in \Sigma_{\fin}-S(\ga\gf_{\eta})}\Lambda_{v}(b)^{2}
\times \nr(\ga)^{2\e}|\nr(b(b+1))|^{\e}
\\
= & C^{\# S(\fa)} \nr(\fa)^{1+2\e}\sum_{I \subset S(\ga)}\fI_{\e}^{\eta}(l,\fn,\ff_\eta\prod_{v \in I}\gp_{v}^{n_{v}}).
\end{align*}
Here we note 
$$\sum_{w_4\in S(b(b+1)\go \cap \go)-S(\ga\gn \gf_{\eta})}\log q_{w_4} \ll_{\e, \eta} \nr(\ga)^{2\e}|\nr(b(b+1))|^{\e}, \quad b \in \gn\gf_{\eta}^{-1}\ga^{-1}-\{0, -1\}.$$
Indeed, if $b\in \gn\gf_{\eta}^{-1}\ga^{-1}$, we have $b(b+1)\gf_{\eta}^{2}\ga^{2} \subset \go$ and $S(b(b+1)\go\cap\go)- S(\ga\gn \gf_{\eta})\subset S(b(b+1)\gf_{\eta}^{2}\ga^{2})$. Hence, 
\begin{align*}
\sum_{w_4\in S(b(b+1)\go \cap \go)-S(\ga\gn \gf_{\eta})}\log q_{w_4}
\le\, &
\sum_{w_4\in S(b(b+1)\gf_{\eta}^{2}\ga^{2})}\log q_{w_4}
\le \log \nr(b(b+1)\gf_{\eta}^{2}\ga^{2}) \\
\ll_{\e}\,& |\nr(b(b+1))\nr(\gf_{\eta})^{2}\nr(\ga)^2|^{\e}
, \hspace{5mm} b \in \gn\gf_{\eta}^{-1}\ga
-\{ 0, -1\}.
\end{align*}
Therefore, the assertion follows from Proposition \ref{fI-EST} and from $C^{\#S(\fa)}\ll_{\e}\nr(\fa)^{\e}$.
\end{proof}

As a consequence, Proposition \ref{WWhypglobal} follows from
Lemmas \ref{WWhyp-archi-part}, \ref{WWhyp-q-part},
\ref{WWhyp-n-part}, \ref{WWhyp-eta-part} and \ref{WWhyp-rest-part}.

\subsection{Unipotent terms}

We compute the local terms for $\tilde \WW_{\rm{u}}^{\eta}(l,\fn|\a)$ at a place $v\in S$. For $\a_v\in \ccA_v$, set
\begin{align}
U_v^{\eta_v}(\alpha_v)&=\frac{1}{2\pi i} \int_{\s-i \infty}^{\s+i \infty} \frac{1}{(1-\eta_v(\varpi_v)q_v^{-(s+1)/2})(1-q_v^{(s+1)/2})}\,\a(s)\,\d\mu_{v}(s), 
 \label{Unipterm-1}
\\
\tilde U_v^{\eta_v}(\alpha_v)&=\frac{1}{2\pi i} \int_{\s-i \infty}^{\s + i \infty} \frac{\eta_v(\varpi_v)\,\log q_v}{(1-\eta_v(\varpi_v)q_v^{-(s+1)/2})^2(1-q_v^{-(s+1)/2})\,q_v^{s+1}}\,\a(s)\,\d\mu_{v}(s)
 \label{Unipterm-2}
\end{align}
with $\d\mu_v(s)=2^{-1}\log q_v\,(q_v^{(1+s)/2}-q_v^{(1-s)/2}) \d s$
and $\s>0$. The integral $U_v^{\eta_v}$ is already computed in \cite[Proposition 11.1]{SugiyamaTsuzuki}. In the same way, we have the following lemma easily. 

\begin{lem} \label{DUNIP}
 For any $m\in \N_0$, we have 
\begin{align*}
\tilde U_v^{\eta_v}(\a_v^{(m)})&=-\delta(m>0)\,q_v^{-m/2}(\log q_v)\,
\begin{cases}
\left\{\frac{q_v-1}{2}m(-1)^m-\frac{3q_v+1}{4}(-1)^{m}+\frac{1-q_v}{4}\right\}, \quad &(\eta_v(\varpi_v)=-1), \\
\left\{\frac{(m-1)(m-2)}{2}q_{v}-\frac{m(m+1)}{2} \right\}, \quad &(\eta_v(\varpi_v)=+1).
\end{cases}
\end{align*}
\end{lem}
%
%

\section{Appendix: an estimation of a certain lattice sum}
Let $d\geq 1$ be an integer. We fix $l=(l_j)_{1\leq j\leq d}\in \R^{d}$ such that $l_d\geq \dots \geq l_1\geq 4$, and consider a positive function $f(x)$ on $\R^{d}$ defined by
\begin{align*}
f(x)=\prod_{j=1}^{d}(1+|x_j|)^{-l_j/2}, \quad x=(x_j)_{1\leq j\leq d}\in \R^{d}.\end{align*}
Given a $\Z$-lattice $\Lambda \subset \R^{d}$ (of full rank), we define   
\begin{align*}
\theta(\Lambda)&=\sum_{b\in \Lambda-\{0\}}f(b). 
\end{align*}
Viewing this as a function in $\Lambda$, we need to compare its asymptotic size with a certain power of $D(\Lambda)$, the Euclidean volume of a fundamental domain of $\R^{d}/\Lambda$. To state the main result of this section, we need another quantity $r(\Lambda)$ given by 
$$ r(\Lambda)=\frac{1}{2}\min_{b\in \Lambda-\{0\}}\|b\|. $$ 
 
\begin{thm} \label{theta-est}
Let $F$ be a totally real number field of degree $d$. Let $\Lambda_0$ and $\Lambda$ be fractional ideals such that $\Lambda\subset \Lambda_0$; we regard them as a $\Z$-lattice in $\R^{d}$ by the embedding $F\rightarrow \R^{{\rm Hom}(F,\R)}=\R^{d}$. Then, 
\begin{align*}
\theta(\Lambda)\ll \{1+r(\Lambda_0)\}^{dl_d/2}\,D(\Lambda_0)^{-1}\,D(\Lambda)^{(1-{l_1}/2)/d}
\end{align*}
with the implied constant independent of $\Lambda$ and $\Lambda_0$. 
\end{thm}
The proof is given at the last part of the next subsection after several lemmas. 

\subsection{The proof of Theorem~\ref{theta-est}}
Let $\d\mu(\omega)$ denote the Euclidean measure on the sphere
$\bS^{d-1}= \{ x=(x_{j})_{1 \le j \le d}\in \RR^{d} \ | \ \sum_{j=1}^{d}x_{j}^{2}=1 \}$. 

\begin{lem} \label{integralI}
For any $\lambda=(\lambda_j)\in \C^{d}$ such that $\Re(\lambda_j)<1$, we have
$$
I(\lambda)=\int_{\bS^{d-1}}\prod_{j=1}^{d}|\omega_j|^{-\lambda_j}\,\d \mu(\omega)=2\Gamma\left(\sum_{j=1}^{d}\tfrac{1-\lambda_j}{2}\right)^{-1}\prod_{j=1}^{d}\Gamma\left(\tfrac{1-\lambda_j}{2}\right).
$$
\end{lem}
\begin{proof}
The formula is obtained by computing the integral 
\begin{align}
\int_{\R^{d}}\exp(-\epsilon\|x\|^2)\prod_{j=1}^{d}|x_j|^{-\lambda_j}\,\d x
 \label{tildeI}
\end{align}
in two different ways, where $\epsilon>0$ and $\Re(\lambda_j)<1$ for the absolute convergence of the integral. By expressing \eqref{tildeI} as an iterating integral, we compute it as
\begin{align*}
\prod_{j=1}^{d}\int_{\R}e^{-\epsilon x_j^2}|x_j|^{-\lambda_j}\d x_j
=\prod_{j=1}^{d}\epsilon^{(\lambda_j-1)/2}\Gamma\left(\tfrac{1-\lambda_j}{2}\right)= \epsilon^{(\sum_{j=1}^{d}\lambda_j -d)/2}\prod_{j=1}^{d}\Gamma\left(\tfrac{1-\lambda_j}{2}\right)
\end{align*}
on one hand. On the other hand, by the polar decomposition, \eqref{tildeI} becomes
\begin{align*}
&\int_{0}^{\infty} \int_{\bS^{d-1}} e^{-\epsilon \rho^2} \prod_{j=1}^{d}|\rho\omega_j|^{-\lambda_j}\,\rho^{d-1}\,\d \rho\,\d\mu(\omega)
\\
&=\left(
\int_{\bS^{d-1}}\prod_{j=1}^{d}|\omega_j|^{-\lambda_j}\,\d \mu(\omega)\right)\,\left(\int_{0}^{\infty} e^{-\epsilon \rho^2}\rho^{-\sum_{j=1}^{d}\lambda_j +d-1}\,\d\rho\right)
\\
&=I(\lambda)\,2^{-1}{\epsilon^{(\sum_{j=1}^{d}\lambda_j -d)/2}}\,\Gamma\left(\sum_{j=1}^{d}\tfrac{1-\lambda_j}{2}\right).
\end{align*}
\end{proof}

\begin{lem} \label{Upperbound}
For $t=(t_j)_{1\leq j\leq d}\in [1,\infty)^{d}$, set 
$$
\varphi(t_1,\dots,t_{d})=\int_{\bS^{d-1}}f(t_1\omega_1,\dots,t_d\omega_d)\,\d\mu(\omega).
$$
For $t>1$, let $\underline t$ denote the diagonal element $(t_{j})$ defined by $t_j=t$ $(1\leq j\leq d)$. Then, 
$$
\varphi(\underline{t})=O(t^{1-d-l_1/2}), \quad t \in [1,\infty). 
$$
\end{lem}
\begin{proof}
For $\lambda=(\lambda_j)\in\C^{d}$ such that $0<\Re(\lambda_j)<1$, we compute the multiple Mellin-transform 
$$
\tilde \varphi(\lambda)=\int_{0}^{\infty}\cdots\int_{0}^{\infty} \varphi(t_1,\dots,t_d)\,\prod_{j=1}^{d}t_j^{\lambda}\,\frac{\d t_j}{t_j}. 
$$
By Lemma~\ref{integralI}, we compute this in the following manner. 
\begin{align*}
\tilde\varphi(\lambda)&=
\int_{\bS^{d-1}} \{\prod_{j=1}^{d} \int_{0}^{\infty} (1+t_j|\omega_j|)^{-l_j/2} t_j^{\lambda_j-1}\,\d t_j\}\,\d\mu(\omega)
\\
&=\int_{\bS^{d-1}} \{\prod_{j=1}^{d}|\omega_j|^{-\lambda_j}\, \int_{0}^{\infty} (1+t_j)^{-l_j/2}t_j^{\lambda_j-1}\,\d t_j\}\,\d\mu(\omega) \\
&=\left(\int_{\bS^{d-1}}\prod_{j=1}^{d}|\omega_j|^{-\lambda_j}\,\d\mu(\omega)\right) \,\left(\prod_{j=1}^{d}\int_0^{\infty} (1+t_j)^{-l_j/2}t_j^{\lambda_j-1}\,\d t_j\right)\\
&=I(\lambda)\,\{\prod_{j=1}^{d}\Gamma(l_j/2)^{-1}\,\Gamma(l_j/2-\lambda_j)\,\Gamma(\lambda_j)\}
\\
&=2\Gamma\left(\sum_{j=1}^{d}\tfrac{1-\lambda_j}{2}\right)^{-1}\,\{\prod_{j=1}^{d} \Gamma(l_j/2)^{-1}\,\Gamma((1-\lambda_j)/2)\,\Gamma(l_j/2-\lambda_j)\,\Gamma(\lambda_j)\}.
\end{align*}
By Stirling's formula, this is bounded by a constant multiple of $P(\Im \lambda)\,\exp(-\pi\sum_{j=1}^{d}|\Im(\lambda_j)|)$ with some polynomial $P(x_1,\dots,x_d)$ which can be taken uniformly with $\Re(\lambda)$ varied compactly. Thus, by  a successive application of the Mellin inversion formula, we obtain  
\begin{align*}
\varphi(\underline t)=\left(\frac{1}{2\pi i}\right)^{d}\int_{(\sigma_1)} \dots\int_{(\sigma_d)} 2\,\{\prod_{j=1}^{d}\frac{\Gamma\left(\frac{1-\lambda_j}{2}\right)\,\Gamma\left(\frac{l_j}{2}-\lambda_j\right)\,\Gamma(\lambda_j)}{\Gamma({l_j}/{2})} \}\,\frac{t^{-\sum_{j=1}^{d}\lambda_j}}{
\Gamma\left(\sum_{j=1}^{d}\tfrac{1-\lambda_j}{2}\right)}
 \,\prod_{j=1}^{d}\d\lambda_j,
\end{align*}
where the contour $(\sigma_j)=\{\Re(\lambda)=\sigma_j\}$ should be contained in the band $0<\Re(\lambda_j)<1$. We shift the contours $(\sigma_j)$ in some order far to the right. The residues arise when the moving contour $(\sigma_j)$ passes the points in $(1+2\Z_{\geq 0}) \cup (l_j/2+\Z_{\geq 0})$. Among those residues, the one with the smallest possible power of $t^{-1}$ comes from the pole at $\lambda_1=l_1/2$, $\lambda_j=1 \,(2\leq j\leq d)$ if $l_2>l_1$, which we assume for simplicity in the rest of the proof of this lemma. (When $l_2=l_1$, there are several terms giving the same power in $t^{-1}$.) The residue term is $O(t^{-(d-1+l_2/2)})$, by which the contribution from the remaining terms are majorized. This completes the proof. 
\end{proof}

\begin{lem}\label{L1}
\begin{itemize}
\item[(1)] 
\begin{align}
f(x+y) \geq f(x)\,f(y), \qquad x,\,y\in \R^{d}
 \label{L1-1}
\end{align}
\item[(2)]  
$$
\vol(\bS^{d-1})\,(1+\rho)^{-dl_d/2}\leq \int_{\bS^{d-1}} f(\rho\,\omega)\,\d\mu(\omega) \ll (1+\rho)^{1-d-l_1/2}, \quad \rho>0,
$$
with the implied constant depending on $l$ and $d$.  
\end{itemize}
\end{lem}
\begin{proof} 
(1) is immediate from the inequality $1+|x_j+y_{j}|\leq (1+|x_j|)(1+|y_j|)$. As for (2), we first note the inequality $0\leq |\omega_j|\leq 1$ for $\omega \in \bS^{d-1}$. Using this, we have $\prod_{j=1}^{d}(1+|\rho\omega_j|)\leq(1+\rho)^{d}$. By this,   
\begin{align*}
f(\rho\omega)&\geq \bigl\{\prod_{j=1}^{d}(1+|\rho\omega_j|)\bigr\}^{-{l_d}/2}\geq (1+\rho)^{-dl_d/2}. 
\end{align*}
Taking integral in $\omega$, we have the estimation from below as desired. The upper bound is provided by Lemma~\ref{Upperbound}. 
\end{proof}

We compare $\theta(\Lambda)$ with the integral of $f(x)$ on the ball $B_{\Lambda}=\{x\in \R^{d}|\,\|x\|\leq r(\Lambda)\,\}$. For convenience, we set 
$I(D)=\int_{D} f(x)\,\d x
$ for any Borel set $D$ in $\R^{d}$. 

\begin{lem} \label{L2}
Let $\Lambda_0$ and $\Lambda$ be $\Z$-lattices such that $\Lambda\subset \Lambda_0$. Then, we have the inequality
$$
\theta(\Lambda)\leq I(B_{\Lambda_0})^{-1}\,I(\R^{d}-B_{\Lambda})
$$
\end{lem}
\begin{proof} 
The inequality \eqref{L1-1} gives us
\begin{align*}
I(B_{\Lambda})\,\theta(\Lambda)\leq \sum_{b\in \Lambda-\{0\}}\int_{B_\Lambda} f(b+x)\,\d x. 
\end{align*}
Since $\Lambda\subset \Lambda_0$, we have $B_{\Lambda_0}\subset B_{\Lambda}$, from which $I(B_{\Lambda_0})\leq I(B_{\Lambda})$ is obtained by the non-negativity of $f(x)$. Since $(B_\Lambda+B_\Lambda)\cap \Lambda=\{0\}$, the translated sets $B_{\Lambda}+b\,(b\in \Lambda-\{0\})$ are mutually disjoint. From this remark,  
$$
\sum_{b\in \Lambda-\{0\}}\int_{B_\Lambda} f(b+x)\,\d x\leq \int_{\R^{d}-B_\Lambda} f(x)\,\d x=I(\R^{d}-B_\Lambda).
$$
Putting altogether, we are done.
\end{proof}

\begin{lem} \label{L3}
Let $\Lambda$ be a $\Z$-lattice. 
\begin{align*}
I(B_\Lambda)&\geq \vol(\bS^{d-1})\,(1+r(\Lambda))^{-d{l_d}/2}\,r(\Lambda)^{d}/d, \\
I(\R^{d}-B_\Lambda)&\ll r(\Lambda)^{1-{l_1}/2}
\end{align*}
with the implied constant independent of $\Lambda$. 
\end{lem}
\begin{proof}
By Lemma~\ref{L1} (2), 
\begin{align*}
I(B_\Lambda)&=\int_{0}^{r(\Lambda)} \int_{\bS^{d-1}} f(\rho\omega)\,\d\omega\,\rho^{d-1}\,\d \rho
\\
&\geq \vol(\bS^{d-1})\, \int_{0}^{r(\Lambda)} (1+\rho)^{-d{l_d}/2}\,\rho^{d-1}\,\d \rho
\\
&\geq \vol(\bS^{d-1})(1+r(\Lambda))^{-d{l_d}/2} \int_{0}^{r(\Lambda)}\rho^{d-1}\,\d \rho 
&=\vol(\bS^{d-1})(1+r(\Lambda))^{-d{l_d}/2}\,r(\Lambda)^{d}/d.
\end{align*}
In a similar way, 
\begin{align*}
I(\R^{d}-B_{\Lambda})&=\int_{r(\Lambda)}^{\infty}\int_{\bS^{d-1}} f(\rho\omega)\,\d\omega\,\rho^{d-1}\,\d \rho 
\\
&\ll \int_{r(\Lambda)}^{\infty}(1+\rho)^{1-d-{l_1}/2}\,\rho^{d-1}\,\d \rho
\leq
\int_{r(\Lambda)}^{\infty}\rho^{-{l_1}/2}\,\d \rho
=(l_{1}/2-1)^{-1}r(\Lambda)^{1-{l_1}/2}. 
\end{align*}
\end{proof}

\begin{lem} \label{L4} 
Let $F$ be a totally real number field of degree $d$. There exist constants $C_d$ and $C'_d$ such that $C_d\,r(\Lambda)^{d}\leq D(\Lambda)\leq C'_{d}\,r(\Lambda)^{d}$ for any fractional ideal $\Lambda$.
\end{lem}
\begin{proof}
 The first inequality follows from Minkowski's convex body theorem. The second inequality is proved as follows. For any $b\in \Lambda-\{0\}$, there exists an ideal $\fa\subset \cO$ such that $(b)=\fa\Lambda$; hence $|\nr(b)|=\nr(\Lambda)\nr(\fa)\geq \nr(\Lambda)$. Thus, by the arithmetic-geometric mean inequality,
\begin{align*}
D(\Lambda)^{1/d}=\nr(\Lambda)^{1/d} \leq \{\prod_{j=1}^{d}|b_j|^2\}^{1/(2d)}
\leq \bigl\{{\sum_{j=1}^{d}|b_j|^2}/{d} \bigr\}^{1/2}=d^{-1/2}\|b\| 
\end{align*}
Hence, $D(\Lambda)^{1/d} \leq 2d^{-1/2}r(\Lambda)$. This shows $D(\Lambda)\leq C_d'r(\Lambda)^{d}$ with $C_d'=(2d^{-1/2})^{d}$. 
\end{proof}

Theorem~\ref{theta-est} follows from Lemmas~\ref{L2}, \ref{L3} and \ref{L4} immediately.

\section*{Acknowledgements}
The first author was supported by Grant-in-Aid for JSPS Fellows (25$\cdot$668).
The second author was supported by Grant-in-Aid for Scientific Research (C) 22540033.


\medskip
\noindent
{Shingo SUGIYAMA \\ Department of Mathematics, Graduate School of Science, Osaka University, Toyonaka, Osaka 560-0043, Japan} \\
{\it E-mail} : {\tt s-sugiyama@cr.math.sci.osaka-u.ac.jp}  

\medskip
\noindent
{Masao TSUZUKI \\ Department of Science and Technology, Sophia University, Kioi-cho 7-1 Chiyoda-ku Tokyo, 102-8554, Japan} \\
{\it E-mail} : {\tt m-tsuduk@sophia.ac.jp}

\end{document}